[1994/12/01]
\documentclass[10pt, reqno]{amsart}
\usepackage[margin=1.9cm]{geometry}
\usepackage[english, activeacute]{babel}
\usepackage{amsmath,amsthm,amsxtra}
\usepackage{epsfig}
\usepackage{amssymb}
\usepackage{latexsym}
\usepackage{amsfonts}
\usepackage{hyperref}
\usepackage{multicol}
\usepackage{color}

\title{Global solutions and stability properties of the 5th order Gardner equation}
\author{Miguel A. Alejo}
\address{Departamento de Matem\'atica, Universidade Federal de Santa Catarina, Brasil}
\thanks{M. A. was partially funded by Product. CNPq grant no. 305205/2016-1  and VI PPIT-US program ref. I3C}
\author{Chulkwang Kwak}
\address{Facultad de Matem\'aticas, Pontificia Universidad Cat\'olica de Chile, Campus San Joaqu\'in. Avda. Vicu\~na Mackenna 4860, Santiago, Chile and Institute of Pure and Applied Mathematics, Chonbuk National University}
\thanks{C. K. is supported by FONDECYT Postdoctorado 2017 Proyect No. 3170067.}
\email{miguel.alejo@ufsc.br}
\email{chkwak@mat.uc.cl}
\date{\today}
\thanks{}

\chardef\bslash=`\\ 

\newtheorem{thm}{Theorem}[section]
\newtheorem{cor}[thm]{Corollary}
\newtheorem{lem}[thm]{Lemma}
\newtheorem{prop}[thm]{Proposition}

\newtheorem{defn}[thm]{Definition}

\theoremstyle{remark}
\newtheorem{rem}{Remark}[section]

\numberwithin{equation}{section}

\newcommand{\R}{\mathbb{R}}

\newcommand{\Z}{\mathbb{Z}}
\newcommand{\T}{\mathbb{T}}

\newcommand{\la}{\lambda}

\newcommand{\al}{\alpha}
\newcommand{\bt}{\beta}
\newcommand{\ga}{\gamma}

\newcommand{\si}{\sigma}

\newcommand{\spawn}{\operatorname{span}}

\newcommand{\supp}{\operatorname{supp}}

\newcommand{\wt}{\widetilde}
\newcommand{\ft}{\mathcal{F}}

\newcommand{\be}{\begin{equation}}
\newcommand{\ee}{\end{equation}}
\newcommand{\bp}{\begin{proof}}
\newcommand{\ep}{\end{proof}}
\newcommand{\bel}{\begin{equation}\label}
\newcommand{\eeq}{\end{equation}}
\newcommand{\bea}{\begin{eqnarray}}
\newcommand{\eea}{\end{eqnarray}}
\newcommand{\bee}{\begin{eqnarray*}}
\newcommand{\eee}{\end{eqnarray*}}
\newcommand{\ben}{\begin{enumerate}}
\newcommand{\een}{\end{enumerate}}
\newcommand{\nonu}{\nonumber}

\newcommand{\ms}{\medskip}

\def\normo#1{\left\|#1\right\|}
\def\normb#1{\Big\|#1\Big\|}
\def\norm#1{\|#1\|}
\def\set#1{\{#1\}}
\def\bra#1{\langle#1\rangle}

\newcommand{\px}{\partial_x}
\newcommand{\pt}{\partial_t}

 \providecommand{\norm}[1]{\lVert#1 \rVert}
 
\newcommand{\eval}[2][\right]{\relax
  \ifx#1\right\relax \left.\fi#2#1\rvert}

\let\norm=\enVert

\begin{document}
\begin{abstract}
In this work, we deal with the initial value problem of the 5th-order Gardner equation in $\R$, presenting the local  well-posedness result in $H^2(\R)$. As a consequence of the local result, in addition to $H^2$-energy conservation law, we are able to prove the global well-posedness result in $H^2(\R)$. Finally, we  present a stability result for 5th order Gardner  breather solution in the Sobolev space $H^2(\R)$.
\end{abstract}
\maketitle \markboth{ IVP for the 5th order Gardner equation}{Miguel A. Alejo and Chulkwang Kwak}
\renewcommand{\sectionmark}[1]{}

\section{Introduction}\label{0}
In this work, we are concerned with the \emph{focusing} 5th order Gardner equation

\be\begin{aligned}\label{5G}
& u_t + u_{5x}+10 \mu ^2 u_{3x}  +20 \mu  u u_{3x}+10 u^2 u_{3x}+120 \mu^3 u u_x+180 \mu ^2 u^2 u_x\\
&+120 \mu  u^3 u_x+10 u_x^3+40 \mu  u_xu_{xx}+40uu_xu_{xx} +30 u^4u_x=0,~~\mu\in\R^+.\\
\end{aligned}\ee
\noindent
This higher order Gardner equation can be obtained from the corresponding 5th order \emph{focusing} modified Korteweg-de Vries equation (shortly, 5th mKdV)
\be\begin{aligned}\label{5mkdv}
& v_{t} +(v_{4x} +10vv_x^2 +10v^2v_{xx} + 6v^5)_x=0,
\end{aligned}\ee
when one considers mKdV solutions of the form $v(t,x) = \mu + u(t,x),$ with $\mu\in\R^+$ and a suitable spatial translation\footnote{Such a spatial translation is performed in order to provide a simpler expression of the \emph{N-soliton solution} in Section \ref{sect5}. On the other hand, it is known that not only the first order linear term but also the third order term of the linear part in \eqref{5G} are negligible in the study of the well-posedness theory compared to the fifth order term.}.

\medskip

The 5th order Gardner equation  \eqref{5G}, as well as the 5th mKdV equation, is a well-known \emph{completely integrable} model \cite{Ga,AC,La}, with infinitely many conservation laws and well-known (long-time) asymptotic behavior of its solutions obtained with the help of the inverse scattering transform \cite{GrSl}. As a physical model, the 5th Gardner \eqref{5G} and the 5th mKdV \eqref{5mkdv} equations describe large-amplitude internal solitary waves, showing a dynamics which can look rather different from the KdV form. On the other hand, solutions of \eqref{5G} are invariant under space and time translations. Indeed, for any $t_0, x_0\in \R$, $u(t-t_0, x-x_0)$ is also a solution of both equations. Beside that, the scaling invariance is not respected by \eqref{5G}.

\medskip

As seen in \eqref{5G}, the 5th order Gardner equation \eqref{5G} contains mixed nonlinearities of 5th KdV equation
\be\begin{aligned}\label{5kdv}
& v_{t} +(v_{4x} + 5v_x^2 +10vv_{xx} + 10v^3)_x=0,
\end{aligned}\ee
\noindent
and 5th mKdV \eqref{5mkdv}, and hence the well-posedness theory of \eqref{5G} is highly relevant to the well-posedness of both equations. Ponce \cite{Ponce1993}, first, showed the local well-posedness of 5th KdV in $H^s(\mathbb R)$, $s \ge 4$ via the energy method in addition to the dispersive smoothing effect and a parabolic approximation method. Later, this local result has been improved by Kwon \cite{Kwon2008}, precisely, the local well-posedness in $H^s(\mathbb R)$, $s > \frac52$. Thereafter, Guo, Kwon and the second author \cite{GKK2013} and Kenig and Pilod \cite{KP2015}, independently, proved the local well-posedness in $H^s(\mathbb R)$, $s \ge 2$. Both works were based on the short time Fourier restriction norm method \cite{IKT2008}, while an additional weight and the (frequency localized) modified energy were used to prove the crucial energy estimates, respectively. Thanks to the $H^2$-level energy conservation law, the local result extended to the global one. 

\medskip

On the other hand, the 5th mKdV \eqref{5mkdv} has been studied by Linares \cite{Lin}. Linares proved the local well-posedness in $H^2(\mathbb R)$ via the contraction mapping principle in addition to the dispersive smoothing effect \cite{KPV1991.0, KPV1991}. Later, Kwon \cite{Kwon} improved the local result in $H^s(\mathbb R)$, $s > \frac34$, by using the standard Fourier restriction norm method \cite{Bourgain1993} in addition to Tao's $[k,Z]$-multiplier norm method \cite{Tao2001}. 

\medskip

We also refer to \cite{KPV1991-0, KPV1994, Pilod2008, Gru, Kato2012} for the local well-posedness of for higher order KdV and mKdV equations.

\medskip

It is known that the Initial value problem (IVP) of 5th KdV \eqref{5kdv} is a quasilinear problem in the sense that the solution map is (not uniformly) continuous, while the Cauchy problem of 5th mKdV is a semilinear problem in the sense that the flow map is Lipschitz continuous (via the Picard iteration method, and hence, is analytic). Thus, one may expect that the IVP of the 5th Gardner equation \eqref{5G} is also a quasilinear problem due to the strong (high-low) quadratic nonlinearity. Moreover, one expects to obtain the local well-posedness in $H^2(\mathbb R)$ (and hence the global well-posedness in $H^2(\mathbb R)$) from \cite{GKK2013, KP2015}. However, to prove the local well-posedness of the 5th Gardner equation is definitely non-trivial, thus one of aims in this work is to indeed prove the local well-posedness.

\medskip

As related problems, we also refer to \cite{Bourgain1995, Kwak2016, Kwak2018, KP2016, Tsugawa2017} for the well-posedness of 5th KdV, 5th mKdV and higher order equations in KdV hierarchy under the periodic boundary condition. 


\medskip

Concerning explicit solutions of  higher order mKdV and Gardner models,  Matsuno \cite{Mat1} proved the existence and built explicitly the N-soliton solution of the focusing mKdV hierarchy of equations by using inverse scattering technics and the bilinear Hirota decomposition. Recently, Gomes et al \cite{Gomes} dealt with the defocusing mKdV with NVBC and the associated defocusing Gardner hierarchy, showing multisolitonic structures. Unfortunately, many of the solutions they obtained are singular solutions (up to the kink which is in $L^\infty$). 


\medskip

The 5th order Gardner equation \eqref{5G}, as a completely integrable system, has an infinite set of conserved quantities.
Indeed  some of the (first) standard conservation laws of the \eqref{5G} are the \emph{mass}
\begin{eqnarray}\label{M1cor} M[u](t)  :=  \frac 12 \int_\R u^2(t,x)dx =
M[u](0), \end{eqnarray} the \emph{energy}

\be\label{E1}
E_\mu[u](t)  :=  \int_\R \left(\frac 12 u_x^2 -2\mu u^3  - \frac 12 u^4\right)(t,x)dx = E[u](0),
\ee

\ms \noindent and the \emph{higher order energy}, defined
respectively in $H^2(\R)$ 
\be\label{E5} E_{5\mu}[u](t)  :=   
\int_\R \left(\frac 12 u_{xx}^2 -10\mu uu_x^2 + 10 \mu^2u^4 - 5u^2u_x^2 + 6\mu u^5 + u^6\right)(t,x)dx= E_5[u](0).
\ee

\subsection{Main results}
We are interested in the regularity properties of the 5th order Gardner equation \eqref{5G} and long time behavior of $H^2$ global solutions to \eqref{5G}. 

\subsubsection{Well-posedness theory}
In comparison with the 5th mKdV \eqref{5mkdv}, the nonlinearity of \eqref{5G} consists of more terms which break the balance with the  5th order linear dispersive part of \eqref{5G}. Precisely, additional quadratic terms with three derivatives, pose technical problems, for instance, the failure of bilinear $X^{s,b}$ estimates, see Remark \ref{rem:failure} below (also see Remark 2.3 in \cite{GKK2013}). However, an analogous argument used in \cite{GKK2013, KP2015} enables us to attack the initial value problem of \eqref{5G} in $H^2$.

\medskip

The notion of the well-posedness, which is taken into account in this paper, is as follows:
\begin{defn}[Well-posedness]
We say that the 5th Gardner equation \eqref{5G} is local-in-time (or locally) well-posed in $H^s(\R)$, if for any $R > 0$ and any $u_0 \in \set{f \in H^s(R) : \norm{f}_{H^s} \le R}$, there exist a local time $T=T(R) > 0$ and a unique solution $u$ to \eqref{5G} in $C([0,T];H^s(\R)) \cap X_T$, for some auxiliary space $X_T$. Moreover, the solution map $u_0 \mapsto u(t)$ is continuous from $\set{f \in H^s(\R) : \norm{f}_{H^s} \le R}$ to $C([0,T];H^s(\R))$. The local result is extended to the global one, if $T >0$ is independent of $R$. 
\end{defn}

\medskip

We are, first, going to show that  the 5th order Gardner equation \eqref{5G} is locally well-posed in $H^2$ via the classical energy method in addition to the short time Fourier restriction norm method. We state the local well-posedness result as follows:
\begin{thm}\label{LWP1}
The 5th order Gardner equation \eqref{5G} is locally well-posed in $H^s(\R)$, $s \ge 2$. 
\end{thm}

For the proof of Theorem \ref{LWP1}, we use the short-time Fourier restriction norm method in a frequency dependent time interval. This is introduced by Ionescu, Kenig and Tataru \cite{IKT2008} in the context of KP-I equation in the Besov-type space setting, see also \cite{KT2007, CCT2008} for similar ideas in the different settings. The short-time Fourier restriction norm method has been further developed in, for instance, \cite{Guo2011, GPWW2011, Guo2012, GKK2013, KP2015, Kwak2016, Kwak2018, GO2018}.

\medskip

The main difficulty arising in \eqref{5G} is the strong \emph{high-low} bilinear interaction component of the following type\footnote{Here $P$ is a appropriate truncation operator in the Fourier space, thus $P_{high}u$ means the high frequency ($\|\xi| \gg 1$) localized portion of $u$, while the frequency support of $P_{\leq 0}u$ is in $[-1,1]$.}
\begin{equation}\label{hl}
(P_{\leq 0}u)\cdot (P_{high}u_{xxx})
\end{equation}
newly generated from the map $v(t,x) = \mu + u(t,x)$. The standard bilinear $X^{s,b}$-estimates\footnote{The $X^{s,b}$ spaces are equipped with the norm
\[\norm{f}_{X^{s,b}} = \norm{\bra{\xi}^s\bra{\tau-\xi^5}^b \wt{f}}_{L_{\tau,\xi}^2},\]
where $\wt{f}$ is the space time Fourier coefficient (also denoted by $\ft(f)$) and $\bra{\cdot} = (1+|\cdot|^2)^{\frac12}$. For more details, see Section \ref{sec:WP}.} ($ \|uu_{xxx} \|_{X^{s,b-1}} \lesssim \|u\|_{X^{s,b}}^2$) fails in usual $X^{s,b}$ spaces for any $s \in \R$ (see Remark \ref{rem:failure} below), where the $X^{s,b}$ norm is defined in \eqref{Xsb}, since the dispersive smoothing effect in a coherent case occurring in \eqref{hl} is not enough to control the three derivative in the high frequency mode. The following remark provides a counter-example to show the failure of the standard bilinear estimate:
\begin{rem}[Remark 2.3 in \cite{GKK2013}]\label{rem:failure}
Similarly as the 5th KdV case, also as mentioned before, the standard $X^{s,b}$ bilinear estimate fails to hold:
\begin{equation}\label{eq:counterexample}
\norm{u\px^3v}_{X^{s,b-1}} \leq C\norm{u}_{X^{s,b}}\norm{v}_{X^{s,b}},
\end{equation}
due to the following \emph{high-low} interactions causing the coherence, for instance, 
\[u(t,x) = \ft^{-1}[\mathbf{1}_\Omega(\tau, \xi)](t,x) \quad \mbox{and} \quad v(t,x) = \ft^{-1}[\mathbf{1}_\Sigma(\tau, \xi)],\]
where space-time frequency sets $\Omega$ and $\Sigma$ are given by\footnote{It suffices to regard only $\px^5$ as a linear part of \eqref{5G}, since $\px^3$ is negligible in a sense of the dispersion effect.}
\[\Omega=\set{(\tau, \xi) \in \R^2 : |\tau - \xi^5| \le 1, N \le |\xi| \le N + 1} \quad \mbox{and} \quad \Sigma=\set{(\tau, \xi) \in \R^2 : |\tau - \xi^5| \le 1, |\xi| \le 1},\]
for fixed large frequency $N \gg 1$. Indeed, a direct calculation gives  $\text{LHS of } \eqref{eq:counterexample} = NN^s$, while $\text{RHS of } \eqref{eq:counterexample} = N^s$.
\end{rem}
However, using $X^{s,b}$ structure in a short time interval $(\approx \text{(frequency)}^{-2})$, one reduces the contribution of high frequency with low modulation, so that one handles \emph{high-low} interaction component \eqref{hl} (see Remark \ref{rem:recover} below and Proposition \ref{prop:bilinear}).

\begin{rem}[Remark 2.3 in \cite{GKK2013}]\label{rem:recover}
The short time $X^{s,b}$ spaces ($F^s$ and $N^s$ to be introduced in Section \ref{sec:sol space}) in the interval of the length $(\approx \text{(frequency)}^{-2})$ resolves the low-high interaction counter-example presented in Remark \ref{rem:failure}. The corresponding sets in this setting are given by
\[\wt{\Omega}=\set{(\tau, \xi) \in \R^2 : |\tau - \xi^5| \le N^2, N \le |\xi| \le N + N^{-2}} \quad  \mbox{and} \quad \wt{\Sigma}=\set{(\tau, \xi) \in \R^2 : |\tau - \xi^5| \le 1, |\xi| \le 1},\]
and define $u$ and $v$ similarly as in Remark \ref{rem:failure}, but with respect to $\wt{\Omega}$ and $\wt{\Sigma}$, respectively. Then, one immediately obtains for any $s \in \R$ that
\[\norm{u\px^3v}_{N^s} \sim N^sN^3N^{-1}N^{-2}N \sim N^sN \quad \mbox{and} \quad \norm{u}_{F^s}\norm{v}_{F^s} \sim N^sN.\]
\end{rem}
A price to pay for the profit of the short-time argument is an energy-type estimate. However, the strong high-low interactions, where the low frequency component has the largest modulation, cause a trouble in the energy estimates when following Ionescu-Kenig-Tataru's method. A way to treat this interaction is to use a weight, which was suggested in \cite{IK2007} to handle the same interaction for the Benjamin-Ono equation (see also \cite{GPWW2011}). Note that the modified energy, initially introduced in \cite{Kwon2008} and further developed in \cite{KP2015, Kwak2016, Kwak2018, MPV2018-1, MPV2018-2}, plays a similar role as an additional weight. See \cite{GKK2013} and \cite{KP2015} for a comparison.   

\medskip

Note moreover that a scaling equivalence enables us to focus on small solutions to \eqref{5G_Gen} instead of \eqref{5G} (see Section \ref{sec:WP}). To close the energy method argument for \eqref{5G_Gen}, we gather linear, nonlinear and energy estimates, 
\begin{equation}\label{eq:brief proof}
\left \{
\begin{array}{l}
\norm{u}_{F^{s}(T')}\lesssim \norm{u}_{E^s(T')} + \norm{\mathcal N_2(u) + \mathcal N_3(u) + \mathcal{SN}(u)}_{N^s(T')},\\
\norm{\mathcal N_2(u) + \mathcal N_3(u) + \mathcal{SN}(u)}_{N^s(T')} \lesssim \sum_{j=2}^{5} \norm{u}_{F^s(T)}^j,\\
\norm{u}_{E^s(T')}^2 \lesssim \norm{u_0}_{H^s}^2 + \sum_{j=3}^{6} \norm{u}_{F^s(T)}^j.
\end{array}
\right.
\end{equation}
The continuity argument ensures \emph{a priori} bound of solutions to \eqref{5G_Gen}. Moreover, a similar estimate as in \eqref{eq:brief proof} for the difference of two solutions completes the limiting argument (compactness argument). We note that the energy estimate for the difference of two solutions does not hold true in $F^s$ spaces due to the lack of the symmetry, but hold in the intersection of the weaker ($F^0$) and the stronger ($F^{2s}$) spaces, thus the Bona-Smith argument is essential to close the compactness method.

\medskip


The global well-posedness follows immediately from the above local result and the conservation of the second order energy \eqref{E5}.
\begin{thm}\label{GWP1}
The 5th order Gardner equation \eqref{5G} is globally well-posed in the energy space $H^2(\R)$\footnote{The persistence of regularities ensures the global well-posedness in $H^s(\R)$, $s \ge 2$.}.
\end{thm}
\begin{rem}
It is well-known that local results can be extended to the global one in the energy space without the smallness assumption for defocusing equations (for simple models), while the smallness condition is necessary for the proof of the global well-posedness in the energy space for focusing equations (the large data global well-posedness for focusing equations has a different story). However, \eqref{5G} admits the scaling equivalence, which is slightly different from the standard scaling symmetry (or invariance), but still plays an almost same role in the local (or perturbation) theory. Thus, one has Theorem \ref{GWP1} from Theorem \ref{LWP1} in addition to the (rescaled) conservation law \eqref{E5}. See Section \ref{sec:WP}, in particular Section \ref{sec:LWP}, for more details.
\end{rem}

On the other hand, an observation explained in Remark \ref{rem:failure} above naturally poses an interesting question: Does the flow map from data to solutions fail to be (locally) uniformly continuous for all regularities? As an immediate answer to the question, we state the following (weak) ill-posedness result, which extends Cardoso and the first author's recent result \cite{AC2018} to all regularities:
\begin{thm}\label{Illposed}
The 5th order Gardner equation \eqref{5G} is weakly ill-posed in $H^s(\R)$, for $s > 0$ in the following sense: there exist $c, C > 0$, $0 < T \le 1$, and two sequences $u_n$ and $v_n$ of solutions to \eqref{5G} such that  
\[\sup_n\norm{u_n(t)}_{H^s} + \sup_n\norm{v_n(t)}_{H^s} \le C, \quad t \in [0,T]\]
and initially
\[\lim_{n \to \infty}\norm{u_n(0) - v_n(0)}_{H^s} = 0,\]
but for every $t \in [0,T]$
\[\liminf_{n \to \infty}\norm{u_n(t) - v_n(t)}_{H^s} \ge c|\sin t| \sim c|t|.\]
\end{thm}

Theorem \ref{Illposed} can be expected from the observation in the linear \emph{local smoothing effect} \cite{KPV1991.0, KPV1991}
\[\norm{\px^2e^{-t\px^5}u_0}_{L_x^{\infty}L_t^2} \lesssim \norm{u_0}_{L^2}\] 
compared to the three derivatives in the quadratic nonlinearity. In other words, the local smoothing effect, which recovers only two derivatives, is not enough to handle the nonlinear term $u\partial_x^3u$, as already seen in Remark \ref{rem:failure}. Such a strong high-low interaction phenomenon can be seen in other dispersive equations, for instance, the Benjamin-Ono equation (BO) and the Kadomtsev-Petviashvili I equation (KP-I). Early, constructing examples reflecting \eqref{hl}, the flow map has been shown to be not $C^2$ continuous \cite{MST2001, MST2002}, and uniformly continuous \cite{KT2005, KT2008}.

\medskip

To prove Theorem \ref{Illposed}, we take an argument introduced in \cite{KT2005} (but essentially follows from \cite{Kwon2008}) in order to construct the approximate solutions, which indeed reveals the ill-posedness phenomenon. Using the local well-posedness theory, one shows the approximate solutions are indeed "good" approximate solutions in $H^s$ sense, $s \ge 2$. Moreover, since the equation \eqref{5G} is completely integrable (thus it admits infinitely many conservation laws), we are able to show the same conclusion in the regularity range not only $s \ge 2$, but also $0 < s < 2$ by using $L^2$ and $H^2$ conservation laws.

\medskip

The strategy employed in \cite{AC2018} was to use Gardner breather solutions as a way to measure the regularity of the associated Cauchy problem in $H^s$. This allowed to find the sharp Sobolev index under which the local well-posedness of the problem is lost, meaning that the dependence of 5th order Gardner solutions upon initial data fails to be continuous. We refer to, for instant,  \cite{KPV2, CCT2003, Kwon} for analogous arguments. 

\medskip

Finally, together with the result in \cite{AC2018}, we get the following

\begin{cor}\label{cor:illposed}
The 5th order Gardner equation \eqref{5G} is (weakly) ill-posed in $H^s(\R)$, for $s \in \R$, in the sense of the statement given in Theorem \ref{Illposed}.
\end{cor}

As already seen above, the 5th Gardner equation \eqref{5G} contains the mixed nonlinearities of 5th order KdV and mKdV equations \eqref{5kdv}-\eqref{5mkdv}, so that one can see both ill-posedness nature of semilinear and quasilinear equations. In the proof of Theorem \ref{Illposed}, approximate solutions are constructed in the following manner: the separation of the phase shift ($\mp t$) and the dispersion effect ($\Phi_N(t)$) in \eqref{hiwave} inspired by the observation on the Burgers equation. However, in low regularity Sobolev space ($L^2$ or below), it is not clear to see such a phenomenon, see \cite{KT2005, KT2008, Kwon2008}. Nonetheless, the cubic nonlinearity (5th mKdV nonlinearity) reveals another ill-posedness phenomenon, breaking the uniform continuity of the flow map by the self-interaction of a single high frequency wave in low regularity spaces \cite{AC2018}. This nature can be seen in some semilinear equations, for instance \cite{KPV2, BGT2002, BGT2003, CCT2003, Kwon, Ale1}. The mixed nonlinearities in \eqref{5G}, thus, ensure to claim the lack of the uniform continuity of the flow map of the 5th Gardner equation \eqref{5G} in all regularity Sobolev spaces.

\subsubsection{Global stability theory}
Moreover, once we have characterized the IVP for \eqref{5G} and with respect to stability properties of specific solutions of the \eqref{5G}, we present  the following stability result 
for the 5th order  breather solutions \eqref{5BreG}.

%
 
 \begin{thm}\label{StaBre1}
Let $\alpha, \beta \in  \R\backslash\{0\}$ be given. Breather solutions \eqref{5BreG}  of the 5th order Gardner equation  \eqref{5G} are orbitally stable for  $H^2$ perturbations, whenever the parameter
$\mu\in(0,\frac{\sqrt{\al^2+\bt^2}}{2})$.
\end{thm}

For more detailed statements and background about this stability property of breather solutions, see Section \ref{sect5}.

\section{Well-posedness results}\label{sec:WP}
\subsection{Setting}
It is well-known that the integrability of equations (fixed coefficients of the nonlinearities) is no longer important for mathematical analysis in the local well-posedness theory.  

\begin{rem}
As mentioned in Section \ref{0}, \eqref{5G} does not allow the scaling invariance. However, defining $u_{\lambda} := \lambda u (\lambda^5 t , \lambda x)$, $\lambda > 0$, ensures an equivalence between \eqref{5G} and
\begin{equation}\label{5G_Gen}
w_t+w_{5x}+10\mu^2\lambda^2w_{3x} +\mathcal N_2(w) + \mathcal N_3(w) + \mathcal{SN}(w) = 0,
\end{equation}
where $\mathcal N_2(w)$ is the nonlinearity from the fifth order KdV given by
\begin{equation}\label{N2}
\mathcal N_2(w) = 20\mu\lambda w_x w_{xx} + 40\mu\lambda ww_{xxx} + 180\mu^2\lambda^2 w^2w_x,
\end{equation}
$\mathcal N_3(w)$ is the nonlinearity from the fifth order mKdV given by
\[\mathcal N_3(w) = 10w^2w_{3x} + 10w_x^3 + 40 ww_xu_{xx} + 30 w^4w_x\]
and $\mathcal{SN}(w)$ is the rest terms generated from the transformation $u \mapsto \mu + u$, which is weaker compared to $\mathcal N_2(w)$ and $\mathcal N_3(w)$ in some sense, given by
\begin{equation}\label{SN}
\mathcal{SN}(w) = 120\mu^3\lambda^3ww_x + 120\mu\lambda w^3w_x.
\end{equation}
That is, $u_{\lambda}$, $\lambda > 0$ is a solution to \eqref{5G_Gen}, if and only if $u$ is a solution to \eqref{5G}. See Section \ref{sec:LWP} for the details.
\end{rem}

We use the notation $\widetilde{f}$ or $\mathcal F (f)$ for the space-time Fourier transform of $f$ defined by
\[\widetilde{f}(\tau, \xi)=\int _{\R ^2} e^{-ix\xi}e^{-it\tau}f(x,t) \;dxdt \]
for any $f \in \mathcal S'(\mathbb R \times \mathbb R)$.  Similarly, we use $\mathcal F_x$ (or  $\widehat{\;}$ ) and $\mathcal F_t$ to denote the Fourier transform with respect to space and time variable respectively. 

Let $\Z_+$ denote the set of nonnegative integers. For $k \in \Z_+$, let define dyadic intervals $I_k$, $k \in \Z_+$ as
\[I_0=\{\xi:|\xi|\le2\} \quad I_k=\{\xi:|\xi|\in [2^{k-1},2^{k+1}]  \} \quad k \ge 1.\]

Let $\eta_0: \R \to [0,1]$ denote a smooth bump function supported in $ [-2,2]$ and equal to $1$ in $[-1,1]$ with the following property of regularities:
\begin{equation}\label{eq:regularity}
\partial_n^{j} \eta_0(\xi) = O(\eta_0(\xi)/\langle \xi \rangle^j), \hspace{1em} j=0,1,2,
\end{equation}
as $\xi$ approaches end points of the support of $\eta$. For $k \in \Z_+ $, let 
\begin{equation}\label{chi}
\chi_0(\xi) = \eta_0(\xi) \quad \mbox{and} \quad \chi_k(\xi) = \eta_0(\xi/2^k) - \eta_0(\xi/2^{k-1}), \quad k \ge 1,
\end{equation} 
and
\[\chi_{[k_1,k_2]}=\sum_{k=k_1}^{k_2} \chi_k \quad \mbox{ for any} \ k_1 \le k_2 \in \Z_+ .\]
For the time-frequency decomposition, we use the cut-off function $\eta_j$, but the same as $\eta_j = \chi_j$, $j \in \Z_+$. For $k\in \Z$ let $P_k$ denote the (smooth) truncation operators on $L^2(\R)$ defined by $\widehat{P_ku}(\xi)= \chi_k(\xi)\widehat{u}(\xi)$. We also define the operators $P_k$ on $L^2(\R \times \R)$ by formulas $\mathcal F(P_ku)(\xi,\tau)=\chi_k(\xi)\mathcal F(u)(\tau,\xi)$.
For $l\in \Z$ let
\[P_{\le l}=\sum_{k \le l}P_k, \quad P_{\ge l}=\sum_{k \ge l}P_k. \]
For $\xi \in \R$, $w(\xi)=-\xi^5$ is the dispersion relation associated to the equation \eqref{5G_Gen}\footnote{Originally, we have $w(\xi) = -\xi^5 +10\mu^2\xi^3$ corresponding to the linear part of \eqref{5G}. However, for fixed $\mu$ and for large frequency $|\xi| \gg 1$, the third order term are negligible compared to the fifth order term.}. For $k \in\Z$ and $j \in \Z_+$ let
\[D_{k,j}=\{(\tau, \xi) \in \R \times \R : \xi \in  [2^{k-1},2^{k+1}] , \tau - w(\xi) \in I_j\},\quad D_{k, \le j}=\cup_{\ell \le j}D_{k, \ell}.\]

For $f \in L^2(\R)$, let $W(t)f \in C(\R:L^2)$ be the linear solution given by
\begin{equation}\label{Lin. Sol.}
\mathcal F_x[W(t)f](\xi,t)=e^{itw(\xi)}\widehat{f}(\xi).
\end{equation}

\subsection{Function spaces}\label{sec:sol space}

We introduce the $X^{s,b}$ spaces associated to \eqref{5G_Gen}, which is the completion of $\mathcal S'(\mathbb{R}^2)$ under the norm
\begin{equation}\label{Xsb}
\norm{f}_{X^{s,b}}=\norm{\langle \tau - w(\xi)\rangle^b\langle \xi \rangle^s \widetilde f}_{L^2(\R^2)},
\end{equation}
where $\langle \cdot \rangle = (1+|\cdot|^2)^{\frac12}$. This Fourier restriction norm method was first implemented by in its current form by Bourgain \cite{Bourgain1993} and further developed by Kenig, Ponce and Vega \cite{KPV1996} and Tao \cite{Tao2001}. The Fourier restriction norm method turns out to be very useful in the study of low regularity theory for the dispersive equations. We denote the localized space by $X_{T}$ defined by standard localization to the interval $[-T,T]$.

\medskip

As already mentioned in Section \ref{0}, the 5th Gardner equation \eqref{5G} is a quasilinear equation where the flow map is not uniformly continuous. This fact can be seen from \cite{Kwon2008}, which proves that the 5th order KdV equation \eqref{5kdv} is weakly ill-posed, since this phenomenon occurs precisely in a strong interaction between low and high frequencies localized data of the form
\[(u_{\le 0})\cdot(\partial_x^3u_{\gg 1})\]
which is also included in the nonlinearity of 5th order Gardner \eqref{5G}. For this reason, we must focus specifically on quadratic nonlinearity to prove the local well-posedness of the 5th Gardner equation \eqref{5G}. In what follows, we briefly introduce the functions spaces used in \cite{GKK2013}\footnote{The basic method is similar to that used in \cite{KP2015}, but it is chosen to avoid complicated calculations in the energy estimate.}.

\medskip

One of the purposes in this paper, as mentioned in Section \ref{0}, is to obtain $H^2$ global solutions to \eqref{5G}. Moreover, this regularity threshold is determined by the estimates of 
quadratic terms with three derivatives, which is already known from \cite{GKK2013, KP2015}. In what follows, we only focus on obtaining the estimates of cubic terms with three derivatives in $H^2$, since the cubic terms are another nontrivial and strong nonlinearities in \eqref{5G}. On the other hand, we expect that all estimates of this cubic terms can be obtained below $H^2$ compared to the quadratic nonlinearities, since the degree $3$ of nonlinearities allows more smoothing effects in \emph{high-low} interactions. However, we do not  here explore such estimates below $H^2$ for our purpose.  

\medskip

We fix $k \in \Z_+$, and define the weighted Besov-type ($X^{0,\frac12,1}$) space $X_k$ for frequency localized functions in $\widetilde I_k$,
\[X_k=\left\{ f \in L^2(\R^2) : \supp f \subset \R \times I_k, \quad \norm{f}_{X_k}<\infty\ \right\},\]
equipped with the norm
\[\norm{f}_{X_k}:=\sum_{j=0}^\infty
2^{j/2}\beta_{k,j}\norm{\eta_j(\tau-w(\xi)) f(\xi,\tau)}_{L^2_{\xi,\tau}},\]
where
\begin{eqnarray}\label{beta weight}
\beta_{k,j} =\left \{ \begin{array}{lr}
2^{j/2}, &k=0,\\
1+2^{(j-5k)/8}, &k \ge 1.
\end{array}
\right.
\end{eqnarray}

\begin{rem}
The use of the weight $\beta_{k,j}$ is essential to control the localized energy for the quadratic terms in $\mathcal N_2(u)$, in particular, the \emph{high-low} interaction components, where the low frequency component has the largest modulation. See Lemma \ref{lem:energy-quadratic}. Moreover, it enables us to avoid the logarithmic divergence in $H^2$ appearing in the energy estimates for the cubic nonlinearities in $\mathcal N_3(u)$, see Remark \ref{rem:weight effect} and Propositions \ref{prop:energy} and \ref{prop:energy diff}. 
\end{rem}

\begin{rem}
An opposite effect of the use of the weight is to worsen the \emph{high $\times$ high $\to$ low} interactions in the nonlinear estimates for the quadratic terms in $\mathcal N_2(u)$
\[P_{\leq 0}(P_{high}u\cdot P_{high}v_{xxx}).\]
However, thanks to the representation of the quadratic nonlinearities as the compact, conservative form, i.e., $c_1\px u\px^2u + c_2u\px^3u=c_1'\partial_x(\px u \px u)+c_2'\px(u\px^2u)$, one derivative is removed, and hence we are able to balance both purposes.
\end{rem}

\begin{rem}
Finally,  the choice of a parameter $\frac 18$ in the weight for high frequency can be replaced by any parameter in $[1/8, 3/16]$. However, another choice of parameter is not able to improve the result, since the essential effect of the weight occurs in the \emph{high-low} interactions, where the low frequency part has the largest modulation, as mentioned before.
\end{rem}

At each frequency $2^k$, we define functions spaces based on $X_k$, uniformly on the $2^{-2k}$ time scale. 
\[
F_k=\left\{ f \in L^2(\R^2) : \supp f \subset \R \times I_k, \quad \norm{f}_{F_k}<\infty \right\},
\]
equipped with the norm
\[\norm{f}_{F_k}=\sup\limits_{t_k\in \R}\norm{\mathcal F[f\cdot \eta_0(2^{2k}(t-t_k))]}_{X_k}\]
and
\[
N_k=\left\{ f \in L^2(\R^2) : \supp f \subset \R \times I_k, \quad \norm{f}_{N_k}<\infty \right\},
\]
equipped with the norm
\[\norm{f}_{N_k}=\sup\limits_{t_k\in \R}\norm{(\tau-\omega(\xi)+i2^{2k})^{-1}\mathcal F[f\cdot \eta_0(2^{2k}(t-t_k))]}_{X_k}.\]

The standard way to construct localized spaces gives, for $T \in (0,1]$, that
\[\begin{aligned}
F_k(T)=&\{f \in C([-T,T]:L^2): \norm{f}_{F_k(T)}=\inf_{\widetilde{f}=f \mbox{ in } [-T,T] \times \R}\|\widetilde f\|_{F_k}\},\\
N_k(T)=&\{f\in C([-T,T]:L^2): \norm{f}_{N_k(T)}=\inf_{\widetilde{f}=f \mbox{ in }  [-T,T] \times \R}\|\widetilde f\|_{N_k}\}.
\end{aligned}\]

We collect all pieces of spaces introduced above at dyadic frequency $2^k$ in the Littlewood-Paley way. For $s\geq 0$ and $T\in (0,1]$, we define function spaces for solutions and nonlinear terms:
\begin{eqnarray*}
&&F^{s}(T)=\left\{ u:\
\norm{u}_{F^{s}(T)}^2=\sum_{k=0}^{\infty}2^{2sk}\norm{P_k(u)}_{F_k(T)}^2 <\infty \right\},
\\
&&N^{s}(T)=\left\{ u:\
\norm{u}_{N^{s}(T)}^2=\sum_{k=0}^{\infty}2^{2sk}\norm{P_k(u)}_{N_k(T)}^2 <\infty \right\}.
\end{eqnarray*}
In order to take the short time structure for IVP of \eqref{5G}, it is required to define the energy space as follows: for $s\geq 0$ and $u\in
C([-T,T]:H^\infty)$
\[\norm{u}_{E^{s}(T)}^2=\norm{P_{\leq 0}(u(0))}_{L^2}^2+\sum_{k\geq 1}\sup_{t_k\in [-T,T]}2^{2sk}\norm{P_k(u(t_k))}_{L^2}^2.\]

\begin{rem}
The short time Fourier restriction norm method used in this work was introduced by Ionescu, Kenig and Tataru \cite{IKT2008}, where the local well-posedness of KP-I equation in the energy space was proved, and further developed in \cite{Guo2011, GPWW2011, Guo2012, GKK2013, KP2015, Kwak2016, Kwak2018, GO2018} and references therein. We also refer to \cite{KT2007, CCT2008} for different formulas of short time analysis. 
\end{rem}

For the extension argument of functions in the spaces introduced above, we follow from \cite{IKT2008} to define the set $S_k$ of $k$-\emph{acceptable time multiplication factors} for any $k\in \Z_+$:
\[S_k=\{m_k:\R\rightarrow \R: \norm{m_k}_{S_k}=\sum_{j=0}^{10} 2^{-2jk}\norm{\partial^jm_k}_{L^\infty}< \infty\}.\] 
Direct estimates using the definitions and \eqref{eq:pXk3} show that for any $s\geq 0$ and $T\in (0,1]$
\[\begin{cases}
\normb{\sum\limits_{k\in \Z_+} m_k(t)\cdot P_k(u)}_{F^{s}(T)}\lesssim (\sup_{k\in \Z_+}\norm{m_k}_{S_k})\cdot \norm{u}_{F^{s}(T)},\\
\normb{\sum\limits_{k\in \Z_+} m_k(t)\cdot P_k(u)}_{N^{s}(T)}\lesssim (\sup_{k\in \Z_+}\norm{m_k}_{S_k})\cdot \norm{u}_{N^{s}(T)},\\
\normb{\sum\limits_{k\in \Z_+} m_k(t)\cdot P_k(u)}_{E^{s}(T)}\lesssim (\sup_{k\in \Z_+}\norm{m_k}_{S_k})\cdot \norm{u}_{E^{s}(T)}.
\end{cases}\]

We end this subsection with the following important lemma.
\begin{lem}[Properties of $X_k$]\label{lem:propXk}
Let $k, l\in \Z_+$ with $l \le 5k$ and $f_k\in X_k$. Then
\begin{equation}\label{eq:pre2}
\begin{split}
&\sum_{j=l+1}^\infty 2^{j/2}\beta_{k,j}\normo{\eta_j(\tau-\omega(\xi)) \int_{\R}|f_k(\tau',\xi)| 2^{-l}(1+2^{-l}|\tau-\tau'|)^{-4}d\tau'}_{L^2}\\
&+2^{l/2}\normo{\eta_{\leq l}(\tau-\omega(\xi)) \int_{\R} |f_k(\tau',\xi)| 2^{-l}(1+2^{-l}|\tau-\tau'|)^{-4}d\tau'}_{L^2}\lesssim \norm{f_k}_{X_k}.
\end{split}
\end{equation}
In particular, if $t_0\in \R$ and $\gamma\in \mathcal S(\R)$, then
\begin{eqnarray}\label{eq:pXk3}
\norm{\mathcal F[\gamma(2^l(t-t_0))\cdot \mathcal F^{-1}(f_k)]}_{X_k}\lesssim \norm{f_k}_{X_k}.
\end{eqnarray}
\end{lem}
\begin{proof}
See \cite{GKK2013} for the proof. 
\end{proof}

\subsection{$L^2$-block estimates}
For $x,y \in \R_+$, $x \lesssim y$ means that there exists $C>0$ such that $x \le Cy$, and $x \sim y$ means $x \lesssim y$ and $y\lesssim x$. We also use $\lesssim_s$ and $\sim_s$ as similarly, where the implicit constants depend on $s$. Let $a_1,a_2,a_3,a \in \R$. The quantities $a_{max} \ge a_{sub} \ge a_{thd} \ge a_{min}$ can be conveniently defined to be the maximum, sub-maximum, third-maximum and minimum values of $a_1,a_2,a_3,a$ respectively.

\medskip
For $\xi_1,\xi_2 \in \R$, let denote the (quadratic) resonance function by
\[H = H(\xi_1,\xi_2) = w(\xi_1) + w(\xi_2) - w(\xi_1+\xi_2) =\frac52\xi_1\xi_2(\xi_1+\xi_2)(\xi_1^2+\xi_2^2+(\xi_1+\xi_2)^2).\]
Similarly, for $\xi_1,\xi_2,\xi_3 \in \R$, let
\begin{equation}\label{eq:tri-resonant function}
\begin{aligned}
G(\xi_1,\xi_2,\xi_3) &=  w(\xi_1) + w(\xi_2) + w(\xi_3) - w(\xi_1+\xi_2+\xi_3) \\
&= \frac52(\xi_1+\xi_2)(\xi_2 + \xi_3)(\xi_3 + \xi_1)(\xi_1^2 + \xi_2^2 +\xi_3^2+(\xi_1 + \xi_2 + \xi_3)^2)
\end{aligned}
\end{equation}
be the (cubic) resonance function. Such resonance functions play an important role in the nonlinear $X^{s,b}$-type estimates. 

\medskip

Let $f,g,h \in L^2(\R^2)$ be compactly supported functions. We define a quantity by
\[J_2(f,g,h) = \int_{\R^4} f(\zeta_1,\xi_1)g(\zeta_2,\xi_2)h(\zeta_1+\zeta_2+H(\xi_1,\xi_2), \xi_1+\xi_2) \; d\xi_1d\xi_2d\zeta_1\zeta_2. \]
The change of variables in the integration yields
\[J_2(f,g,h)=J_2(g^*,h,f)=J_2(h,f^*,g),\]
where $f^*(\zeta,\xi)=f(-\zeta,-\xi)$. From the identities
\[\xi_1 + \xi_2 = \xi_3 \quad \mbox{and} \quad (\tau_1 - w(\xi_1))+(\tau_2 - w(\xi_2)) = (\tau_3 - w(\xi_3)) + H(\xi_1,\xi_2)\]
on the support of $J_2(f^{\sharp},g^{\sharp},h^{\sharp})$, where $f^{\sharp}(\tau,\xi) = f(\tau-w(\xi),\xi)$ with the property $\norm{f}_{L^2} = \norm{f^{\sharp}}_{L^2}$, we see that $J(f^{\sharp},g^{\sharp},h^{\sharp})$ vanishes unless
\[2^{k_{max}} \sim 2^{k_{med}} \gtrsim 1 \quad \mbox{and} \quad 2^{j_{max}} \sim \max(2^{j_{med}}, |H|).\]

For compactly supported functions $f_i \in L^2(\R \times \R)$, $i=1,2,3,4$, we define 
\[J_3(f_1,f_2,f_3,f_4) = \int_{*}f_1(\zeta_1,\xi_1)f_2(\zeta_2,\xi_2)f_3(\zeta_3,\xi_3)f_4(\zeta_1 + \zeta_2 + \zeta_3  + G(\xi_1,\xi_2,\xi_3), \xi_1 + \xi_2 + \xi_3) ,\]
where the $\int_* = \int_{\R^6} \cdot \; d\xi_1d\xi_2 d\xi_3 d\zeta_1 d \zeta_2 d\zeta_3$. From the identities
\[\xi_1+\xi_2+\xi_3 = \xi_4 \quad \mbox{and} \quad (\tau_1 - w(\xi_1)) + (\tau_2 - w(\xi_2)) + (\tau_3 - w(\xi_3)) = (\tau_4 - w(\xi_4)) + G(\xi_1,\xi_2,\xi_3)\]
on the support of $J_3(f_1^{\sharp},f_2^{\sharp},f_3^{\sharp},f_4^{\sharp})$, we see that $J_3(f_1^{\sharp},f_2^{\sharp},f_3^{\sharp},f_4^{\sharp})$ vanishes unless
\begin{equation}\label{eq:tri-support property}
\begin{array}{c}
2^{k_{max}} \sim 2^{k_{sub}}\\
2^{j_{max}} \sim \max(2^{j_{sub}}, |G|),
\end{array}
\end{equation}
where $|\xi_i| \sim 2^{k_i}$ and $|\zeta_i| \sim 2^{j_i}$, $i=1,2,3,4$. A direct calculation shows 
\[|J_3(f_1,f_2,f_3,f_4)|=|J_3(f_2,f_1,f_3,f_4)|=|J_3(f_3,f_2,f_1,f_4)|=|J_3(f_1^{\ast},f_2^{\ast},f_4,f_3)|.\]

\medskip

We give $L^2$-block estimates for the quadratic and cubic nonlinearities. The bi- and tri-linear $L^2$-block estimates for the 5th order equations have already been introduced and used in several works, we refer to \cite{CLMW2009, CG2011,GKK2013, KP2015, Kwak2016, Kwak2018, CK2018-1, CK2018-2}.
\begin{lem}\label{lem:block estimate}
	Let $k_i \in \Z,j_i\in \Z_+,i=1,2,3$. Let $f_{k_i,j_i} \in L^2(\R\times\R) $ be nonnegative functions supported in $[2^{k_i-1},2^{k_i+1}]\times I_{j_i}$.
	
	(a) For any $k_1,k_2,k_3 \in \Z$ with $|k_{max}-k_{min}| \le 5$ and $j_1,j_2,j_3 \in \Z_+$, then we have
	\[J_2(f_{k_1,j_1},f_{k_2,j_2},f_{k_3,j_3}) \lesssim 2^{j_{min}/2}2^{j_{med}/4}2^{- \frac34 k_{max}}\prod_{i=1}^3 \|f_{k_i,j_i}\|_{L^2}.\]
	
	(b) If $2^{k_{min}} \ll 2^{k_{med}} \sim 2^{k_{max}}$, then for all $i=1,2,3$ we have
	\[J_2(f_{k_1,j_1},f_{k_2,j_2},f_{k_3,j_3}) \lesssim 2^{(j_1+j_2+j_3)/2}2^{-3k_{max}/2}2^{-(k_i+j_i)/2}\prod_{i=1}^3 \|f_{k_i,j_i}\|_{L^2}.\]
	
	(c) For any $k_1,k_2,k_3 \in \Z$ and $j_1,j_2,j_3 \in \Z_+$, then we have
	\[J_2(f_{k_1,j_1},f_{k_2,j_2},f_{k_3,j_3}) \lesssim 2^{j_{min}/2}2^{k_{min}/2}\prod_{i=1}^3 \|f_{k_i,j_i}\|_{L^2}.\]
\end{lem}
\begin{proof}
We refer to \cite{CLMW2009, CG2011, KP2015} for the proof.
\end{proof}

\begin{cor}\label{cor:block estimate}
Assume $k_i \in \Z$ and $j_i\in \Z_+$, $i=1,2,3$ and $f_{k_i,j_i} \in L^2(\R\times\R) $ be functions supported in $D_{k_i,j_i}$, $i=1,2$.

(a) For any $k_1,k_2,k_3 \in \Z$ with $|k_{max}-k_{min}| \le 5$ and $j_1,j_2,j_3 \in \Z_+$, then we have
\[\| \mathbf{1}_{D_{k_3,j_3}}(\tau,\xi)(f_{k_1,j_1}\ast f_{k_2,j_2})\|_{L^2} \lesssim 2^{j_{min}/2}2^{j_{med}/4}2^{- \frac34 k_{max}}\prod_{i=1}^2 \|f_{k_i,j_i}\|_{L^2}.\]

(b) If $2^{k_{min}} \ll 2^{k_{med}} \sim 2^{k_{max}}$, then for all $i=1,2,3$ we have
\[\| \mathbf{1}_{D_{k_3,j_3}}(\tau,\xi)(f_{k_1,j_1}\ast f_{k_2,j_2})\|_{L^2} \lesssim 2^{(j_1+j_2+j_3)/2}2^{-3k_{max}/2}2^{-(k_i+j_i)/2}\prod_{i=1}^2 \|f_{k_i,j_i}\|_{L^2}.\]

(c) For any $k_1,k_2,k_3 \in \Z$ and $j_1,j_2,j_3 \in \Z_+$, then we have
\[\| \mathbf{1}_{D_{k_3,j_3}}(\tau,\xi)(f_{k_1,j_1}\ast f_{k_2,j_2})\|_{L^2} \lesssim 2^{j_{min}/2}2^{k_{min}/2}\prod_{i=1}^2 \|f_{k_i,j_i}\|_{L^2}.\]
\end{cor}

\begin{lem}\label{lem:tri-L2}
Let $k_i \in \Z$ and $j_i\in \Z_+$, $i=1,2,3,4$. Let $f_{k_i,j_i} \in L^2(\R \times \R) $ be nonnegative functions supported in $I_{j_i} \times [2^{k_i-1},2^{k_i+1}]$.

(a) For any $k_i \in \Z$ and $j_i \in \Z_+$, $i=1,2,3,4$, we
have
\[J_3(f_{k_1,j_1},f_{k_2,j_2},f_{k_3,j_3},f_{k_4,j_4}) \lesssim 2^{(j_{min}+j_{thd})/2}2^{(k_{min}+k_{thd})/2}\prod_{i=1}^4 \|f_{k_i,j_i}\|_{L^2}.\]

(b) Let $k_{thd} \le k_{max}-10$.

(b-1) If $(k_i,j_i) = (k_{thd},j_{max})$ for $i=1,2,3,4$, we have
 \[J_3(f_{k_1,j_1},f_{k_2,j_2},f_{k_3,j_3},f_{k_4,j_4}) \lesssim 2^{(j_1+j_2+j_3+j_4)/2}2^{-2k_{max}}2^{k_{thd}/2}2^{-j_{max}/2}\prod_{i=1}^4 \|f_{k_i,j_i}\|_{L^2}.\]

(b-2) If $(k_i,j_i) \neq (k_{thd},j_{max})$ for $i=1,2,3,4$, we have
\[ J_3(f_{k_1,j_1},f_{k_2,j_2},f_{k_3,j_3},f_{k_4,j_4}) \lesssim 2^{(j_1+j_2+j_3+j_4)/2}2^{-2k_{max}}2^{k_{min}/2}2^{-j_{max}/2}\prod_{i=1}^4 \|f_{k_i,j_i}\|_{L^2}.\]
\end{lem} 

\begin{proof}
We refer to \cite{KP2015, Kwak2018} for the proof. In \cite{Kwak2018}, the second author established (cubic) $L^2$-block estimates for functions $f_{k_i,j_i} \in L^2(\R \times \Z)$, but the proof, here, is almost identical and easier, see \cite{KP2015}.
\end{proof}

\begin{cor}\label{cor:tri-L2}
Let $k_i \in \Z$ and $j_i\in \Z_+$, $i=1,2,3,4$. Let $f_{k_i,j_i} \in L^2(\R \times \R) $ be nonnegative functions supported in $D_{k_i,j_i}$.

(a) For any $k_i \in \Z$ and $j_i \in \Z_+$, $i=1,2,3,4$, we
have
\[\| \mathbf{1}_{D_{k_4,j_4}}(\tau,\xi)(f_{k_1,j_1}\ast f_{k_2,j_2}\ast f_{k_3,j_3})\|_{L^2} \lesssim 2^{(j_{min}+j_{thd})/2}2^{(k_{min}+k_{thd})/2}\prod_{i=1}^4 \|f_{k_i,j_i}\|_{L^2}.\]

(b) Let $k_{thd} \le k_{max}-10$.

(b-1) If $(k_i,j_i) = (k_{thd},j_{max})$ for $i=1,2,3,4$, we have
 \[\| \mathbf{1}_{D_{k_4,j_4}}(\tau,\xi)(f_{k_1,j_1}\ast f_{k_2,j_2}\ast f_{k_3,j_3})\|_{L^2} \lesssim 2^{(j_1+j_2+j_3+j_4)/2}2^{-2k_{max}}2^{k_{thd}/2}2^{-j_{max}/2}\prod_{i=1}^4 \|f_{k_i,j_i}\|_{L^2}.\]

(b-2) If $(k_i,j_i) \neq (k_{thd},j_{max})$ for $i=1,2,3,4$, we have
\[\| \mathbf{1}_{D_{k_4,j_4}}(\tau,\xi)(f_{k_1,j_1}\ast f_{k_2,j_2}\ast f_{k_3,j_3})\|_{L^2} \lesssim 2^{(j_1+j_2+j_3+j_4)/2}2^{-2k_{max}}2^{k_{min}/2}2^{-j_{max}/2}\prod_{i=1}^4 \|f_{k_i,j_i}\|_{L^2}.\]
\end{cor} 

We end this subsection  introducing the Strichartz estimates for the family of the fifth-order operators $\{e^{t\px^5}\}_{t=-\infty}^{\infty}$.
\begin{lem}[Strichartz estimates for $e^{t\px^5}$ operator \cite{CT2005}]\label{lem:strichartz}
Assume that $-1 < \sigma \le \frac32$ and $0 \le \theta \le 1$. Then there exists $C>0$ depending on $\sigma$ and $\theta$ such that
\[\norm{D^{\frac{\sigma\theta}{2}e^{t\px^5}\varphi}}_{L^q_tL^p_x} \le C \norm{\varphi}_{L^2}\]
for $\varphi \in L^2$, where $p=\frac{2}{1-\theta}$ and $q = \frac{10}{\theta(\sigma +1)}$. In particular, we have 
\[\norm{e^{t\px^5}P_k\varphi}_{L_{t,x}^6} \lesssim  2^{-k/2}\norm{P_k\varphi}_{L^2}, \quad k\ge 1.\]
\end{lem}

\subsection{Nonlinear estimates}\label{sec:nonlinear}
We first recall from \cite{GKK2013} the bilinear estimates as follows:
\begin{prop}[Nonlinear estimates for $\mathcal N_2(u)$, \cite{GKK2013}]\label{prop:bilinear}
\begin{itemize}
 \item[(a)] If $s \ge 1$, $T\in(0,1]$, and $u,v\in F^s(T)$ then 
\[\|\mathcal N_2(u)\|_{N^s(T)} \lesssim \|u\|_{F^{s}(T)}^2 + \|u\|_{F^{s}(T)}^3.\]
and
\[\|\mathcal N_2(u) - \mathcal N_2(v)\|_{N^s(T)} \lesssim \left(\|u\|_{F^s(T)} + \|v\|_{F^s(T)} \right)\|u-v\|_{F^s(T)} + \left(\|u\|_{F^s(T)}^2 + \|v\|_{F^s(T)}^2 \right)\|u-v\|_{F^s(T)}.\]
\item[(b)] If $T \in (0,1]$, $u, v \in F^0(T) \cap F^{2}(T)$, then
\[\|\mathcal N_2(u) - \mathcal N_2(v)\|_{N^0(T)}\lesssim \left(\|u\|_{F^2(T)} + \|v\|_{F^2(T)}\right)\|u-v\|_{F^{0}(T)} + \left(\|u\|_{F^2(T)}^2 + \|v\|_{F^2(T)}^2\right)\|u-v\|_{F^{0}(T)}.\]
\end{itemize}
\end{prop}

\begin{proof}
The proof for the quadratic term, we refer to \cite{GKK2013}. On the other hand, one can easily control the cubic term in $\mathcal N_2(u)$ rather than not only the quadratic term in $\mathcal N_2(u)$, but also the cubic term in $\mathcal N_3(u)$, since the cubic term in $\mathcal N_2(u)$ contains only one (total) derivative. Hence we omit the details, but one can capture the estimates in the proof of Proposition \ref{prop:trilinear}.
\end{proof}

\begin{prop}[Nonlinear estimates for $\mathcal N_3(u)$]\label{prop:trilinear}
(a) If $s \ge 2$, $T\in(0,1]$, and $u,v\in F^s(T)$ then 
\begin{equation}\label{eq:trilinear1.1}
\|\mathcal N_3(u)\|_{N^s(T)} \lesssim \|u\|_{F^{s}(T)}^3 + \|u\|_{F^{s}(T)}^5.
\end{equation}
and
\[\|\mathcal N_3(u) - \mathcal N_3(v)\|_{N^s(T)} \lesssim \left(\|u\|_{F^{s}(T)}^2 + \|v\|_{F^{s}(T)}^2  \right)\|u-v\|_{F^s(T)} + \left(\|u\|_{F^{s}(T)}^4 + \|v\|_{F^{s}(T)}^4 \right)\|u-v\|_{F^s(T)}.\]

(b) If $T \in (0,1]$, $u, v \in F^0(T) \cap F^{2}(T)$, then
\[\|\mathcal N_3(u) - \mathcal N_3(v)\|_{N^0(T)}\lesssim \left(\|u\|_{F^2(T)}^2 + \|v\|_{F^2(T)}^2  \right)\|u-v\|_{F^0(T)} + \left(\|u\|_{F^2(T)}^4 + \|v\|_{F^2(T)}^4 \right)\|u-v\|_{F^0(T)}.\]
\end{prop}

\begin{proof}
We first consider the cubic term in $\mathcal N_3(u)$. From the support property \eqref{eq:tri-support property} in addition to \eqref{eq:tri-resonant function}, we know that
\begin{equation}\label{eq:high modulation}
\max(|\tau_j-w(\xi_j)|;j=1,2,3,4) \gtrsim |(\xi_1+\xi_2)(\xi_1+\xi_3)(\xi_2+\xi_3)|(\xi_1^2+\xi_2^2+\xi_3^2+\xi_4^2).
\end{equation}

\medskip

From the definition of $N_k$ norm, the left-hand side of  (cubic terms in) \eqref{eq:trilinear1.1} is bounded by
\begin{equation}\label{eq:hihihihi-1}
\begin{aligned}
\sup_{t_{k_4} \in \R} &\Big\|(\tau_4 - w(\xi_4)+ i2^{2k_4})^{-1}2^{3k_4}\mathbf{1}_{I_{k_4}}(\xi)\mathcal F \left[\eta_0\left(2^{2k_4-2}(t-t_{k_4})\right)P_{k_1}u\right] \\
&\hspace{3em}\ast \mathcal F \left[\eta_0\left(2^{2k_4-2}(t-t_{k_4})\right)P_{k_2}u\right] \ast \mathcal F \left[\eta_0\left(2^{2k_4-2}(t-t_{k_4})\right)P_{k_3}u\right]\Big\|_{X_{k_4}}
\end{aligned}
\end{equation}
We set $u_{k_i} = \mathcal F \left[\eta_0\left(2^{2k_4-2}(t-t_{k_4})\right)P_{k_i}u\right],$ $i=1,2,3$. We decompose each $u_{k_i}$ into $u_{k_i,j_i}(\tau,\xi) = u_{k_i}(\tau,\xi)\eta_{j_i}(\tau - w(\xi))$ with usual modification like $f_{\le j}(\tau) = f(\tau)\eta_{\le j}(\tau-\mu(n))$. Then, \eqref{eq:hihihihi-1} is bounded by
\[\sum_{j_4 \ge 0} \frac{2^{3k_4}2^{j_4/2}\beta_{k_4,j_4}}{\max(2^{j_4},2^{2k_4})} \sum_{j_1,j_2,j_3 \ge 2k_4}\norm{\mathbf{1}_{D_{k_4,j_4}}\cdot(u_{k_1,j_1} \ast u_{k_2,j_2} \ast u_{k_3,j_3})}_{L^2}.\]

\medskip

We, instead of Corollary \ref{cor:tri-L2}, use the following observation to control 
\[\norm{\mathbf{1}_{D_{k_4,j_4}}\cdot(u_{k_1,j_1} \ast u_{k_2,j_2} \ast u_{k_3,j_3})}_{L^2}\]
for particular case: Lemma \ref{lem:strichartz} yields
\begin{equation}\label{eq:L6}
\begin{aligned}
\norm{\mathcal F ^{-1}[u_{k_i,j_i}]}_{L^6} &= \left\|\int e^{it\tau}e^{ix\xi}e^{itw(\xi)}u_{k_i,j_i}^{\sharp}(\tau,\xi) \;d\xi d\tau \right\|_{L^6}\\
&\lesssim \int\left\| \int e^{ix\xi}e^{itw(\xi)}u_{k_i,j_i}^{\sharp}(\tau,\xi) \;d\xi  \right\|_{L^6} \; d\tau\\
&\lesssim 2^{-k_i/2}2^{j_i/2}\|u_{k_i,j_i}^{\sharp}\|_{L^2},
\end{aligned}
\end{equation}
where $u_{k_i,j_i}^{\sharp}(\tau, \xi) = u_{k_i,j_i}(\tau+w(\xi),\xi)$ with $\|u_{k_i,j_i}^{\sharp}\|_{L^2} = \|u_{k_i,j_i}\|_{L^2}$. With this, Plancherel's theorem and the H\"older inequality give 
\begin{equation}\label{eq:tri-hhh}
\norm{\mathbf{1}_{D_{k_4,j_4}}\cdot(u_{k_1,j_1} \ast u_{k_2,j_2} \ast u_{k_3,j_3})}_{L^2} \lesssim 2^{-(k_1+k_2+k_3)/2}2^{(j_1+j_2+j_3)/2}\prod_{i=1}^3 \norm{u_{k_i,j_i}}_{L^2}.
\end{equation}

\medskip

\textbf{Case I.} (high-high-high $\Rightarrow$ high). Let $k_4 \ge 20$ and $|k_1-k_4|, |k_2-k_4|, |k_3-k_4| \le 5$. Applying \eqref{eq:tri-hhh} to $\norm{\mathbf{1}_{D_{k_4,j_4}}\cdot(u_{k_1,j_1} \ast u_{k_2,j_2} \ast u_{k_3,j_3})}_{L^2}$, one has
\[\eqref{eq:hihihihi-1} \lesssim \sum_{j_4 \ge 0} \frac{2^{3k_4}2^{j_4/2}\beta_{k_4,j_4}}{\max(2^{j_4},2^{2k_4})} \sum_{j_1,j_2,j_3 \ge 2k_4}2^{-k_4/2}2^{(j_1+j_2+j_3)/2} \prod_{i=1}^3 \norm{u_{k_i,j_i}}_{L^2}.\]
Note that $\beta_{k_4,j_4} \sim 1$ when $0 \le j_4 \le 5k$. Let denote the summand by $\mathcal M_I$, i.e.,
\begin{equation}\label{summand1}
\mathcal M_I := \frac{2^{3k_4}2^{j_4/2}\beta_{k_4,j_4}}{\max(2^{j_4},2^{2k_4})}2^{-3k_4/2}2^{(j_1+j_2+j_3)/2}.
\end{equation}
Then, we know
\[\mathcal M_I \lesssim 2^{j_4/2}2^{-k_4/2}2^{(j_1+j_2+j_3)/2}, \quad \mbox{when } \;\; 0 \le j_4 \le 2k_4,\]
\[\mathcal M_I \lesssim 2^{-j_4/2}2^{3k_4/2}2^{(j_1+j_2+j_3)/2}, \quad \mbox{when } \;\; 2k_4 \le j_4 \le 5k_4\]
and
\[\mathcal M_I \lesssim 2^{-j_4/2}2^{(j_4-5k_4)/8}2^{3k_4/2}2^{(j_1+j_2+j_3)/2}, \quad \mbox{when } \;\; 5k_4 \le j_4.\]
Performing summations over $j_i$, $i=1,2,3,4$, we have
\[\eqref{eq:hihihihi-1} \lesssim 2^{k_4/2}\prod_{i=1}^3 \norm{P_{k_i}u}_{F_{k_i}}.\]
Indeed, we have from the definition of $X_k$-norm and (2.4) in \cite{Guo2012} that
\begin{align*}
\sum_{j_1 \ge 2k_4}2^{j_1/2}\norm{u_{k_1,j_1}}_{L^2} \lesssim&~{} \sum_{j_1 > 2k_4} 2^{j_1/2} \beta_{k_1,j_1} \Big\|\eta_{j_1}(\tau-\omega(\xi)) \int_{\R}|\bar{u}_{k_1}(\xi,\tau')|\cdot 2^{-2k_4}(1+2^{-2k_4}|\tau-\tau'|)^{-4}d\tau'\Big\|_{L^2}\\
&+2^{(2k_4)/2}\normo{\eta_{\leq 2k_4}(\tau-\omega(\xi)) \int_{\R} |\bar{u}_{k_1}(\xi,\tau')| 2^{-2k_4}(1+2^{-2k_4}|\tau-\tau'|)^{-4}d\tau'}_{L^2}\\
&\lesssim \norm{u_{k_1}}_{X_{k_1}}  \lesssim \norm{P_{k_1}u}_{F_{k_1}},
\end{align*}
where $\bar{u}_{k_1}=\mathcal F [P_{k_1}u \cdot \eta_0(2^{2k_1}(t-t_{k_1}))]$.

\medskip

\begin{rem}
As seen in the proof of \textbf{Case I} (also for other cases except for the case when the resulting frequency ($\xi_4$) is not the maximum frequency), the weight $\beta_{k_4,j_4}$ does not play any role in the estimates, hence is negligible. See \textbf{Case I} and \textbf{Case III} (below) for comparison.
\end{rem}

\medskip

\textbf{Case II.} (high-high-low $\Rightarrow$ high). Let $k_4 \ge 20$, $|k_2-k_4|, |k_3-k_4| \le 5$ and $k_1 \le k_4-10$\footnote{Thanks to the symmetry of frequencies, our assumption that $\xi_1$ is the minimum frequency does not lose of the generality.}. In this case, we have $j_{max} \ge 5k_4$ due to \eqref{eq:high modulation}. The exactly same argument used in \textbf{Case I} (but use Corollary \ref{cor:tri-L2} (a) instead of \eqref{eq:tri-hhh} to control $\norm{\mathbf{1}_{D_{k_4,j_4}}\cdot(u_{k_1,j_1} \ast u_{k_2,j_2} \ast u_{k_3,j_3})}_{L^2}$) gives better result as follows:
\[\eqref{eq:hihihihi-1} \lesssim 2^{k_1/2}\prod_{i=1}^3 \norm{u_{k_i}}_{X_{k_i}} \lesssim 2^{k_1/2}\prod_{i=1}^3 \norm{P_{k_i}u}_{F_{k_i}}.\]
The last inequality holds true, thanks to \eqref{eq:pXk3}, more precisely, 
\[\norm{u_{k_1}}_{X_{k_1}} = \norm{\mathcal F \left[\eta_0\left((2^{2k_4-2}(t-t_{k_4})\right) \cdot P_{k_1}u \cdot \eta_0\left(2^{2k_1}(t-t_{k_4})\right)\right]}_{X_{k_1}} \lesssim \norm{P_{k_1}u}_{F_{k_1}}.\] 
We omit the details.

\medskip

\textbf{Case III.} (high-high-high $\Rightarrow$ low). Let $k_3 \ge 20$, $|k_1-k_3|, |k_2-k_3| \le 5$ and $k_4 \le k_3-10$. 

\begin{rem}\label{rem:longer interval}
The trade-off of the use of the short time advantage (also, the use of the weight as in \eqref{beta weight}) is to worsen some interactions for which the resulting frequency is lower than (at least) one of others, in particular, \emph{high-high-high $\Rightarrow$ low} and \emph{high-high-low $\Rightarrow$ low} interaction components. More precisely, in the case of the \emph{high-high-high $\Rightarrow$ low}, the time interval of length $2^{-2k_4}$, on which the $N_{k_4}$-norm is taken, is longer than the interval of length $2^{-2k_i}$, on which $F_{k_i}$-norm is taken, $i=1,2,3$. In order to cover whole intervals of length $2^{-2k_4}$ in the estimates, one needs to divide the time interval of length $2^{-2k_4}$ into $2^{2k_3-2k_4}$ intervals of length $2^{-2^{2k_3}}$. Let choose $\gamma: \R \to [0,1]$ (a kind of the partition of unity), which is a smooth function supported in $[-1,1]$ with $ \sum_{m\in \Z} \gamma^3(x-m) \equiv 1$. Then, the left-hand side of  (cubic terms in) \eqref{eq:trilinear1.1} is bounded by (instead of \eqref{eq:hihihihi-1})
\begin{equation}\label{eq:low frequency}
\begin{split}
\sup_{t_k\in \R}&2^{3k_3}\Big\|(\tau_4-w(\xi_4) +i 2^{2k_4})^{-1}\mathbf{1}_{I_{k_4}}\cdot  \sum_{|m| \le C 2^{2k_3-2k_4}} \mathcal F [\eta_0(2^{2k_4}(t-t_k))\gamma (2^{2k_3}(t-t_k)-m)P_{k_1}u]\\
 &\hspace{3em}\ast \mathcal F [\eta_0(2^{2k_4}(t-t_k))\gamma (2^{2k_3}(t-t_k)-m)P_{k_2}u] \ast \mathcal F [\eta_0(2^{2k_4}(t-t_k))\gamma (2^{2k_3}(t-t_k)-m)P_{k_3}u]\Big\|_{X_{k_4}}.
\end{split}
\end{equation}
The analogous procedure will be applied to the estimate of \emph{high-high-low $\Rightarrow$ low} interaction component below.
\end{rem}

\medskip

\noindent When $k_4 = 0$, \eqref{eq:low frequency} is bounded by
\[\sum_{j_4 \ge 0} 2^{5k_3} \sum_{j_1,j_2,j_3 \ge 2k_3}\norm{\mathbf{1}_{D_{0,j_4}}\cdot(u_{k_1,j_1} \ast u_{k_2,j_2} \ast u_{k_3,j_3})}_{L^2}\]
due to \eqref{beta weight}.

When $0 \le j_4 \le 5k_3 - 5$ or $\le 5k_3 + 5 \le j_4$, we apply Corollary \ref{cor:tri-L2} (a) to $\norm{\mathbf{1}_{D_{k_4,j_4}}\cdot(u_{k_1,j_1} \ast u_{k_2,j_2} \ast u_{k_3,j_3})}_{L^2}$ to obtain
\[\eqref{eq:low frequency} \lesssim \left(\sum_{0 \le j_4 \le 5k_3 - 5} + \sum_{5k_3 + 5 \le j_4}\right)2^{5k_3} \sum_{j_1,j_2,j_3 \ge 2k_3}2^{(j_{min}+j_{thd})/2}2^{k_3/2} \prod_{i=1}^3 \norm{u_{k_i,j_i}}_{L^2}.\]
We know $j_4 \neq j_{max}$ in the former case, while $j_4 = j_{max}$ and $2^{j_{max}} \sim 2^{j_{med}} \gg |H|$ in the latter case. 

On the other hand, when $5k_3 - 5 \le j_4 \le 5k_3 + 5$ ($2^{j_4} = 2^{j_{max}} \sim |H| \gg 2^{j_{med}}$), we use \eqref{eq:tri-hhh}. Then, similarly as the previous cases, we have (when $k_4=0$)\footnote{One can see that the worst bound comes from the low frequency with high modulation case ($j_4 = j_{max} > j_{med} + 5$).}
\[\eqref{eq:low frequency} \lesssim 2^{\frac72k_3} \prod_{i=1}^{3}\norm{P_{k_i}u}_{F_{k_i}}.\]

\medskip

\noindent When $k_4 \neq 0$, similarly as above, \eqref{eq:low frequency} is bounded by
\[\sum_{j_4 \ge 0}  \frac{2^{5k_3-2k_4}2^{j_4/2}\beta_{k_4,j_4}}{\max(2^{j_4},2^{2k_4})} 2^{(k_3+k_4)/2}\sum_{j_1,j_2,j_3 \ge 2k_3}2^{(j_{min}+j_{thd})/2} \prod_{i=1}^3 \norm{u_{k_i,j_i}}_{L^2},\]
thanks to Corollary \ref{cor:tri-L2} (a), except for the case when $5k_3 - 5 \le j_4 \le 5k_3 + 5$. Let denote the summand by $\mathcal M_{III}$, similarly as in \eqref{summand1} i.e.,
\[\mathcal M_{III} := \frac{2^{5k_3-2k_4}2^{j_4/2}\beta_{k_4,j_4}}{\max(2^{j_4},2^{2k_4})} 2^{(k_3+k_4)/2}2^{(j_{min}+j_{thd})/2}.\]
If $2k_3 < 5k_4$, we know
\[\mathcal M_{III} \lesssim 2^{j_4}2^{2k_3}2^{-7k_4/2}2^{(j_1+j_2+j_3)/2}, \quad \mbox{when } \;\; 0 \le j_4 \le 2k_4,\]
\[\mathcal M_{III} \lesssim 2^{2k_3}2^{-3k_4/2}2^{(j_1+j_2+j_3)/2}, \quad \mbox{when } \;\; 2k_4 \le j_4 \le 2k_3,\]
\[\mathcal M_{III} \lesssim 2^{-j_4/2}2^{3k_3}2^{-3k_4/2}2^{(j_1+j_2+j_3)/2}, \quad \mbox{when } \;\; 2k_3 \le j_4 \le 5k_4,\]
\[\mathcal M_{III} \lesssim 2^{-j_4/2}2^{(j_4-5k_4)/8}2^{3k_3}2^{-3k_4/2}2^{(j_1+j_2+j_3)/2}, \quad \mbox{when } \;\; 5k_4 \le j_4 \le 5k_3-4\]
and
\[\mathcal M_{III} \lesssim 2^{-j_4/2}2^{(j_4-5k_4)/8}2^{3k_3}2^{-3k_4/2}2^{(j_1+j_2+j_3)/2}, \quad \mbox{when } \;\; 5k_3+4 \le j_4.\]
Otherwise (when $5k_4 < 2k_3$), the estimates of $\mathcal M_{III}$ on $2k_4 \le j_4 \le 5k_3-4$ are replaced by
\[\mathcal M_{III} \lesssim 2^{2k_3}2^{-3k_4/2}2^{(j_1+j_2+j_3)/2}, \quad \mbox{when } \;\; 2k_4 \le j_4 \le 5k_4,\]
\[\mathcal M_{III} \lesssim 2^{(j_4-5k_4)/8}2^{2k_3}2^{-3k_4/2}2^{(j_1+j_2+j_3)/2}, \quad \mbox{when } \;\; 5k_4 \le j_4 \le 2k_3\]
and
\[\mathcal M_{III} \lesssim 2^{-j_4/2}2^{(j_4-5k_4)/8}2^{3k_3}2^{-3k_4/2}2^{(j_1+j_2+j_3)/2}, \quad \mbox{when } \;\; 2k_3 \le j_4 \le 5k_3-4.\]
On the other hand, when $5k_3 - 5 \le j_4 \le 5k_3 + 5$, we use \eqref{eq:tri-hhh} to obtain
\[\eqref{hihihilo-2} \lesssim 2^{5k_3-2k_4}2^{-5k_3/2}2^{\frac58(k_3-k_4)}2^{-3k_3/2}\sum_{j_1,j_2,j_3 \ge 2k_3}\prod_{i=1}^3 2^{j_i}\norm{u_{k_i,j_i}}_{L^2}.\]
Summing over $j_i$, $i=1,2,3,4$, one has 
\begin{equation}\label{hihihilo-2}
\eqref{eq:low frequency} \lesssim C_1(k_3,k_4)\prod_{i=1}^3 \norm{P_{k_i}u}_{F_{k_i}},
\end{equation}
where
\begin{eqnarray*}
C_1(k_3,k_4) =  \left \{ \begin{array}{lr}
2^{\frac75k_3}, &2k_3 < 5k_4,\\
2^{\frac94k_3}2^{-\frac{17}{8}k_4}, &5k_4 \le 2k_3.
\end{array}
\right.
\end{eqnarray*}

\begin{rem}\label{rem:total derivative}
A direct computation in the cubic term in $\mathcal N_3(u)$, one has  
\[40uu_x u_{xx} + 10 u^2u_{xxx} + 10 u_x^3 =  10 (u^2u_{xx})_x + 10 (uu_x^2)_x.\]
Then, one can reduce $2^{3k_3}$ in \eqref{eq:low frequency} by $2^{2k_3+k_4}$, and hence obtain a better result. However, our regularity threshold is $s=2$, and hence we, here, do not explore the trilinear estimates in lower regularity. 
\end{rem}

\textbf{Case IV.} (high-low-low $\Rightarrow$ high). Let $k_4 \ge 20$, $|k_3-k_4| \le 5$ and $k_1, k_2 \le k_4 -10$. Without loss of generality, we may assume that $k_1 \le k_2$, thanks to the symmetry. Similarly as the \textbf{Case I}, it is enough to consider
\begin{equation}\label{hilolohi}
\sum_{j_4 \ge 0} \frac{2^{3k_4}2^{j_4/2}\beta_{k_4,j_4}}{\max(2^{j_4},2^{2k_4})} \sum_{j_1,j_2,j_3 \ge 2k_4}\norm{\mathbf{1}_{D_{k_4,j_4}}\cdot(u_{k_1,j_1} \ast u_{k_2,j_2} \ast u_{k_3,j_3})}_{L^2}.
\end{equation}
Let denote the summand in \eqref{hilolohi} by $\mathcal M_{IV}$, i.e.,
\[\mathcal M_{IV} :=  \frac{2^{3k_4}2^{j_4/2}\beta_{k_4,j_4}}{\max(2^{j_4},2^{2k_4})} \norm{\mathbf{1}_{D_{k_4,j_4}}\cdot(u_{k_1,j_1} \ast u_{k_2,j_2} \ast u_{k_3,j_3})}_{L^2}.\]
We further split this case into three cases: \textbf{Case IV-a} $k_2=0$, \textbf{Case IV-b}  $k_1= 0$ and $k_2 \neq 0 $, and  \textbf{Case IV-c}  $k_1 \neq 0$. 

\medskip

\textbf{Case IV-a.} $k_2=0$. We do not distinguish Corollary \eqref{cor:tri-L2} (b.1) and (b.2), since $2^{k_{min}} \le 2^{k_{thd}} \le 1$. Then, from Corollary \ref{cor:tri-L2} (b), we have\footnote{We use, here, $2^{j_{max}} \ge 2^{2k_4}$ to deal with a maximum modulation, since our purpose is to obtain the local well-posedness only in $H^s(\R)$, $s \ge 2$. However, one may obtain the better result by performing a delicate calculation in addition to $2^{j_{max}} \ge |H|$, instead of $2^{j_{max}} \ge 2^{2k_4}$. For the same reason, so the \emph{high-high-low $\Rightarrow$ low} case below as well.}
\[\mathcal M_{IV} \lesssim 2^{j_4}2^{-2k_4}\prod_{i=1}^3 2^{j_i/2}\norm{u_{k_i,j_i}}_{L^2},\quad \mbox{when } \;\; 0 \le j_4 \le 2k_4\]
and
\[\mathcal M_{IV} \lesssim 2^{-j_4/2}\beta_{k_4,j_4}\prod_{i=1}^3 2^{j_i/2}\norm{u_{k_i,j_i}}_{L^2},\quad \mbox{when } \;\;  2k_4 \le j_4.\]
Summing over $j_i$, $i=1,2,3,4$, one has
\[\eqref{eq:hihihihi-1} \lesssim \prod_{i=1}^3 \norm{u_{k_i}}_{F_{k_i}}.\]

\medskip

\textbf{Case IV-b.}  $k_1= 0$ and $k_2 \neq 0 $. Note that we have $j_{max} \ge 4k_4 + k_2$ due to \eqref{eq:high modulation}. 
We use Corollary \ref{cor:tri-L2} (b.1) (the worst case occurring in $j_2 = j_{max}$) when $0 \le j_4 \le 4k_4 + k_2 -5$, and (b.2) ($j_2 = j_{max}$ never happens) when $4k_4 + k_2 -5 \le j_4$ to control $\norm{\mathbf{1}_{D_{k_4,j_4}}\cdot(u_{k_1,j_1} \ast u_{k_2,j_2} \ast u_{k_3,j_3})}_{L^2}$ in $\mathcal M_{IV}$, then we have 
\[\mathcal M_{IV} \lesssim 2^{j_4}2^{-3k_4}\prod_{i=1}^3 2^{j_i/2}\norm{u_{k_i,j_i}}_{L^2},\quad \mbox{when } \;\; 0 \le j_4 \le 2k_4,\]
\[\mathcal M_{IV} \lesssim 2^{-k_4}\prod_{i=1}^3 2^{j_i/2}\norm{u_{k_i,j_i}}_{L^2},\quad \mbox{when } \;\; 2k_4 \le j_4 \le 4k_4+k_2-5\]
and
\[\mathcal M_{IV} \lesssim 2^{-j_4/2}\beta_{k_4,j_4}2^{k_4}\prod_{i=1}^3 2^{j_i/2}\norm{u_{k_i,j_i}}_{L^2},\quad \mbox{when } \;\;  4k_4+k_2-5 \le j_4.\]
Summing over $j_i$, $i=1,2,3,4$, one has
\[\eqref{eq:hihihihi-1} \lesssim k_42^{-k_4}\prod_{i=1}^3 \norm{u_{k_i}}_{F_{k_i}}.\]

\medskip
 
\textbf{Case IV-c.}  $k_1 \neq 0$. Similarly as \textbf{Case IV-a} (if $|k_1-k_2| < 5$ with Corollary \ref{cor:tri-L2} (b.2)) or \textbf{Case IV-b} (if $k_1 < k_2 -5$), we have at most 
\[\eqref{eq:hihihihi-1} \lesssim 2^{k_1/2}\prod_{i=1}^3 \norm{u_{k_i}}_{F_{k_i}}.\]

\medskip

\textbf{Case V.} (high-high-low $\Rightarrow$ low). Let $k_3 \ge 20$, $|k_2-k_3| \le 5$ and $k_1,k_4 \le k_3 -10$. We first divide this case into two cases: \textbf{Case V-a} $k_4 = 0$ and \textbf{Case V-b} $k_4 \neq 0$.

\medskip

\textbf{Case V-a.} $k_4 = 0$. From Remark \ref{rem:longer interval}, it suffices to consider
\begin{equation}\label{hihilolo-1}
\sum_{j_4 \ge 0} 2^{5k_3} \sum_{j_1,j_2,j_3 \ge 2k_3}\norm{\mathbf{1}_{D_{0,j_4}}\cdot(u_{k_1,j_1} \ast u_{k_2,j_2} \ast u_{k_3,j_3})}_{L^2}
\end{equation}
due to \eqref{beta weight}. If $k_1 = 0$, by using Corollary \ref{cor:tri-L2} (a), we have
\begin{equation}\label{hihilolo-2}
\eqref{hihilolo-1} \lesssim 2^{4k_3}\prod_{i=1}^3 \norm{u_{k_i}}_{F_{k_i}}.
\end{equation}
More precisely, when $|\xi_2+\xi_3| \ll \xi_3^{-2}$ (equivalently $|H| \ll 2^{2k_3}$) we know 
\[2^{(j_{min}+j_{thd})/2} \le 2^{(j_1+j_2+j_3+j_4)/2}2^{-2k_3}, \quad \mbox{when } \;\; 0 \le j_4 \le 2k_3\]
and
\[2^{(j_{min}+j_{thd})/2} \le 2^{(j_1+j_2+j_3)/2}2^{-j_4/2}, \quad \mbox{when } \;\; 2k_3 \le j_4.\]
When $|\xi_2+\xi_3| \gtrsim \xi_3^{-2}$ (equivalently $|H| \gtrsim 2^{2k_3}$), we know
\[2^{(j_{min}+j_{thd})/2} \le 2^{(j_1+j_2+j_3+j_4)/2}2^{-2k_3}, \quad \mbox{when } \;\; 1 \le 2^{j_4} \le 2^{2k_3}\]
\[2^{(j_{min}+j_{thd})/2} \le 2^{(j_1+j_2+j_3)/2}2^{-j_4/2}, \quad \mbox{when } \;\; 2^{2k_3} \le 2^{j_4} \le |H|/2\]
\[2^{(j_{min}+j_{thd})/2} \le 2^{(j_1+j_2+j_3)/2}2^{-k_3}, \quad \mbox{when } \;\; |H|/2 \le 2^{j_4} \le 3|H|/2\]
and
\[2^{(j_{min}+j_{thd})/2} \le 2^{(j_1+j_2+j_3)/2}2^{-j_4/2}, \quad \mbox{when } \;\; 3|H|/2 \le 2^{j_4}.\]
Note that the number of $j_4$ is finite ($\le 10$) when $|H|/2 \le 2^{j_4} \le 3|H|/2$. Thus, the summation over $j_i$, $i=1,2,3,4$, yields \eqref{hihilolo-2}. 

Otherwise ($k_1 \neq 0$), similarly as \textbf{Case IV-b}, we have
\[\eqref{hihilolo-1} \lesssim 2^{3k_3}2^{k_1/2}\prod_{i=1}^3 \norm{u_{k_i}}_{F_{k_i}}.\]

\medskip

\textbf{Case V-b.} $k_4 \neq 0$. Similarly, it is enough to consider
\begin{equation}\label{hihilolo-4}
\sum_{j_4 \ge 0}  \frac{2^{5k_3-2k_4}2^{j_4/2}\beta_{k_4,j_4}}{\max(2^{j_4},2^{2k_4})} \sum_{j_1,j_2,j_3 \ge 2k_3}\norm{\mathbf{1}_{D_{k_4,j_4}}\cdot(u_{k_1,j_1} \ast u_{k_2,j_2} \ast u_{k_3,j_3})}_{L^2}.
\end{equation}
Let denote the summand in \eqref{hihilolo-4} by $\mathcal M_{V}$, i.e.,
\[\mathcal M_{V} :=  \frac{2^{5k_3-2k_4}2^{j_4/2}\beta_{k_4,j_4}}{\max(2^{j_4},2^{2k_4})}\norm{\mathbf{1}_{D_{k_4,j_4}}\cdot(u_{k_1,j_1} \ast u_{k_2,j_2} \ast u_{k_3,j_3})}_{L^2}.\]
If $k_1 = 0$, we use Corollary \ref{cor:tri-L2} (b) to control $\norm{\mathbf{1}_{D_{k_4,j_4}}\cdot(u_{k_1,j_1} \ast u_{k_2,j_2} \ast u_{k_3,j_3})}_{L^2}$. Then, we know
\[\mathcal M_{V} \lesssim 2^{j_4}2^{k_3}2^{-9k_4/2}2^{(j_1+j_2+j_3)/2}, \quad \mbox{when } \;\; 0 \le j_4 \le 2k_4,\]
\[\mathcal M_{V} \lesssim 2^{k_3}2^{-5k_4/2}2^{(j_1+j_2+j_3)/2}, \quad \mbox{when } \;\; 2k_4 \le j_4 \le 5k_4,\]
\[\mathcal M_{V} \lesssim 2^{3k_3/2}2^{-3k_4}2^{(j_1+j_2+j_3)/2}, \quad \mbox{when } \;\; 5k_4 \le j_4 \le 4k_3 + k_4 -5\]
and
\[\mathcal M_{V} \lesssim 2^{-j_4/2}2^{(j_4-5k_4)/8}2^{3k_3}2^{-3k_4/2}2^{(j_1+j_2+j_3)/2}, \quad \mbox{when } \;\; 4k_3 + k_4 - 5 \le j_4.\]
Summing over $j_i$, $i=1,2,3,4$, one has
\[\eqref{hihilolo-1} \lesssim \max \left(k_32^{3k_3/2}2^{-3k_4}, 2^{3k_3/2}2^{-2k_4} \right)\prod_{i=1}^3 \norm{u_{k_i}}_{F_{k_i}}.\]

Otherwise ($k_1 \neq 0$), analogous arguments as \textbf{Case V-b} (for $|k_1-k_4| \ge 5$ case) and \textbf{Case V-a}, in particular $k_1=0$ case, (for $|k_1-k_4| \le 5$ case) can be applied, and hence we have (but, omit the details)
\[\eqref{hihilolo-1} \lesssim C_2(k_1,k_3,k_4)\prod_{i=1}^3 \norm{u_{k_i}}_{F_{k_i}},\]
where
\begin{eqnarray*}
C_2(k_1,k_3, k_4) =  \left \{ \begin{array}{lr}
2^{3(k_3-k_4)/2}\max\left(k_32^{-k_4}, 1 \right), &k_1 \le k_4 -5,\\
2^{3(k_3-k_4)/2}\max\left(k_32^{-k_4/2}2^{k_1/8}, 2^{-k_4}2^{(k_1-k_4)/8}\right), &k_4 \le k_1 -5,\\
2^{4k_3}2^{k_1},& |k_1-k_4| \le 5
\end{array}
\right.
\end{eqnarray*}

\medskip

The estimate of \emph{low-low-low $\Rightarrow$ low} interaction component can be easily obtained, and hence we omit the details. On the other hand, the estimate of quintic term in $\mathcal N_3(u)$ will be taken into account in the estimate of $\mathcal{SN}(u)$ below. Thus, by collecting all, we complete the proof. 
\end{proof}

\begin{prop}[Nonlinear estimates for $\mathcal{SN}(u)$]\label{prop:multilinear}
(a) If $s \ge 2$, $T\in(0,1]$, and $u,v\in F^s(T)$ then 
\[\|\mathcal{SN}(u)\|_{N^s(T)} \lesssim \|u\|_{F^{s}(T)}^2 + \|u\|_{F^{s}(T)}^4.\]
and
\[\|\mathcal{SN}(u) -\mathcal{SN}(v)\|_{N^s(T)} \lesssim \left(\|u\|_{F^{s}(T)} + \|v\|_{F^{s}}(T) + \|u\|_{F^{s}(T)}^3 + \|v\|_{F^{s}(T)}^3 \right)\|u-v\|_{F^s(T)}.\]

(b) If $T \in (0,1]$, $u, v \in F^0(T) \cap F^{2}(T)$, then
\[\|\mathcal{SN}(u) - \mathcal{SN}(v)\|_{N^0(T)}\lesssim \left(\|u\|_{F^{2}(T)} + \|v\|_{F^{2}}(T) + \|u\|_{F^{2}(T)}^3 + \|v\|_{F^{2}(T)}^3 \right)\|u-v\|_{F^{0}(T)}.\]
\end{prop}

\begin{proof}
The quadratic term in $\mathcal{SN}(u)$ can be easily treated compared to one in $\mathcal{N}_2(u)$, due to a less number of derivatives, similarly as the cubic term in $\mathcal{N}_2(u)$ compared to one in $\mathcal{N}_3(u)$. The rest of the proof (also for the quadratic and cubic terms in $\mathcal{SN}(u)$ and $\mathcal{N}_3(u)$, respectively) is based on the following direct computation
\[\norm{\mathbf{1}_{D_{k_{\ell},j_{\ell}}}\cdot(u_{k_1,j_1} \ast \cdots \ast u_{k_{\ell-1},j_{\ell-1}})}_{L^2} \lesssim 2^{-(k_{max}+k_{med})/2}2^{-(j_{max}+j_{med})/2}2^{k_{\ell}/2}2^{k_{\ell}/2}\prod_{i=1}^{\ell-1}2^{j_i/2}2^{k_i/2}\norm{u_{k_i,j_i}}_{L^2},\]
which can be obtained by Cauchy-Schwarz inequality. Due to a less number of derivatives (indeed, one (total) derivative) in $\mathcal{SN}(u)$, the analogous (but much simpler) argument used in the proof of Proposition \ref{prop:trilinear} immediately yields Proposition \ref{prop:multilinear}. In particular, the total derivative form enables us to drop one derivative taken in a high frequency mode, see Remark \ref{rem:total derivative}. We omit the details.
\end{proof}

\subsection{Energy estimates}
Assume that $u, \mathcal{G} \in C([-T,T];L^2)$ satisfy
\[\begin{cases}
u_t + u_{5x} = \mathcal{G}, \quad  (x,t) \in \R \times (-T,T) \\
u(0,x) =u_0(x)
\end{cases}\]
A direct calculation gives
\begin{equation}\label{eq:energy 1}
\sup_{|t_k| \le T}\|u(t_k)\|_{L^2}^2 \le \|u_0\|_{L^2}^2 + \sup_{|t_k| \le T}\left|\int_{\R \times [0,t_k]}u \cdot \mathcal{G} \; dxdt \right|.
\end{equation}
To control the second term (for $\mathcal{N}_2(u)$, $\mathcal{N}_3(u)$ and $\mathcal{SN}(u)$) of the right-hand side of \eqref{eq:energy 1}, we need following lemmas.
\begin{lem}[\cite{GKK2013}]\label{lem:energy-quadratic}
Let $T \in (0,1]$ and $k_1,k_2,k_3 \in \Z_+$.

(a) Assume $k_1\leq k_2\leq k_3$ and $|k_{3}-k_{1}| \le 5$, $u_i \in F_{k_i}(T), i=1,2,3$. Then
\[ \left| \int_{\R\times [0,T]} u_1u_2u_3 \; dxdt \right| \lesssim 2^{-\frac74k_{3}} \prod_{i=1}^{3} \norm{u_i}_{F_{k_i}(T)}.\]
 
(b) Assume $k_1\leq k_2\leq k_3$ and $k_{3} \ge 10$, $2^{k_{1}} \ll 2^{k_{2}} \sim 2^{k_{3}}$ and $u_i \in F_{k_i}(T), i=1,2,3$.

If $k_{1} \ge 1$, then
\[\left| \int_{\R\times [0,T]} u_1u_2u_3 \; dxdt \right| \lesssim 2^{-2k_{3}-\frac12k_{1}}\prod_{i=1}^{3} \norm{u_i}_{F_{k_i}(T)}.\]

If $k_{1}=0$, then
\[\begin{aligned}
\left| \int_{\R\times [0,T]} u_1u_2u_3 \; dxdt \right| \lesssim&~{} 2^{-k_{3}}\prod_{i=1}^{3} \norm{u_i}_{F_{k_i}(T)},\\
\left| \int_{\R\times [0,T]} (\partial_xu_1)u_2u_3 \; dxdt \right| \lesssim&~{} 2^{-2k_{3}}\prod_{i=1}^{3} \norm{u_i}_{F_{k_i}(T)}.
\end{aligned}\]

(c) Assume $k_1 \le k -10 $. Then
\[\left| \int_{\R\times[0,T]} P_k(u)P_k(\px^3u\cdot P_{k_1}v) \; dxdt    \right| \lesssim 2^{\frac12k_1} \norm{P_{k_1}v}_{F_{k_1}(T)}\sum_{|k'-k|\le 10} \norm{P_{k'}u}^2_{F_{k'}(T)}.\]

(d) Under the same condition as in (c), we have
\[\left| \int_{\R\times[0,T]} P_k(u)P_k(\px^2u\cdot P_{k_1}\px v) \; dxdt    \right| \lesssim 2^{\frac12k_1} \norm{P_{k_1}v}_{F_{k_1}(T)}\sum_{|k'-k|\le 10} \norm{P_{k'}u}^2_{F_{k'}(T)}.\]
\end{lem}

\begin{lem}\label{lem:energy1-1}
Let $T \in (0,1]$, $k_1,k_2,k_3,k_4 \in \Z_+$, and $u_i \in F_{k_i}(T)$, $i=1,2,3,4$. We further assume $k_1 \le k_2 \le k_3 \le k_4$ with $k_4 \ge 10$. Then

(a) For $|k_1 - k_4| \le 5 $, we have
\begin{equation}\label{eq:energy1-1.1}
\left| \int_{[0,T] \times \R}  u_1u_2u_3u_4 \;dx dt\right| \lesssim 2^{-k_4/2}\prod_{i=1}^{4}\norm{u_i}_{F_{k_i}(T)}.
\end{equation}

(b) For $|k_2 - k_4| \le 5 $ and $k_1 \le k_4 - 10$, we have
\begin{equation}\label{eq:energy1-1.2}
\left| \int_{[0,T] \times \R}  u_1u_2u_3u_4 \;dx dt\right| \lesssim 2^{-k_4}2^{k_1/2}\prod_{i=1}^{4}\norm{u_i}_{F_{k_i}(T)}.
\end{equation}

(c) Let $|k_3 - k_4| \le 5$ and $k_2 \le k_4 -10$. 

In general, we have
\begin{equation}\label{eq:energy1-1.3}
\left| \int_{[0,T] \times \R}  u_1u_2u_3u_4 \;dx dt\right| \lesssim 2^{-k_4}2^{k_1/2}\prod_{i=1}^{4}\norm{u_i}_{F_{k_i}(T)}.
\end{equation}

In particular, if $k_1 =0$ and $k_2 \ge 1$, we have
\begin{equation}\label{eq:energy1-1.3c}
\left| \int_{[0,T] \times \R}  u_1u_2u_3u_4 \;dx dt\right| \lesssim 2^{-2k_4}\prod_{i=1}^{4}\norm{u_i}_{F_{k_i}(T)}.
\end{equation}

If $0 < k' \le k_4 -10$, we have
\begin{equation}\label{eq:energy1-1.3a}
\left| \int_{[0,T] \times \R}  P_{k'}(u_1u_2)u_3u_4 \;dx dt\right| \lesssim 2^{-2k_4}2^{k_1/2}\prod_{i=1}^{4}\norm{u_i}_{F_{k_i}(T)}.
\end{equation}

If $k' = 0$, we have
\begin{equation}\label{eq:energy1-1.3b}
\left| \int_{[0,T] \times \R}  P_0(\partial_x(u_1u_2))u_3u_4 \;dx dt\right| \lesssim 2^{-2k_4}2^{k_1/2}\prod_{i=1}^{4}\norm{u_i}_{F_{k_i}(T)}.
\end{equation}
\end{lem}

\begin{rem}
In \cite{KP2015}, a weaker estimate \eqref{eq:energy1-1.3} is enough to control the cubic term with one derivative, while, in this paper, \eqref{eq:energy1-1.3c}--\eqref{eq:energy1-1.3b} are necessary to control the cubic terms with three derivatives. On the other hand, under the periodic boundary condition, \eqref{eq:energy1-1.3} is optimal, due to the lack of  \emph{smoothing effect}. We refer to \cite{KP2015} and \cite{Kwak2018} for a part of proof and the periodic case (also for the comparison), respectively.

\end{rem}

\begin{proof}[Proof of Lemma \ref{lem:energy1-1}]
We only prove $(a)$ and $(c)$. The proof of $(b)$ is analogous to the proof of $(c)$, thus we omit it. For part $(b)$, see \cite{KP2015, Kwak2018}.

\medskip

$(a)$ We apply a similar argument as in Remark \ref{rem:longer interval} to the interval $[0,T]$. Let choose $\rho : \R \to [0,1]$ to make a partition of unity, that is, $\sum_{m \in \Z}\rho^4(x-m) \equiv 1$, for all $x \in \R$. It follows that 
\[\left| \int_{[0,T] \times \R}  u_1u_2u_3u_4 \;dx dt\right| \lesssim \sum_{|m| \lesssim 2^{2k_4}} \left| \int_{\R^2}\prod_{i=1}^{4}\left(\rho(2^{2k_4}t-m)\mathbf{1}_{[0,T]}(t)u_i\right) \; dxdt \right|.\]
Set
\[A := \left\{m:\rho(2^{2k_4}t-m)\mathbf{1}_{[0,T]}(t) \mbox{ non-zero and } \neq \rho(2^{2k_4}t-m) \right\}.\]
Note that $|A| \le 4$. We split 
\[\sum_{|m| \lesssim 2^{2k_4}} = \sum_{m \in A} + \sum_{m \in A^c}.\]
It suffices to show $(a)$ on the second summation, since otherwise, the same argument in addition to 
\[\sup_{j \in \Z_+}2^{j/2}\norm{\eta_j(\tau-w(\xi)) \cdot \mathcal{F}[\mathbf{1}_{[0,1]}(t)\rho(2^{2k}t-m)u]}_{L^2} \lesssim \norm{\rho(2^{2k}t-m)\widetilde{u}}_{X_{k}}\]
gives better result, thanks to the absence of $2^{2k_4}$ (see \cite{Guo2012, KP2015} for the details).

\medskip 

On the second summation ($\sum_{m \in A^c}$), we can ignore $\mathbf{1}_{[0,T]}(t)$. Similarly as in Section \ref{sec:nonlinear}, let $u_{k_i} = \mathcal{F}[\rho(2^{2k_4}t-m)\widehat{u}_i(\xi_i)]$ and $u_{k_i,j_i} = \eta_{j_i}(\tau_i - w(\xi_i))u_{k_i}$, $i=1,2,3,4$. Parseval's identity and \eqref{eq:pre2} yield
\begin{equation}\label{eq:energy1-1.5}
\sum_{m \in A^c} \left| \int_{\R^2}\prod_{i=1}^{4}\left(\rho(2^{2k_4}t-m)\mathbf{1}_{[0,T]}(t)u_i\right) \; dxdt \right| \lesssim \sup_{m \in A^c} 2^{2k_4} \sum_{j_1,j_2,j_3,j_4 \ge 2k_4} \left|\int_{\R^2} \prod_{i=1}^4 \mathcal{F}^{-1} [u_{k_i,j_i}] (t,x) \; dxdt \right|.
\end{equation}
H\"older inequality and \eqref{eq:L6} ensure
\[\begin{aligned}
\left|\int_{\R^2} \prod_{i=1}^4 \mathcal{F}^{-1} [u_{k_i,j_i}] (t,x) \; dxdt \right| \lesssim&~{} \norm{u_{k_1,j_1}}_{L^2}\prod_{i=2}^4 \norm{ \mathcal{F}^{-1} [u_{k_i,j_i}]}_{L^6}\\
\lesssim&~{} 2^{-3k_4/2}2^{-j_1/2}\prod_{i=1}^42^{j_i/2}\norm{u_{k_i,j_i}}_{L^2},
\end{aligned}\]
together with \eqref{eq:energy1-1.5}, one concludes \eqref{eq:energy1-1.1} and we complete the proof.

\medskip

$(c)$ The proof of \eqref{eq:energy1-1.3} can be found in \cite{KP2015}. The proof of \eqref{eq:energy1-1.3c} follows the proof of \eqref{eq:energy1-1.3} with a modification $j_{max} \ge 4k_4 + k_2 -10$ instead of $j_{max} \ge 2k_4$. Thus we omit the detail.

Note that 
\begin{equation}\label{c0}
\begin{cases}
2^{k'} \sim 2^{k_2}, \quad &\mbox{if} \quad  k_1 \le k_2 -4,\\
2^{k'} \ll 2^{k_2}, \quad &\mbox{if} \quad |k_1-k_2| \le 4,
\end{cases}
\end{equation}
since $2^{k'} \sim |\xi_1 + \xi_2| \lesssim 2^{k_2}$\footnote{The case $|\xi_1 + \xi_2| \sim 2^{k_2}$, when $|k_1-k_2| \le 4$, exists, if both $\xi_1$ and $\xi_2$ have same sign. However, under this condition, one has the same conclusion as \eqref{eq:energy1-1.3c}.}. 

We first show \eqref{eq:energy1-1.3a}. Similarly as above, it suffices to estimate on $\sum_{m \in A^c}$. Using $2^{j_{max}} \gtrsim 2^{4k_4}2^{k'}$ in addition to \eqref{c0}, one immediately obtains from Lemma \ref{lem:tri-L2} $(b)$ that
\[\mbox{LHS of } \eqref{eq:energy1-1.3a} \lesssim 2^{-2k_4}2^{k_1/2} \sum_{j_1,j_2,j_3,j_4 \ge 2k_4} \prod_{i=1}^42^{j_i/2}\norm{u_{k_i,j_i}}_{L^2} \lesssim 2^{-2k_4}2^{k_1/2}\prod_{i=1}^{4}\norm{u_i}_{F_{k_i}(T)}.\]

The proof of \eqref{eq:energy1-1.3b} is analogous to the proof of Lemma 4.1 (b) (in particular (4.6)) in \cite{GKK2013}. The left-hand side of \eqref{eq:energy1-1.3b} can be replaced by
\[\sum_{\ell \le 0}\left| \int_{[0,T] \times \R}  P_{\ell}(u_1u_2)u_3u_4 \;dx dt\right|.\]
If $k_2 =0$, similarly as the proof of \eqref{eq:energy1-1.3a}, we have
\[\mbox{LHS of } \eqref{eq:energy1-1.3b} \lesssim 2^{-2k_4}\prod_{i=1}^{4}\norm{u_i}_{F_{k_i}(T)}.\]
Otherwise ($k_1 \ge1$)\footnote{The case $|\xi_1 + \xi_2| \le 1$ cannot happen when $k_1 =0$ and $k_2 \ge 1$.}, the same argument in the proof of \eqref{eq:energy1-1.3a} yields
\[\mbox{LHS of } \eqref{eq:energy1-1.3b} \lesssim 2^{-2k_4}2^{k_1/2}\prod_{i=1}^{4}\norm{u_i}_{F_{k_i}(T)}.\]
Thus, we complete the proof.
\end{proof}

\begin{rem}\label{rem:weight effect}
Using the weight, one can have at least $2^{k_4/4}$ more derivative gain, while $2^{\frac58k_1}$ derivative loss occurs. Indeed, a direct computation gives
\[\sum_{j_1 \ge 2k_4} 2^{j_1/2}\norm{u_{k_1,j_1}}_{L^2} \lesssim 2^{-\frac{2k_4-5k_1}{8}}\norm{u_1}_{F_{k_1}(T)}.\]
Such derivative gain may be helpful to avoid the occurrence of the logarithmic divergence in $H^2$-energy estimates (see \cite{Kwak2018}).  Moreover, the derivative loss in low frequencies is not big in $H^2$, so be handled in $H^2$. This approach may be applied to LWP of the fifth-order mKdV $H^2(\T)$ (improvement of \cite{Kwak2018}, in authors' forthcoming project). 
\end{rem}

\medskip

\begin{prop}\label{prop:energy}
Let $s \ge 2$ and $T \in (0,1]$. Then, for the solution $u \in C([-T,T];H^{\infty}(\T))$ to \eqref{5G_Gen}, we have
\begin{equation}\label{energy0}
\norm{u}_{E^s(T)}^2 \lesssim \norm{u_0}_{H^s}^2 + \sum_{j=3}^{6} \norm{u}_{F^s(T)}^j.
\end{equation}
\end{prop}

\begin{proof}[Proof of Proposition \ref{prop:energy}]
The definition of the $E^s(T)$ norm says
\[\norm{u}_{E^s(T)}^2 -\norm{P_{\le 0}(u_0)}_{L^2}^2 = \sum_{k \ge 1}\sup_{t_k \in [-T,T]}2^{2sk}\norm{P_k(u(t_k))}_{L^2}^2.\]
Then, we immediately have
\[\begin{aligned}
2^{2sk}\norm{P_k(u(t_k))}_{L^2}^2 - 2^{2sk}\norm{P_k(u_0)}_{L^2}^2  \lesssim&~{} 2^{2sk}\left|\int_{\R \times [0,t_k]}P_k(u)P_k(\mathcal N_2(u)) \; dxdt\right|\\
&+2^{2sk}\left|\int_{\R \times [0,t_k]}P_k(u)P_k(\mathcal N_3(u)) \; dxdt\right|\\
&+2^{2sk}\left|\int_{\R \times [0,t_k]}P_k(u)P_k(\mathcal{SN}(u)) \; dxdt\right|\\
&=: I_1(k)+I_2(k)+I_3(k),
\end{aligned}\]
thanks to \eqref{eq:energy 1}. Proposition 4.2 (in addition to Remark 4.3) in \cite{GKK2013} yields
\[\sum_{k \ge 1} I_1(k) \lesssim \norm{u}_{F^s(T)}^3 + \norm{u}_{F^s(T)}^4,\]
for $s \ge \frac54$.

\medskip

We now focus on $I_2(k)$. Note that a direct calculation gives
\[40 uu_x u_{xx} + 10 u^2u_{xxx} + 10 u_x^3 = 10(u^2)_xu_{xx} + 10 (u^2u_{xx})_x + 10 u_x^3.\]
We split $I_2(k)$ (in particular cubic part in $\mathcal N_3(u)$) into $I_{2,1} + I_{2,2} +I_{2,3}$, where
\[I_{2,1} := 2^{2sk}\left|\int_{\R \times [0,t_k]}P_k(u)P_k((u^2)_xu_{xx}) \; dxdt\right|,\]
\[I_{2,2} := 2^{2sk}\left|\int_{\R \times [0,t_k]}P_k(u)P_k((u^2u_{xx})_x) \; dxdt\right|\]
and
\[I_{2,3} := 2^{2sk}\left|\int_{\R \times [0,t_k]}P_k(u)P_k(u_x^3) \; dxdt\right|.\]

We first estimate $I_{2,1}$. We further decompose $I_{2,1}$ as follows:
\[\begin{aligned}
I_{2,1} \lesssim& 2^{2sk} \sum_{k' \le k-10}\left|\int_{\R \times [0,t_k]}P_k(u)P_k(P_{k'}(u^2)_xu_{xx}) \; dxdt\right|\\
&+2^{2sk} \sum_{k' > k-10, k_3 \ge 0}\left|\int_{\R \times [0,t_k]}P_k^2(u)P_{k'}((u^2)_x)P_{k_3}(u_{xx}) \; dxdt\right|\\
=:&~{}I_{2,1,1}+I_{2,1,2}.
\end{aligned}\]
Note that $k' \le k_2 +10$ in $I_{2,1,1}$. Lemma \ref{lem:energy1-1} yields
\[\begin{aligned}
I_{2,1,1} \lesssim&~{} 2^{2sk}\sum_{k_1\le k_2 \le k-10}2^{k_1/2}2^{k_2}\norm{P_{k_1}u}_{F_{k_1}(T)}\norm{P_{k_2}u}_{F_{k_2}(T)}\sum_{|k_0-k|\le 10} \norm{P_{k_0}u}_{F_{k_0}(T)}^2\\
\lesssim&~{} 2^{2sk}\sum_{|k_0-k|\le 10}2^{\frac52k_0}\norm{P_{k_0}u}_{F_{k_0}(T)}^4\\
\lesssim&~{} 2^{2sk}\sum_{k_1 \ge k+10}2^{\frac32k_1}\norm{P_{k_1}u}_{F_{k_1}(T)}^2\sum_{|k_0-k|\le 10} \norm{P_{k_0}u}_{F_{k_0}(T)}^2\\
\lesssim&~{} 2^{2sk}\norm{P_{k}u}_{F_{k}(T)}^2\norm{u}_{F^{\frac54}(T)}^2
\end{aligned}\]

We divide the summation over $k_3$ in $I_{2,1,2}$ by $\sum_{k_3 \le k-10} + \sum_{|k_3 - k| \le 10} + \sum_{k_3 \ge k+10}$. Then, by the support property, we know that the integral vanishes unless $|k'-k| \le 10$ on the first and second summations and $k' \ge k+10$ on the last summation. Note on the last summation that $|k'-k_3| \le 10$.  

\medskip

On the first summation, the following cases of $k_1$ and $k_2$ (assuming $k_1 \le k_2$ by the symmetry) are possible:
\begin{enumerate}
\item $k_1 \le k-10$ and $|k_2 - k| \le 10$
\item $|k_1-k_2| \le 10$ and $k_2 \ge k+ 10$
\item $|k_1-k_2|\le10$ and $|k_2 - k| \le 10$
\end{enumerate} 

It suffices to assume in the first case that $k_1 \le k_3$, since two derivatives are taken in the $k_3$-frequency mode. We use \eqref{eq:energy1-1.3c} and \eqref{eq:energy1-1.3} with the use of the weight (see Remark \ref{rem:weight effect}) when $k_1 = 0$ and $k_1 \ge 1$, respectively, to estimate $I_{2,1,2}$, precisely, $I_{2,1,2}$ is bounded by
\[\begin{aligned}
&2^{2sk}\norm{P_0u}_{F_0(T)}\sum_{k_3 \le k-10}2^{k_3}\norm{P_{k_3}u}_{F_{k_3}(T)}\sum_{|k_2-k|\le 10}\norm{P_{k_2}u}_{F_{k_2}(T)}^2\\
&+2^{2sk}\sum_{1 \le k_1 \le k_3 \le k-10}2^{\frac32k_1}2^{2k_3}2^{-\frac18k}\norm{P_{k_1}u}_{F_{k_1}(T)}\norm{P_{k_3}u}_{F_{k_3}(T)}\sum_{|k_2-k|\le 10}\norm{P_{k_2}u}_{F_{k_2}(T)}^2\\
\lesssim&~{} 2^{2sk}\norm{P_{k}u}_{F_{k}(T)}^2\norm{u}_{F^2(T)}^2 
\end{aligned}\]
Under the second case, by \eqref{eq:energy1-1.3a}, $I_{2,1,2}$ is bounded by
\[\begin{aligned}
&2^{2sk}\sum_{\substack{k_3 \le k-10\\|k_1-k_2|\le 10\\k_2 \ge k+10}}2^{\frac52k_3}2^{k}2^{-2k_2}\prod_{j=1}^3\norm{P_{k_j}u}_{F_{k_j}(T)}\norm{P_{k}u}_{F_{k}(T)}\\
\lesssim&~{} \max(2^{\frac32k},2^{(s-1)k})\norm{P_{k}u}_{F_{k}(T)}\norm{u}_{F^{s}(T)}^3, 
\end{aligned}\]
for $s \ge 0$.

Under the last case, by \eqref{eq:energy1-1.2}, $I_{2,1,2}$ is bounded by
\[\begin{aligned}
&2^{2sk}\sum_{\substack{k_3 \le k-10\\|k_1-k_2|\le 10\\|k_2-k| \le 10}}2^{\frac52k_3}\prod_{j=1}^3\norm{P_{k_j}u}_{F_{k_j}(T)}\norm{P_{k}u}_{F_{k}(T)}\\
\lesssim&~{} 2^{2sk}\norm{P_{k}u}_{F_{k}(T)}^2\norm{u}_{F^{\frac54}(T)}^2. 
\end{aligned}\]

\medskip

On the second summation, the possible cases of $k_1$ and $k_2$ are same as before. Using \eqref{eq:energy1-1.2}, \eqref{eq:energy1-1.3a} and \eqref{eq:energy1-1.1}, one concludes that $I_{2,1,2}$ is bounded by
\[\begin{aligned}
&2^{2sk}\sum_{\substack{k_1 \le k-10\\|k_2-k_3|\le 10\\|k_2-k| \le 10}}2^{2k}2^{k_1/2}\prod_{j=1}^3\norm{P_{k_j}u}_{F_{k_j}(T)}\norm{P_{k}u}_{F_{k}(T)}\\
&+2^{2sk}\sum_{\substack{|k_3 - k| \le 10\\|k_1-k_2|\le 10\\k_2 \ge k+ 10}}2^{\frac32k_2}\prod_{j=1}^3\norm{P_{k_j}u}_{F_{k_j}(T)}\norm{P_{k}u}_{F_{k}(T)}\\
&+2^{2sk}\sum_{\substack{|k_1-k_2|\le 10\\ |k_3 - k| \le 10\\|k_2-k| \le 10}}2^{\frac52k}\prod_{j=1}^3\norm{P_{k_j}u}_{F_{k_j}(T)}\norm{P_{k}u}_{F_{k}(T)}\\
\lesssim&~{} 2^{2sk}\norm{P_{k}u}_{F_{k}(T)}^2\norm{u}_{F^{2}(T)}^2. 
\end{aligned}\]

\medskip

On the last summation, the following cases of $k_1$ and $k_2$ (assuming $k_1 \le k_2$ by the symmetry) are possible:
\begin{enumerate}
\item $k_1 \le k_2-10$ and $|k_2 - k_3| \le 10$
\item $|k_1-k_2| \le 10$ and $k_2 \ge k_3 10$
\item $|k_1-k_2|\le10$ and $|k_2 - k_3| \le 10$
\end{enumerate} 
Since $k$-frequency is the lowest frequency, hence one similarly or easily has
\[\begin{aligned}
&2^{2sk}\sum_{\substack{k_1 \le k_2-10\\|k_2-k_3|\le 10\\k_3 \ge k+10}}2^{2k_3}2^{\frac32k_1}2^{-k_3/8}\prod_{j=1}^3\norm{P_{k_j}u}_{F_{k_j}(T)}\norm{P_{k}u}_{F_{k}(T)}\\
&+2^{2sk}\sum_{\substack{k_2 \ge k_3+10\\|k_1-k_2|\le 10\\k_3 \ge k+10}}2^{k_3}2^{k/2}\prod_{j=1}^3\norm{P_{k_j}u}_{F_{k_j}(T)}\norm{P_{k}u}_{F_{k}(T)}\\
&+2^{2sk}\sum_{\substack{|k_1-k_2|\le 10\\ |k_2 - k_3| \le 10\\k_3 \ge k+10}}2^{2k_3}2^{k/2}\prod_{j=1}^3\norm{P_{k_j}u}_{F_{k_j}(T)}\norm{P_{k}u}_{F_{k}(T)}\\
\lesssim&~{} \max\left(2^{\frac{15}{8}},2^{(\frac{27}{8}-s)k} \right)\norm{P_{k}u}_{F_{k}(T)}\norm{u}_{F^s(T)}^3,
\end{aligned}\]
for $s \ge \frac98$, thanks to Lemma \ref{lem:energy1-1} (c) and (b).

\medskip

The estimate of $I_{2,2}$ is very similar as before. In view of the estimate of $I_{2,1}$, one knows that the worst case appears when the frequency support of $u_{xx}$ is $I_k$. However, a direct calculation (integration by parts) gives
\[\begin{aligned}
\left|\int_{\R \times [0,t_k]}P_k(u)(P_{k'}(u^2)P_k(u_{xx}))_x \; dxdt\right| =&~{} \left|\int_{\R \times [0,t_k]}(P_k(u))_x(P_{k'}(u^2)P_k(u_{xx})) \; dxdt\right|\\
=&~{} \frac12\left|\int_{\R \times [0,t_k]}((P_k(u))_x)^2(P_{k'}(u^2))_x \; dxdt\right|,
\end{aligned}\]
which is exactly same as $I_{2,1}$ (in particular, $I_{2,1,1}$ and $I_{2,1,2}$ under $|k_3-k| \le 10$). The rigorous justification of this observation can be seen in the commutator estimates, see the proof of Lemma 4.1 (c) in \cite{GKK2013} for the details or see the proof of Proposition \ref{prop:energy diff} below. Moreover, one can see that the derivatives are fairly distributed in $I_{2,3}$, and hence it can be easily or similarly controlled as the estimate of $I_{2,1}$. We omit the details.

\medskip
On the other hand, the rest part,
\[2^{2sk}\left|\int_{\R \times [0,t_k]}P_k(u)P_k(F(u)) \; dxdt\right|,\]
where $F(u) = (u^p)_x$, $p=2,4,5$, can be immediately handled by using 
\[\left|\int_{\R \times [0,t_k]}\prod_{j=1}^{p+1} u_j \; dxdt \right|\lesssim 2^{(k_1+\cdots+k_{p-1})/2}\prod_{j=1}^{p+1}\norm{u_j}_{F_{k_j}(T)},\]
where $u_j = P_{k_j}u \in F_{k_j}(T)$, $j=1,\cdots,p+1$ and assuming that $k_1 \le \cdots \le k_{p+1}$, for $p=2,4,5$. 

\medskip

Collecting all, we have 
\[\sum_{k \ge 1} (I_2(k) + I_3(k)) \lesssim \norm{u}_{F^s(T)}^2+\norm{u}_{F^s(T)}^4+\norm{u}_{F^s(T)}^5 + \norm{u}_{F^s(T)}^6,\]
for $s\ge2$, thus we complete the proof of \eqref{energy0}.
\end{proof}

\bigskip

Let $u_1$ and $u_2$ be solutions to \eqref{5G_Gen}. Define $v= u_1-u_2$, then $v$ solves
\begin{equation}\label{5G_Gen_Diff}
v_t + v_{5x} + \mathcal N_2(u_1,u_2) + \mathcal N_3(u_1,u_2) + \mathcal{SN}(u_1,u_2) = 0, \qquad v(0,x) = u_1(0,x) - u_2(0,x),
\end{equation}
where 
\[\mathcal N_2(u_1,u_2)  =  \mathcal N_2(u_1)  -  \mathcal N_2(u_2), \quad \mathcal N_3(u_1,u_2)  =  \mathcal N_3(u_1)  -  \mathcal N_3(u_2) \quad \mbox{and} \quad \mathcal{SN}(u_1,u_2)  =  \mathcal{SN}(u_1)  -  \mathcal{SN}(u_2).\]

\begin{prop}\label{prop:energy diff}
Let $s \ge 2$ and $T \in (0,1]$. Then, for solutions $v \in C([-T,T];H^{\infty}(\T))$ to \eqref{5G_Gen_Diff} and $u_1, u_2 \in C([-T,T];H^{\infty}(\T))$ to \eqref{5G_Gen}, we have
\[\norm{v}_{E^0(T)}^2 \lesssim \norm{v_0}_{L^2}^2 + \left(\sum_{j=1}^2 \left(\norm{u_j}_{F^s(T)}+\norm{u_j}_{F^s(T)}^2 + \norm{u_j}_{F^s(T)}^3 + \norm{u_j}_{F^s(T)}^4 \right) \right) \norm{v}_{F^0(T)}^2,\]
and
\[\begin{aligned}
\norm{v}_{E^s(T)}^2 \lesssim&~{} \norm{v_0}_{H^s}^2 + \left(\sum_{j=1}^2 \left(\norm{u_j}_{F^s(T)}+\norm{u_j}_{F^s(T)}^2 + \norm{u_j}_{F^s(T)}^3 + \norm{u_j}_{F^s(T)}^4 \right) \right) \norm{v}_{F^s(T)}^2\\
&+\left(\norm{u_1}_{F^{2s}(T)}+\norm{u_2}_{F^{2s}(T)}\right)\norm{v}_{F^0(T)}\norm{v}_{F^s(T)}.
\end{aligned}\]
\end{prop}

\begin{rem}
One can see that the cubic terms with three derivatives are not harmful even in $F^s$, while the same terms are the main enemy under the periodic boundary condition. The principal reason is due to the lack of the smoothing effect under the periodic condition. Compare Lemma \ref{lem:energy1-1} (c) and  Lemma 6.4 (c) and (d) in \cite{Kwak2018}.
\end{rem}

\begin{proof}
We first concentrate on the estimate on $\norm{v}_{E^0(T)}^2$. From the definition of $\norm{v}_{E^s(T)}^2$ and \eqref{eq:energy 1}, it suffices to control
\[\begin{aligned}
2^{2sk}\norm{P_k(v(t_k))}_{L^2}^2 - 2^{2sk}\norm{P_k(v_0)}_{L^2}^2  \lesssim&~{} 2^{2sk}\left|\int_{\R \times [0,t_k]}P_k(v)P_k(\mathcal N_2(u_1,u_2)) \; dxdt\right|\\
&+2^{2sk}\left|\int_{\R \times [0,t_k]}P_k(v)P_k(\mathcal N_3(u_1,u_2)) \; dxdt\right|\\
&+2^{2sk}\left|\int_{\R \times [0,t_k]}P_k(v)P_k(\mathcal{SN}(u_1,u_2)) \; dxdt\right|\\
&=: 2^{2sk}\wt{I}_1(k)+2^{2sk}\wt{I}_2(k)+2^{2sk}\wt{I}_3(k).
\end{aligned}\]

Proposition 4.4 in \cite{GKK2013} yields
\[\sum_{k \ge 1} \wt{I}(k) \lesssim \left(\sum_{j=1}^2 (\norm{u_j}_{F^s(T)}+\norm{u_j}_{F^s(T)}^2) \right) \norm{v}_{F^0(T)}^2,\]
and 
\[\sum_{k \ge 1} 2^{2sk}\wt{I}(k) \lesssim \left(\sum_{j=1}^2 (\norm{u_j}_{F^s(T)}+\norm{u_j}_{F^s(T)}^2) \right) \norm{v}_{F^s(T)}^2 + \left(\norm{u_1}_{F^{2s}(T)}+\norm{u_2}_{F^{2s}(T)} \right)\norm{v}_{F^0}\norm{v}_{F^s(T)},\]
for $s \ge 2$. Moreover, since the quintic term in $\mathcal N_3(u_1,u_2)$ and $\mathcal{SN}(u_1,u_2)$ contains only one derivative, one can easily handle them compared to the cubic term in $\mathcal N_3(u_1,u_2)$. Thus, in what follows, we only focus on $\wt{I}_2(k)$ (in particular, the cubic terms), similarly as the proof of Proposition \ref{prop:energy}.

\medskip

We write $\wt{I}_2(k) = \wt{I}_{2,1} - \wt{I}_{2,2} + \wt{I}_{2,3}$, where
\[\wt{I}_{2,1} := \frac12\left|\int_{\R \times [0,t_k]}P_k(v)P_k((u_1^2 + u_2^2)_xv_{xx} + (v(u_1+u_2))_x(u_1+u_2)_{xx}) \; dxdt\right|,\]
\[\wt{I}_{2,2} := \frac12\left|\int_{\R \times [0,t_k]}P_k(v_x)P_k((u_1^2 + u_2^2)v_{xx} + v(u_1+u_2)(u_1+u_2)_{xx}) \; dxdt\right| (=: \wt{I}_{2,2,1} + \wt{I}_{2,2,2})\]
and
\[\wt{I}_{2,3} := \left|\int_{\R \times [0,t_k]}P_k(v)P_k(v_x(u_{1,x}^2+u_{1,x}u_{2,x}+u_{2,x}^2)) \; dxdt\right|.\]

We, here, only consider $\wt{I}_{2,2}$ in order to provide a rigorous proof of the estimate of $I_{2,2}$ in the proof of Proposition \ref{prop:energy}. Moreover, it is easier to handle $\wt{I}_{2,1,2}$ than $\wt{I}_{2,1,1}$ (or similar), since less derivatives are taken in $v$, hence it is enough to estimate only $\wt{I}_{2,1,1}$. We reduce $\wt{I}_{2,2,1}$ as
\[\left|\int_{\R \times [0,t_k]}P_k(v_x)P_k(u^2v_{xx}) \; dxdt\right|.\]
A direct calculation gives
\[\begin{aligned}
\left|\int_{\R \times [0,t_k]}P_k(v_x)P_k(u^2v_{xx}) \; dxdt\right| \lesssim&~{} \sum_{k' \le k-10}\left|\int_{\R \times [0,t_k]}P_k(v_x)P_k(P_{k'}(u^2)v_{xx}) \; dxdt\right|\\
&+\sum_{k' \ge k- 9, k_3 \ge 0}\left|\int_{\R \times [0,t_k]}P_k^2(v_x)P_{k'}(u^2)P_{k_3}(v_{xx}) \; dxdt\right|.
\end{aligned}\]
Since
\[\begin{aligned}
P_k(v_x)P_k(P_{k'}(u^2)v_{xx})=&  P_k(v_x)P_k(v_{xx})P_{k'}(u^2) + P_k(v_x)[P_k, P_{k'}(u^2)]v_{xx}\\
=& \frac12 ((P_k(v_x))^2)_xP_{k'}(u^2) +  P_k(v_x)[P_k, P_{k'}(u^2)]v_{xx},\\
\end{aligned}\]
where $[A,B] = AB - BA$, the integration by parts yields
\[\sum_{k' \le k- 10}\left|\int_{\R \times [0,t_k]}((P_k(v_x))^2)_xP_{k'}(u^2) \; dxdt\right| = \sum_{k' \le k- 10}\left|\int_{\R \times [0,t_k]}((P_k(v_x))^2)P_{k'}((u^2)_x) \; dxdt\right|,\]
which is already dealt with in the proof of Proposition \ref{prop:energy} (in particular, $I_{2,1,1}$). Thus, we have
\[\sum_{k \ge 1} \sum_{k' \le k- 10}\left|\int_{\R \times [0,t_k]}((P_k(v_x))^2)P_{k'}((u^2)_x) \; dxdt\right| \lesssim \norm{u}_{F^s(T)}^2\norm{v}_{F^0(T)}^2\]
and
\[\sum_{k \ge 1} 2^{2sk} \sum_{k' \le k- 10}\left|\int_{\R \times [0,t_k]}((P_k(v_x))^2)P_{k'}((u^2)_x) \; dxdt\right| \lesssim \norm{u}_{F^s(T)}^2\norm{v}_{F^s(T)}^2,\]
for $s \ge 2$.

On the other hand, a direct computation, in addition to the mean value theorem, \eqref{chi} and \eqref{eq:regularity}, gives
\[\ft\left([\widetilde{P}_k,\widetilde{P}_{k'}(u^2)](v_{xx})\right)(\tau,\xi) =C\int_{\R^2}\ft(\widetilde{P}_{k'}((u^2)_x))(\tau', \xi') \cdot \ft(v_x)(\tau-\tau',\xi-\xi')\cdot m(\xi,\xi')\;d\xi'd\tau',\]
where, 
\[|m(\xi,\xi')|=\left|\frac{(\xi-\xi')(\chi_k(\xi)-\chi_k(\xi-\xi'))}{\xi'}\right| \lesssim |(\xi-\xi_1)\chi_k'(\xi-\theta\xi_1)|\lesssim \sum_{|k-k'|\le 4}\chi_{k'}(\xi-\xi_1),\]
for $0 \le \theta \le 1$. Thus, an analogous argument yields
\[\sum_{k \ge 1} \sum_{k' \le k- 10}\left|\int_{\R \times [0,t_k]}P_k(v_x)[P_k, P_{k'}(u^2)]v_{xx} \; dxdt\right| \lesssim \norm{u}_{F^s(T)}^2\norm{v}_{F^0(T)}^2\]
and
\[\sum_{k \ge 1} 2^{2sk} \sum_{k' \le k- 10}\left|\int_{\R \times [0,t_k]}P_k(v_x)[P_k, P_{k'}(u^2)]v_{xx} \; dxdt\right| \lesssim \norm{u}_{F^s(T)}^2\norm{v}_{F^s(T)}^2,\]
for $s \ge 2$.

\medskip

The rest of the proof, which is the estimate of 
\begin{equation}\label{rest}
\sum_{k' \ge k- 9, k_3 \ge 0}\left|\int_{\R \times [0,t_k]}P_k^2(v_x)P_{k'}(u^2)P_{k_3}(v_{xx}) \; dxdt\right|,
\end{equation}
is almost identical to the proof of  the estimate of $I_{2,1,2}$ in the proof of Proposition \ref{prop:energy}. Thus, we have
\[\sum_{k\ge 1} \eqref{rest} \lesssim \norm{u}_{F^s(T)}^2\norm{v}_{F^0(T)}^2\]
and
\[\sum_{k\ge 1} 2^{2sk} \eqref{rest} \lesssim \norm{u}_{F^s(T)}^2\norm{v}_{F^s(T)}^2,\]
for $s \ge 2$. Thus, we complete the proof.
\end{proof}

\subsection{Local and Global well-posedness}\label{sec:LWP}
The local-well-posedness argument (the classical energy method) is now standard. We refer the readers to \cite{IKT2008, GPWW2011, GKK2013, KP2015, Kwak2018} and references therein, for more details, and we, here, give a sketch of proof.

\medskip

We first state fundamental properties of $X^{s,b}$-type norms.

\begin{prop}\label{prop:H^s}
Let $s\ge 0$, $T\in (0,1]$, and $u\in F^s(T)$, then
\[\sup_{t\in [-T,T]} \|u(t)\|_{H^s} \lesssim \|u\|_{F^s(T)}\]
\end{prop}

\begin{prop}\label{prop:linear}
Let $T\in (0,1]$,  $u,v\in C([-T,T]:H^\infty)$ and
\[\partial_tu + \px^5u  = v  \ on \  \R\times(-T,T)\]
Then we have
\[\|u\|_{F^s(T)} \lesssim \|u\|_{E^s(T)}  + \|v\|_{N^s(T)},\]
for any $s \ge 0$.
\end{prop}

See Appendix in \cite{GKK2013} for the proofs. We also refer to \cite{IKT2008, KP2015, GPWW2011}.
\medskip

One can observe that \eqref{5G} admits the scaling equivalence with \eqref{5G_Gen}: For $\lambda > 0$, if $u$ is a solution to \eqref{5G}, then $u_{\lambda}$, defined by
\[u_{\lambda}(t,x) := \lambda u(\lambda^5 t,\lambda x),\]
is a solution to \eqref{5G_Gen}. Moreover, a direct calculation yields $\norm{u_{0,\lambda}}_{\dot{H}^s} = \lambda^{s+\frac12}\norm{u_0}_{\dot{H}^s}$, which says the scaling exponent $s_c = -\frac12$. Thus, a small data local well-posedness of \eqref{5G_Gen} ensures the local-in-time well-posedness of \eqref{5G} for an arbitrary data. 

\medskip

From Duhamel's principle, we know that the solution to \eqref{5G_Gen} is of the following integral form:
\[u(t) = W(t)u_0 + \int_{0}^{t} W(t-s)  \left(\mathcal N_2(u)(s) + \mathcal N_3(u)(s) + \mathcal{SN}(u)(s) \right) ds,\]
where $W(t)$ is defined as in \eqref{Lin. Sol.}. We assume
\begin{equation}\label{eq:assumption}
\norm{u_0}_{H^s} \le \epsilon \ll 1.
\end{equation}
Remark that for any fixed $\mu \in \R^+$, we can choose $0 < \lambda_0 \ll 1$ sufficiently small such that the initial data satisfy \eqref{eq:assumption} and $\mu\lambda \le 1$ for all $\lambda \le \lambda_0$. The second condition ensures that our local well-posedness argument does not depend on $\mu$.

\medskip

We fix $s \ge 2$. Proposition \ref{prop:linear}, Propositions \ref{prop:bilinear}, \ref{prop:trilinear} and \ref{prop:multilinear}, and Proposition \ref{prop:energy} ensures
\[\left \{
\begin{array}{l}
\norm{u}_{F^{s}(T')}\lesssim \norm{u}_{E^s(T')} + \norm{\mathcal N_2(u) + \mathcal N_3(u) + \mathcal{SN}(u)}_{N^s(T')};\\
\norm{\mathcal N_2(u) + \mathcal N_3(u) + \mathcal{SN}(u)}_{N^s(T')} \lesssim \sum_{j=2}^{5} \norm{u}_{F^s(T)}^j;\\
\norm{u}_{E^s(T')}^2 \lesssim \norm{u_0}_{H^s}^2 + \sum_{j=3}^{6} \norm{u}_{F^s(T)}^j,
\end{array}
\right.\]
for any $T' \in [0,T]$, which, in addition to the smallness condition \eqref{eq:assumption} and continuity argument (see Lemma 6.3 in \cite{KP2015} for the details), implies \emph{a priori} bound:
\begin{equation}\label{eq:result}
\sup_{t \in [-T,T]}\norm{u(t)}_{H^s} \lesssim \norm{u(t)}_{F^s(T)} \lesssim \norm{u_0}_{H^s}.
\end{equation}

To complete the proof, we need
\begin{prop}\label{prop:energy diff-1}
Assume $s \ge 2$. Let $u_1, u_2 \in F^{s}(T)$ be solutions to \eqref{5G_Gen} with small initial data $u_{1,0},v_{2,0} \in H^{\infty}$. Let $v = u_1 -u_2$ and $v_0 = u_{1,0} - u_{2,0}$. Then we have
\[\norm{v}_{F^0(T)} \lesssim \norm{v_0}_{L^2}\]
and
\[\norm{v}_{F^s(T)} \lesssim \norm{v_0}_{H^s} + \norm{u_{1,0}}_{H^{2s}}\norm{v_0}_{L^2}.\]
\end{prop}
It immediately follows from Proposition \ref{prop:linear}, Propositions \ref{prop:bilinear}, \ref{prop:trilinear} and \ref{prop:multilinear}, and Proposition \ref{prop:energy diff} under \eqref{eq:result}.

\medskip

For fixed $u_0 \in H^s$, a density argument enables us to choose a sequence $\{u_{0,n}\}_{n=1}^{\infty} \subset H^{\infty}$ such that $u_{0,n} \to u_0$ in $H^s$ as $n \to \infty$. Let $u_n(t) \in H^{\infty}$ is a solution to \eqref{5G_Gen} with initial data $u_{0,n}$. Using a similar argument as above and Proposition \ref{prop:energy diff-1}, one shows $\{u_n\}$ is a Cauchy sequence. Indeed, for $K \in \Z_+$, let $u_{0,n}^K = P_{\le K}(u_{0,n})$. Then, $u_n^K = P_{\le K}u_n$ satisfies the frequency localized equation ($P_{\le K} \eqref{5G_Gen}$) with the initial data $u_{0,n}^K$. We have from the triangle inequality that 
\[\sup_{t \in [-T,T]}\norm{u_m-u_n}_{H^s} \lesssim \sup_{t \in [-T,T]}\norm{u_m-u_m^K}_{H^s} +\sup_{t \in [-T,T]}\norm{u_m^K-u_n^K}_{H^s} +\sup_{t \in [-T,T]}\norm{u_n^K-u_n}_{H^s}.\]
The first and last terms are bounded by $\epsilon$, thanks to \emph{a priori} bound, and the second term is bounded by $\epsilon$, thanks to Proposition \ref{prop:energy diff}, precisely,
\[\begin{aligned}
\sup_{t \in [-T,T]}\norm{v_m^K - v_n^K}_{H^s(\T)} \lesssim&~{} \norm{v_{0,m}^K - v_{0,n}^K}_{H^s} +K^s\norm{v_{0,m}^K-v_{0,n}^K}_{L^2}\\
\lesssim&~{} \epsilon.
\end{aligned}\]
Hence the $3\epsilon$-argument completes the proof and we obtain a solution as the limit. The uniqueness and the continuity of dependence follow from an analogous argument. 

\begin{rem}
In view of all analyses above, we do not use the integrability of the Gardner equation \eqref{5G_Gen} to prove the local well-posedness, and thus we can apply our argument to prove the local result of \eqref{5G_Gen} with arbitrary coefficients.
\end{rem}

Small solutions $u$ to \eqref{5G_Gen} satisfies the (rescaled) conservation laws \eqref{M1cor}-\eqref{E1}-\eqref{E5}, namely
\[M[u](t)  :=  \frac 12 \int_\R u^2(t,x)dx = M[u](0),\]
\[E_\mu[u](t)  :=  \int_\R \left(\frac 12 u_x^2 -2\mu\lambda u^3  - \frac 12 u^4\right)(t,x)dx = E_{\mu}[u](0),\]
and
\begin{equation}\label{E5-1}
E_{5\mu}[u](t)  :=  \int_\R \left(\frac 12 u_{xx}^2 -10\mu\lambda uu_x^2 + 10 \mu^2\lambda^2 u^4 - 5u^2u_x^2 + 6\mu\lambda u^5 + u^6\right)(t,x)dx= E_{5\mu}[u](0).
\end{equation}
Using above conserved quantities and the Sobolev embedding in addition to the smallness condition, one proves Theorem \ref{GWP1}.

\section{Stability of breathers of the 5th order Gardner equation}\label{sect5}

Once we have shown the existence of global solutions of the Cauchy problem for the 5th order Gardner equation \eqref{5G}, we study now 
the stability properties of a special solution of \eqref{5G}. Before dealing with this stability result, we present basic facts on solutions of \eqref{5G}. 
The simplest solution of the 5th order \emph{focusing} Gardner equation is a traveling wave like  solution, usually called as soliton solution, and
explicitly defined as follows

\begin{defn}\label{5solG}
The   1-soliton solution $Q_{\mu}\equiv Q_{\mu,c}$  of the 5th order \emph{focusing} Gardner equation  \eqref{5G} 
is given by

\begin{equation}\label{examplesolGs} 
\begin{aligned}
&Q_{\mu}(t,x) := Q_{\mu,c} (x-v_{\mu,c}t+x_1),\quad Q_{\mu,c} (z) :=\frac{c}{2\mu + \sqrt{4\mu^2+c}\cosh(\sqrt{c}z)}, \quad c>0,~\mu,~x_1\in\R,\\
&\text{with}~~\quad v_{\mu,c}:= c^2 + 10\mu^2c,
\end{aligned}
\end{equation}
\end{defn}
\noindent
which indeed has a completely similar profile to the well known  Gardner soliton profile \cite{Ale11}. The 1-soliton solution $Q_{\mu,c}$ \eqref{examplesolGs} 
of the 5th order Gardner equation \eqref{5G} satisfies the nonlinear second order ODE:

\medskip

\[\begin{array}{ll}
Q_\mu'' -c\, Q_\mu + 6\mu Q_\mu^2 + 2Q_\mu^3=0, \quad Q_\mu>0, \quad Q_{\mu}\in H^1(\R).\\
\end{array}\]
Note that, as a solution of \eqref{5G}, $Q_\mu$ also satisfies naturally the fourth order ODE

\[\begin{aligned}
&Q_\mu^{''''} -v_{\mu,c}\, Q_\mu + \tilde{f}_5(Q_\mu)=0,\quad \text{with}\\
&\tilde{f}_5(Q_\mu)=10 (\mu +Q_\mu)^2Q_{\mu}^{''} +10(\mu +Q_\mu)(Q_{\mu}^{'})^2+60\mu^3Q_\mu^2+60\mu^2Q_\mu^3+30\mu Q_\mu^4 +6Q_\mu^5.\\
\end{aligned}\]

\medskip

Moreover, it is possible to build a solution made of the composition of N of such solitons, which is usually called as the \emph{N-soliton solution} for the 5th order Gardner equation 

\[\begin{aligned}
&u_\mu(t,x):=i\partial_x\log(g_\mu/f_\mu), \quad \mbox{where}\\
&f_\mu(t,x):= \sum_{\si=0,1}\exp\Big[\sum_{k=1}^N\si_k\displaystyle\left(\theta_k  + \rho_k\right) + \sum_{k<m}^N \si_k\si_m A_{km}\Big],\quad \mbox{and} \quad g_\mu(t,x):=f_\mu^*(t,x),\\
\end{aligned}
\]
with
\[\begin{aligned}
&\exp[\rho_k]:=\frac{i\sqrt{c_k}+2\mu}{\sqrt{c_k}},\quad \exp[A_{km}]=\Big(\frac{\sqrt{c_k}-\sqrt{c_m}}{\sqrt{c_k}+\sqrt{c_m}}\Big)^2,~~k,m=1,2,\dots,N,\\
&\theta_k=\sqrt{c_k}(x - v_{5,k}t - \varrho_k)\quad \mbox{and} \quad  v_{5,k}:= c_k^2 + 10\mu^2c_k,
\end{aligned}\]
where $\varsigma^*$ means the complex conjugate of  $\varsigma$, here $c_k$ is the scaling  and $\varrho_k$ an arbitrary constant phase, $ \sum_{\mu=0,1}$ means here the summation over all
possible combinations of $\si_k=0,1,~~k=1,2,\dots N$ and $\sum_{k<m}^N$ the summation over all possible combinations of the N
elements under the constraint $k<m$.

\medskip
From that multi-soliton solution, the 2-soliton solution is apparent. 
Now, if we take this 2-soliton solution of \eqref{5G}, and transforming its corresponding scalings $c_1,c_2$ to complex ones 
$c_1=c_2^*:=(\beta+ i\alpha)^2$, $\alpha, \beta \in \R\backslash\{0\}$, it allows us to build a new solution of \eqref{5G} named as the \emph{breather solution}. 
This is a localized in space and periodic in time  (modulo symmetries of the equation) solution, and it is defined as follows:
%

\begin{defn}[5th order Gardner breather]
Let $\al, \bt \in \R\backslash\{0\},~\mu\in \R^{+}\backslash\{0\}$ such that $\Delta=\al^2+\bt^2-4\mu^2>0$,  and $x_1,x_2\in \R$. The 5th order breather solution $B_\mu\equiv B_{\mu,5}$ of the 5th order Gardner equation \eqref{5G},  is given explicitly by the formula 
\be\label{5BreG}
B_\mu\equiv B_{\al, \bt,\mu,5}(t,x;x_1,x_2)   
 :=   2\partial_x\Bigg[\arctan\Big(\frac{G_\mu(t,x)}{F_\mu(t,x)}\Big)\Bigg],
\ee
where
\[\begin{aligned}
G_\mu(t,x) & :=   \frac{\bt\sqrt{\al^2+\bt^2}}{\al\sqrt{\Delta}}\sin(\al y_1) -\frac{2\mu\bt[\cosh(\bt y_2)+\sinh(\bt y_2)]}{\Delta},\\
F_\mu(t,x) & :=   \cosh(\bt y_2)-\frac{2\mu\bt[\al\cos(\al y_1)-\bt\sin(\al y_1)]}{\al\sqrt{\al^2+\bt^2}\sqrt{\Delta}},
\end{aligned}\]
with $y_1$ and $y_2$
\[\begin{aligned}
&y_1 = x+ \delta_{5} t + x_1, \quad y_2 = x+ \ga_{5} t +x_2,~~~\end{aligned}\]
and with velocities
\be\label{speedsnthG}
\begin{aligned}
&\delta_5:=-\alpha^4+10\alpha^2\beta^2-5\beta^4  +10(\alpha^2-3\beta^2)\mu^2,\\
&\gamma_5:=-\beta^4+10\alpha^2\beta^2-5\alpha^4  +10(3\alpha^2-\beta^2)\mu^2.
\end{aligned}
\ee

\medskip
\noindent
\end{defn}

\medskip


First of all, we remember the following identity for solutions of the of 5th order Gardner equation \eqref{5G} (see \cite[Appendix A]{AM1} for a detailed proof of a similar identity for the classical Gardner equation)

\begin{lem}\label{matsuno1}
 Let $u(t,x)=\partial_x\log(\frac{F_\mu-iG_\mu}{F_\mu+iG_\mu})$ be any solution of the 5th order Gardner equation \eqref{5G}. Then
 
 \[u^2 = \partial_x^2\log(G_\mu^2 + F_\mu^2) - 2\mu u.\]
\end{lem}

Now, we can compute explicitly the mass of such  breather solution: 
\begin{lem}
  Let  $B_\mu\equiv B_{\mu,5}$ be the  breather solution \eqref{5BreG} of  the  5th order Gardner equation \eqref{5G}. Then the mass of $B_\mu$ is   
  \[M[B_\mu]:= 2\beta + 2\mu\arctan\Big[\frac{4\mu\beta}{\Delta}\Big].\]
\end{lem}

\begin{proof}
 The proof follows directly from Lemma \ref{matsuno1}, by using the breather solution \eqref{5BreG}.
\end{proof}
Moreover, the  breather solution \eqref{5BreG} of  the  5th order Gardner equation \eqref{5G} satisfies the following nonlinear identities:

\begin{lem}\label{identitiesB5}
 Let  $B_\mu\equiv B_{\mu,5}$ be the  breather solution \eqref{5BreG} of  the  5th order Gardner equation \eqref{5G}. Then
\ben
\item  $B_\mu =\tilde B_{\mu,x}$, with $\tilde B_\mu=\tilde B_{\al,\bt,\mu}$ given by the smooth $L^\infty$-function
\[\tilde B_\mu(t,x) := 2\arctan \Big(\frac{G_\mu}{F_\mu}\Big). \]
\item For any fixed $t\in \R$, we have $ (\tilde B_\mu)_t$ well-defined in the Schwartz class, satisfiying
\be\label{2ndga}
B_{\mu,4x} + \tilde B_{\mu,t} + 10(\mu + B_\mu)^2B_{\mu,xx} + 10(\mu + B_\mu)B_{\mu,x}^2 + 6(10\mu^3B_\mu^2 + 10\mu^2B_\mu^3 + 5\mu B_\mu^4 + B_\mu^5) = 0.
\ee
\een
\end{lem}

\begin{proof}
The first item above is a direct consequence of the definition of $B_\mu$ in \eqref{5BreG}. On the other hand, \eqref{2ndga} is a consequence of (\ref{5G}) and integration in space (from $-\infty$ to $x$) of (\ref{5G}). 
\end{proof}

Finally, we show that breather solutions \eqref{5BreG} of  the  5th order Gardner equation \eqref{5G} satisfy the following  identity:

\begin{lem}\label{Iddificil}
Let $B_\mu\equiv B_{\mu,5}$ be the  breather solution \eqref{5BreG} of  the  5th order Gardner equation \eqref{5G}. Then, for all $t\in \R$, 
\be\label{ide1nvbc}
 \tilde B_{\mu,t} = (\al^2 + \bt^2)^2 B_\mu +   2\big(\al^2 - \bt^2 - 5\mu^2\big)\big(B_{\mu,xx} + 2B_\mu^3 + 6\mu B_\mu^2\big).
\ee
\end{lem}
\begin{proof} 
See appendix \ref{apIddificil} for a detailed proof of this nonlinear identity.
\end{proof}

Now, we prove that breather solutions \eqref{5BreG} of  the  5th order Gardner equation \eqref{5G} satisfy a fourth order ODE, which indeed is the same as the one satisfied by \emph{classical} Gardner breather solutions (see \cite[Theorem 3.5]{AM1} for further details)

\begin{thm}\label{GBGE} Let $B_\mu\equiv B_{\mu,5}$ be the  breather solution \eqref{5BreG} of  the  5th order Gardner equation \eqref{5G}. Then, for any  fixed $t\in \R$, $B_\mu$  satisfies the nonlinear stationary equation  
\be\label{EcBGEfinal}
\begin{aligned}
\mathcal W (B_{\mu}) := & B_{\mu,4x}  -   2(\beta^2 -\alpha^2) (B_{\mu,xx} + 6\mu B_\mu^2 + 2B_\mu^3) +  (\alpha^2 + \beta^2)^2 B_\mu + 10 B_\mu B_{\mu,x}^2 + 10B_\mu^2 B_{\mu,xx} \\
&+ 6 B_\mu^5  + 10\mu B_{\mu,x}^2  + 20\mu B_\mu B_{\mu,xx} + 40\mu^2B_\mu^3 + 30 \mu B_\mu^4  =0.
\end{aligned}
\ee
\end{thm}

%

\begin{proof}
 We use the identity \eqref{2ndga} to substitute the $B_{\mu,4x}$ term in the left-hand side of \eqref{EcBGEfinal}, simplifying it as:
\[\begin{aligned}
\mathcal W (B_{\mu})=& -\left(\tilde B_{\mu,t} + 10(\mu + B_\mu)^2B_{\mu,xx} + 10(\mu + B_\mu)B_{\mu,x}^2 + 6(10\mu^3B_\mu^2 + 10\mu^2B_\mu^3 + 5\mu B_\mu^4 + B_\mu^5)\right)\\
&- 2(\beta^2 -\alpha^2) (B_{\mu,xx} + 6\mu B_\mu^2 + 2B_\mu^3) +  (\alpha^2 + \beta^2)^2 B_\mu + 10 B_\mu B_{\mu,x}^2 + 10B_\mu^2 B_{\mu,xx}\\
&+ 6 B_\mu^5  + 10\mu B_{\mu,x}^2  + 20\mu B_\mu B_{\mu,xx} + 40\mu^2B_\mu^3 + 30 \mu B_\mu^4\\
=& - \tilde B_t + (\al^2 + \bt^2)^2 B_\mu +   2\Big(\al^2 - \bt^2 - 5\mu^2\Big)(B_{\mu,xx} + 2B_\mu^3 + 6\mu B_\mu^2) =0,
\end{aligned}\]
where in the last line we have used the identity \eqref{ide1nvbc}.
\end{proof}
Note that being the shift parameters $x_1,x_2$ in \eqref{5BreG} selected as independents of time, a simple argument guarantees that the previous Theorem \ref{GBGE} still holds under time dependent, translation parameters $x_1(t)$ and $x_2(t)$.

\begin{cor}\label{Cor32}
Let $B^0_{\mu} \equiv B^0_{\al,\bt,\mu}(t,x; 0,0)$ be any Gardner breather as in \eqref{5BreG}, 
and $x_1(t),$ $x_2(t)\in \R$ two continuous functions, defined for all $t$ in a given interval. Consider the modified breather
\[
 B_{\mu} (t,x):=B^0_{\al,\bt,\mu}(t,x; x_1(t),x_2(t)), \qquad (\hbox{cf. \eqref{5BreG}}). 
\]
Then $B_{\mu} $ satisfies \eqref{EcBGEfinal}, for all $t$ in the given interval.
\end{cor}

\begin{proof}
From the invariance of the equation \eqref{EcBGEfinal} under spatial translations, we conclude.
\end{proof}
Even more, we can characterize variationally these breather solutions of the 5th order Gardner equation. Explicitly, considering the $H^2(\R)$ conserved quantity  \eqref{E5}
\begin{equation}\label{E5a}
 E_{5\mu}[u](t)   :=    \int_\R \left(\frac 12 u_{xx}^2 - 10\mu  uu^2_x + 10\mu^2 u^4  - 5 u^2u_x^2 + 6\mu  u^5 + u^6 \right)\; dx,
\end{equation}
we can introduce a $H^2$ functional, associated to the breather solution. Namely, we define this functional as a linear combination of the energy \eqref{E1}, the mass \eqref{M1cor} and \eqref{E5a} in the following way 
\be\label{LyapunovGE}
\mathcal H_\mu[u](t) := E_{5\mu}[u](t) + 2(\bt^2-\al^2)E_\mu[u](t) + (\al^2 +\bt^2)^2 M[u](t).
\ee
Therefore, $\mathcal H_\mu[u]$ is  a  conserved quantity, well-defined for $H^2$-solutions of \eqref{5G}. Additionally, we have that
\begin{lem}\label{crit}
Breather solutions $B_\mu$ \eqref{5BreG} of the 5th order Gardner equation \eqref{5G} are critical points of the Lyapunov functional $\mathcal H_\mu$ \eqref{LyapunovGE}. In fact, for any $z\in H^2(\R)$  with sufficiently small $H^2$-norm, and $B_\mu=B_{\al,\bt,\mu}$  any 5th Gardner breather solution, one has, for all $t\in \R$, that  
\[\mathcal{H}_\mu[B_\mu +z] - \mathcal{H}_\mu[B_\mu]  = \frac 12\mathcal Q_\mu[z] + \mathcal N_\mu[z],\]
with $\mathcal Q_\mu$ being the quadratic form defined in \eqref{Qmu} below, and $\mathcal N_\mu[z]$ satisfying $|\mathcal N_\mu[z] | \leq K\|z\|_{H^2(\R)}^3$.
\end{lem}
\begin{proof}
A direct computation with the integration by parts yields
\[\begin{aligned}
E_{5\mu}[B_\mu +z] = &~{} E_{5\mu}[B_\mu +z] \\
&+ \int \left(B_{\mu, 4x} + 10 \mu B_{\mu, x}^2 + 20\mu B_{\mu}B_{\mu, xx} +10B_{\mu}B_{\mu, x}^2 + 10B_{\mu}^2B_{\mu,x} + 40\mu^2 B_{\mu}^3 + 30\mu B_{\mu}^4 + 6B_{\mu}^5 \right) z\\
&+ \frac12 \int \Big(\px^4 + (20\mu B_{\mu} + 10B_{\mu}^2)\px^2 - 20(\mu B_{\mu, x} + B_{\mu}B_{\mu,x}) \px \\
& \qquad \qquad + (-10B_{\mu, x}^2 + 120\mu^2 B_{\mu}^2 +120 \mu B_{\mu}^3+ 30B_{\mu}^4) \Big) z \cdot z \\
&+\int \left(-10 \mu z z_x^2 - 10 B_{\mu} z z_x - 10 B_{\mu, x}z_x z^2 + 40\mu^2 B_{\mu}z^3 + 60 \mu B_{\mu}^2 z^3 +20 B_{\mu}^3 z^3 \right),
\end{aligned}\]
\[E_\mu[B_{\mu} + z] = E_{\mu}[B_{\mu}]  - \int ( B_{\mu, xx} + 6\mu B_{\mu}^2 + 2 B_{\mu}^3 ) z - \frac12\int (\px^2 + 12\mu B_{\mu} + 6B_{\mu}^2 )z \cdot z - \int \left(2\mu z^3 + 2 B_{\mu} z^3 + \frac12 z^4 \right)\]
and
\[M[B_{\mu} + z] = M[B_{\mu}] + \int B_{\mu} z +  \frac12 \int z \cdot z.\]
Collecting all, one obtain
\[
\mathcal{H}_\mu[B_\mu+z]   =  \mathcal{H_\mu}[B_\mu] + \int_\R \mathcal W (B_{\mu}) z + \frac 12\mathcal Q_\mu[z] + \mathcal N_\mu[z],
\]
where the quadratic form 
\be\label{Qmu}
\mathcal Q_\mu [z] := \int \mathcal L_\mu z \cdot z,
\ee
associated to the linearized operator $\mathcal L_\mu$ given by
\begin{equation}\label{linearized operator}
\begin{aligned}
\mathcal L_{\mu} :=&~{} \px^4 + (20 \mu B_{\mu} + 10 B_{\mu}^2 - 2(\beta^2 - \alpha^2)) \px^2 - 20(\mu B_{\mu, x} + B_{\mu}B_{\mu,x}) \px  \\
&+ \left(-10B_{\mu, x}^2 + 120\mu^2 B_{\mu}^2 +120 \mu B_{\mu}^3+ 30B_{\mu}^4  - 2(\beta^2 - \alpha^2)(12\mu B_{\mu} + 6B_{\mu}^2) + (\alpha^2 + \beta^2)^2\right)
\end{aligned}
\end{equation}
%
%
%
%
and the collection of higher order terms (with respect to $z$) $N_\mu[z]$ is given by 
\[\begin{aligned}
\mathcal N_\mu[z] :=&~{} \int \left(-10 \mu z z_x^2 - 10 B_{\mu} z z_x - 10 B_{\mu, x}z_x z^2 + 40\mu^2 B_{\mu}z^3 +60 \mu B_{\mu}^2 z^3 +20 B_{\mu}^3 z^3 \right) \\
&-2(\beta^2 - \alpha^2)\int \left(2\mu z^3 + 2 B_{\mu} z^3 + \frac12 z^4 \right).
\end{aligned}\]
Theorem \ref{GBGE} ensures $\int_\R \mathcal W (B_{\mu}) z = 0$, and hence one has $\mathcal{H_\mu}'[B_\mu] = 0$. Moreover, from direct estimates, one has $\mathcal N_\mu[z] = O(\|z\|_{H^2(\R)}^3),$ as desired.
%
%
%
%
\end{proof}

\subsection{Spectral analysis}\label{3}
%
As a direct consequence of the already studied  spectral properties  of the linearized operator $\mathcal L_\mu $, associated to the \emph{classical} Gardner breather solution $B_\mu$,  
in \cite{AM1}, we obtain the same spectral results for  breather solutions  of the 5th order Gardner equation \eqref{5G}. In fact, 
all statements on spectral properties and the main Theorem in \cite{AM1} are valid for the 5th order Gardner equation \eqref{5G}, 
even if explicit coefficients are different. Therefore in the following lines and for the sake of completeness, 
we only summarize and list the main features of $\mathcal L_{\mu}$  \eqref{linearized operator}. Let $B_{\mu}$ as introduced in \eqref{5BreG}. Consider now the two directions associated to spatial translations. We define 
\be\label{B12}
 B_1(t ; x_1,x_2) := \partial_{x_1} B_{\mu}(t ; x_1, x_2),\quad \hbox{ and } \quad  B_2(t ; x_1,x_2): =\partial_{x_2} B_{\mu}(t ;  x_1, x_2).\\
\ee
Moreover, we compute and denote as \emph{scaling directions}, the derivatives
\begin{align}\label{DDeltaAB}
\Lambda_\al B_\mu =\frac{\partial B_\mu}{\partial \al},\quad \Lambda_\bt B_\mu = \frac{\partial B_\mu}{\partial \bt}.
\end{align}
\medskip
\noindent
We get the following (see \cite{AM1} for more details)

\medskip

\begin{lem}\label{Lspectral} For any   breather solution $B_\mu$ \eqref{5BreG} of the 5th order Gardner equation \eqref{5G},  we get that
\begin{enumerate}
 \item (\emph{Continuous spectrum}) $\mathcal L_\mu$ is a linear, unbounded operator in $L^2(\R)$, with dense domain $H^4(\R)$. Moreover, $\mathcal L_\mu$ is self-adjoint, and is a compact perturbation of the constant coefficients operator
\[
\mathcal L_{\mu,0} :=  \px^4 - 2(\bt^2 -\al^2) \px^2 + (\al^2 +\bt^2)^2 .
\]
In particular, the continuous spectrum of $\mathcal L_\mu$ is the closed interval $[(\al^2 +\bt^2)^2,+\infty)$ in the case $\beta\geq \al$, and $[ 4\al^2 \bt^2 ,+\infty)$ in the case $\beta< \al$, with no embedded eigenvalues are contained in this region.

\medskip

\item (\emph{Kernel}) For each $t\in \R$, one has
\[
\ker \mathcal L_\mu =\spawn \big\{ B_1(t;x_1,x_2), B_2(t;x_1,x_2)\big\}.
\]

\medskip

\item Consider the scaling directions $\Lambda_\al B$ and $\Lambda_\bt B$ introduced in \eqref{DDeltaAB}. Then, given $\al,\bt>0$ and $\forall\mu\in(0,\frac{\sqrt{\al^2+\bt^2}}{2})$, we have

\[\int_\R  \Lambda_\al B_\mu \, \mathcal L_\mu [\Lambda_\al B_\mu]  = 16 \al^2\bt\left[1+\frac{4\mu^2\Delta}{\Delta^2+16\mu^2\bt^2}\right]>0,\]
and
\[\int_\R  \Lambda_\bt B_\mu\, \mathcal L_\mu [\Lambda_\bt B_\mu]  =  -16 \bt\left[\al^2 + 2\mu^2\left(1+\frac{(\Delta-2\beta^2)(\alpha^2 + \beta^2 +4\mu^2)}{\Delta^2+ 16\mu^2\bt^2}
\right)\right]<0.\]

\medskip

\item Let
\[B_{0,\mu} := \frac{\al\Lambda_\bt B_\mu + \bt\Lambda_\al B_\mu}{8\al\bt(\al^2+\bt^2)}.\\\]
Then $B_{0,\mu}$ is in the Schwartz class, satisfying $\mathcal L_\mu[B_{0,\mu}] = - B_\mu$ and
\[\int_\R B_{0,\mu} B_\mu = \frac{1}{4\bt(\al^2+\bt^2)}\Big(\frac{\Delta^2+4\mu^2\Delta}{\Delta^2+16\mu^2\bt^2}\Big)>0. \quad \forall \mu\in\left(0,\frac{\sqrt{\al^2+\bt^2}}{2}\right)\]
%

\medskip

\item Let $B_1$ and $B_2$ be the kernel elements defined in \eqref{B12},  $D_\mu = F_\mu^2 + G_\mu^2$  and $W$ be the Wronskian matrix of the functions $B_1$ and $B_2$, precisely given by
\[W[B_1, B_2] (t,x) := \left[ \begin{array}{cc} B_1 & B_2 \\  (B_1)_x & (B_2)_x  \end{array} \right] (t,x).\]
Then

\[\begin{aligned}
\det W[B_1,B_2](t,x):= &~{}\frac{2\bt^3(\al^2+\bt^2)^2((\al^2+\bt^2)^2-8\mu^2(\al^2-2\mu^2))}{\Delta^3 D_\mu^2} \times \\
&\Bigg[\sinh(2\bt y_2) + \frac{8\bt^2\mu^2\cosh(2\bt y_2)}{(\al^2+\bt^2)^2-8\mu^2(\al^2-\mu^2)} - \frac{\bt\Delta((\al^2+\bt^2)^2 - 4\mu^2(\al^2-\bt^2))\sin(2\al y_1)}{\al(\al^2+\bt^2)((\al^2+\bt^2)^2 - 8\mu^2(\al^2-\mu^2))}\\
&+\frac{8\bt^2\mu^2\Delta\cos(2\al y_1)}{(\al^2+\bt^2)((\al^2+\bt^2)^2 - 8\mu^2(\al^2-2\mu^2))}\Bigg].\\
\end{aligned}\]

\medskip

\item For every $\mu\in(0,\frac{\sqrt{\al^2+\bt^2}}{2})$, the operator $\mathcal L_\mu$  defined in \eqref{linearized operator} has a 
unique negative eigenvalue $-\la_0^2<0$, of multiplicity one, where $\la_0$ depends on $\al$, $\bt$, $\mu$, $x_1$, $x_2$ and $t$.

\medskip

\item (\emph{Coercivity}) Let $\alpha, \beta > 0$ and $\mu\in(0,\frac{\sqrt{\al^2+\bt^2}}{2})$. For the quadratic from $Q_{\mu}[z]$ as in \eqref{Qmu}, associated to $\mathcal{L}_\mu$ \eqref{linearized operator},
%
there exists a well-defined and positive continuous function $\nu_0 =\nu_0(\al,\bt,\mu)$ such that, for all $z_0\in H^2(\R)$ satisfying 
\[\int_\R z_0 B_{-1} =\int_\R z_0 B_1 =\int_\R z_0 B_2 =0,\]
the following Coercivity condition holds true:
\be\label{coee} 
\mathcal Q_\mu[ z_0] \geq \nu_0\| z_0\|_{H^2(\R)}^2.
\ee
\end{enumerate}
\end{lem}

For the proof of this Lemma, we refer the interested reader to \cite[Lemma 5.10]{AM1}. Finally, we present the stability result for breather solutions \eqref{5BreG} of the 5th order Gardner equation \eqref{5G}:

\begin{thm}[$H^2$-stability of 5th order Gardner breathers]\label{T1BreG} 
Let $\al, \bt \in \R\backslash\{0\}$ and $\mu\in(0,\frac{\sqrt{\al^2+\bt^2}}{2})$. Let  $B_\mu\equiv B_{\mu,5}$ the  breather solution \eqref{5BreG} of  the  5th order Gardner equation \eqref{5G}. Then,  there exist positive parameters $\eta_0, A_0$, depending on $\al,\beta$ and $\mu$, such that the following holds: Consider $u_0 \in H^2(\R)$, and assume that there exists $\eta \in (0,\eta_0)$ such that 
\[\|  u_0 - B_{\mu}(t=0;0,0) \|_{H^2(\R)} \leq \eta.\]
Then there exist $x_1(t), x_2(t)\in \R$ such that the solution $u(t)$ of the Cauchy problem for the 5th order Gardner equation \eqref{5G} with initial data $u_0$, satisfies
\[\sup_{t\in \R}\big\| u(t) - B_{\mu}(t; x_1(t),x_2(t)) \big\|_{H^2(\R)}\leq A_0 \eta,\]
with
\[\sup_{t\in \R}|x_1'(t)| +|x_2'(t)| \leq KA_0 \eta,\]
for a constant $K>0$.
\end{thm}
\begin{proof}
We take $u=u(t) \in H^2(\R)$ as the corresponding local in time solution of the Cauchy problem associated to (\ref{5G}), with  initial condition $u(0)=u_0\in H^2(\R)$. Therefore, once we guaranteed for the case of the breather solution of the  5th order  Gardner equation, that it satisfies the same 4th order ODE \eqref{EcBGEfinal} as the \emph{classical} Gardner breather, that a suitable coercivity property holds for the bilinear form $\mathcal Q_\mu$ associated to the breather solution of \eqref{5G} (see \eqref{coee}), and the existence of a unique negative eigenvalue (Lemma \ref{Lspectral} (6)) of the linearized operator $\mathcal L_\mu$ given in \eqref{linearized operator}, the stability proof follows the same steps as the $H^2$-stability of \emph{classical} Gardner breathers \cite[Theorem 6.1]{AM1} (see also \cite[Theorem 6.1]{AM}).  Namely, we proceed assuming that the maximal time of stability $T$ is finite and we arrive to a contradiction. 
\end{proof}
\appendix

\section{Proof of Theorem \ref{Illposed}}
The aim of this section is to prove the ill-posedness of \eqref{5G} for $s > 0$, which, in addition to the first author's recent work \cite{AC2018}, completely justifies that the 5th Gardner equation \eqref{5G} is the quasilinear equation in the sense that the flow map from data to solutions is not (locally) uniformly continuous for all regularities, see Corollary \ref{cor:illposed}. Since the weak-illposedness phenomenon occurs due to the strong high-low interaction in the quadratic nonlinearity with three derivatives, Theorem 1.2 in \cite{Kwon2008} seems to guarantee the lack of uniform continuity of the flow map associated to \eqref{5G} for $s > 0$. This section contributes to prove that the equation \eqref{5G} is indeed weakly ill-posed for $s > 0$.

\medskip

The proof basically follow the argument used in \cite{Kwon2008}, initially introduced by Koch-Tzvetkov \cite{KT2005}. Since the (weak) ill-posedness phenomenon arises from the strong high-low quadratic nonlinearity (high frequency waves with low frequency perturbations), the main part of the proof is identical to the argument in \cite{Kwon2008}. Thus, we, here, provide an additional estimate to be needed for the other nonlinearities.

\medskip
	
In view of the argument presented in Section \ref{sec:LWP}, it suffices to show the ill-posedness of \eqref{5G_Gen} with small initial data.

\subsection{Setting}
We first define the approximate solution, which is an ansatz to cause the (weak) ill-posedness phenomenon. Let $\phi, \wt{\phi} \in C_0^{\infty}(\R)$ be smooth bump functions satisfying
\[\phi \equiv 1, \quad |x| < 1, \quad \mbox{and} \quad \phi \equiv 0, \quad |x| > 2\]
and
\[\wt{\phi} \equiv 1, \quad x \in \mbox{supp} (\phi) \quad \mbox{and} \quad \wt{\phi}\phi \equiv \phi,\]
respectively. For $N \ge 1$ and $0 < \delta < 1$, set
\[\phi_{N}(x) := \phi\left(\frac{x}{N^{4+\delta}}\right), \quad \wt{\phi}_{N}(x) := \wt{\phi}\left(\frac{x}{N^{4+\delta}}\right).\]
Let $\epsilon > 0$ be a sufficiently small for the initial data to satisfy \eqref{eq:assumption}. Let 
\[u_{0,l}^{\pm}(x) := \pm\epsilon N^{-3}\wt{\phi}_{N}(x)\]
and $u_l^{\pm}(t,x)$ be the solution to \eqref{5G_Gen} with the initial data $u_{0,l}^{\pm}(x)$. 
Let $\Phi_N(t) := (N^5-10\mu^2\lambda^2N^3)t$ and
\begin{equation}\label{hiwave}
u_h^{\pm}(t,x) :=
 N^{-\frac{4+\delta}{2}-s}\phi_{N}(x) \cos\left(N x -\Phi_N(t) \mp t\right) 
\end{equation}
be a high frequency part of the approximate solution, and thus define the approximate solution as 
\[u_{ap}^{\pm}(t,x):=u_l^{\pm}(t,x) + u_h^{\pm}(t,x).\]
Then the main task is to prove the following proposition:
\begin{prop}[Proposition 6.2 in \cite{Kwon2008}]\label{prop:ill-posed}
Let $\max(0,2-2s) < \delta < 1$. Let $u_{N}^{\pm}$ be the unique solution to \eqref{5G_Gen} with initial data
\[u_{N}^{\pm}(0,x) = \pm\epsilon N^{-3}\wt{\phi}_{N}(x) + N^{-\frac{4+\delta}{2}-s}\phi_{N}(x) \cos\left(N x \right).\]
Then, we have
\begin{equation}\label{eq:approx}
\norm{u_{N}^{\pm} - u_{ap}^{\pm}}_{H^s} = o(1),
\end{equation}
for $s > 0$ and $|t| <1$, as $N \to \infty$.
\end{prop}
Once \eqref{eq:approx} holds true, one conclude that
\[\begin{aligned}
\norm{u_{N}^+ - u_{N}^-}_{H^s} =&~{} N^{-\frac{4+\delta}{2}-s}\norm{\phi_{N}(x)\left(\cos\left(N x -\Phi_N(t) + t\right)-\cos\left(N x -\Phi_N(t) - t\right)\right)}_{H^s} + o(1)\\
=&~{} 2N^{-\frac{4+\delta}{2}-s}\norm{\phi_{N}(x)\sin\left(N x -\Phi_N(t)\right)}_{H^s}|\sin t| + o(1),
\end{aligned}\]
which, in addition to Lemma \ref{lem:limit} below, implies
\[\lim_{N \to \infty}\norm{u_{N}^+ - u_{N}^-}_{H^s}  \ge c|\sin t| \sim c|t|,\]
for $|t| < 1$. This completes the proof of Theorem \ref{Illposed}.

\medskip

We recall from \cite{KT2005, Kwon2008} the following useful lemmas to prove Proposition \ref{prop:ill-posed}.
\begin{lem}[Lemma 2.3 in \cite{KT2005}]\label{lem:limit}
Let $s \ge 0$, $\delta > 0$ and $\gamma \in \R$. Then,
\[\lim_{N \to \infty} N^{-\frac{4+\delta}{2}-s}\norm{\phi_{N}(x)\sin\left(N x + \gamma\right)}_{H^s} = c_0 \norm{\phi}_{L^2},\]
for some $c_0 >0$.
\end{lem}

\begin{lem}[Lemma 6.3 in \cite{Kwon2008}]\label{lem:low}
Let $K$ be a positive integer and $K-2-s \ge k \ge 0$. Then, we have
\begin{equation}\label{low_1}
\norm{\px^ku_l^{\pm}(t,\cdot)}_{L^2} \lesssim _K N^{-\frac{2-\delta}{2}-k(4+\delta)}
\end{equation}
\begin{equation}\label{low_2}
\norm{\px^ku_l^{\pm}(t,\cdot)}_{L^{\infty}} \lesssim _K N^{-3-k(4+\delta)}
\end{equation}
\begin{equation}\label{low_3}
\norm{u_l^{\pm}(t,\cdot)-u_{0,l}^{\pm}(\cdot)}_{L^2} \lesssim _K N^{-15-3\delta}
\end{equation}
\end{lem}
\begin{proof}
The proof of \eqref{low_1} and \eqref{low_2} follows from a direct computation and Theorem \ref{LWP1}, in particular, a priori bound \eqref{eq:result}. Moreover, the proof of \eqref{low_3} follows from a direct calculation in \eqref{5G_Gen} and \eqref{low_1}--\eqref{low_2}. The proof is almost identical to the proof of Lemma 6.3 in \cite{Kwon2008}, thus we omit the details.
\end{proof}

\begin{lem}\label{lem:approx}
Let 
\begin{equation}\label{F}
\mathcal P^{\pm}(t,x):=u_{ap, t}^{\pm}+u_{ap, 5x}^{\pm}+10\mu^2\lambda^2u_{ap, 3x}^{\pm} +\mathcal N_2(u_{ap}^{\pm}) + \mathcal N_3(u_{ap}^{\pm}) + \mathcal{SN}(u_{ap}^{\pm}),
\end{equation}
where $\mathcal N_2(\cdot)$, $\mathcal N_3(\cdot)$ and $\mathcal{SN}(\cdot)$ are defined as in \eqref{N2}--\eqref{SN}, respectively.
Let $s > 0$, $0<\delta <2$ and $|t| \le 1$. Then, we have 
\begin{equation}\label{eq:F_L2}
\norm{\mathcal P^{\pm}(t,\cdot)}_{L^2} \lesssim N^{-s-\delta} + N^{\frac{2-\delta}{2} - 2s} + N^{-1-\delta -3s} + N^{1-\frac{3(4+\delta)}{2}-4s} +N^{1-2(4+\delta)-5s}.
\end{equation}
Moreover, if $\sigma > 0$, we have
\begin{equation}\label{eq:F_H}
\norm{\mathcal P^{\pm}(t,\cdot)}_{H^{\sigma}} \lesssim N^{-s-\delta+\sigma} + N^{\frac{2-\delta}{2} - 2s + \sigma} + N^{-1-\delta -3s + \sigma} + N^{1-\frac{3(4+\delta)}{2}-4s + \sigma} +N^{1-2(4+\delta)-5s + \sigma}.
\end{equation}
\end{lem}

\begin{proof}
It suffices to consider $\mathcal P^{+}$, since an identical argument holds true for $\mathcal P^{-}$. We drop the super-index $+$. We decompose $\mathcal P$ into $\mathcal P_1+\mathcal P_2$, where $\mathcal P_2= \mathcal N_3(u_{ap}) -\mathcal N_3(u_l) + \mathcal{SN}(u_{ap}) - \mathcal{SN}(u_l)$ and $\mathcal P_1 = \mathcal P -\mathcal P_2$. Lemma 6.4 in \cite{Kwon2008} exactly shows \eqref{eq:F_L2} and \eqref{eq:F_H} for $\mathcal P_1$\footnote{A small difference between $\mathcal P_1$ and $F$ in Lemma 6.4 in \cite{Kwon2008} does not make any trouble. Indeed, our setting of $u_h$ corresponds to \eqref{5G_Gen}, so that one can immediately  apply the argument in the proof of Lemma 6.4 in \cite{Kwon2008} to our case. Moreover, the cubic term with one derivative in $\mathcal N_2(u_{ap})$ can be dealt with similarly as $\mathcal{SN}(u_{ap})$.}. Our setting of $\phi$, $\wt{\phi}$ and $u_l$ is essential to deal with 
\[\Lambda := \epsilon N^{-\frac{4+\delta}{2}-s}\phi_N(x)(\pt+\px^5+10\mu^2\lambda^2\px^3 + \epsilon^{-1}u_l\px^3)\cos\left(N x -\Phi_N(t) - t\right)\]
contained in $\mathcal P_1$ (compared to $F_4$ in the proof of Lemma 6.4 in \cite{Kwon2008}). Indeed, a direct calculation in addition to $u_{0,l}(x) := \epsilon N^{-3}\wt{\phi}_{N}(x)$ and $\phi \wt{\phi} = \phi$ gives
\[\begin{aligned}
\Lambda =&~{} N^{-\frac{4+\delta}{2}-s}\phi_N(x)\left(u_lN^3 - \epsilon \right)\sin\left(N x -\Phi_N(t) - t\right) \\
=&~{} N^{-\frac{4+\delta}{2}-s}N^3\phi_N(x)\left(u_l - u_{0,l} \right)\sin\left(N x -\Phi_N(t) - t\right),
\end{aligned}\]
which is handled by using \eqref{low_3}. Thus, it suffices to show \eqref{eq:F_L2} and \eqref{eq:F_H} for $\mathcal P_2$. Putting first $u_{ap} = u_l + u_h$ into $10u_{ap}^2u_{ap, 3x} - 10 u_l^2u_{l, x}$ in $\mathcal P_2$, one has
\begin{equation}\label{example1}
10u_l^2u_{h, xxx} + 20u_lu_hu_{l, xxx} + 20u_lu_hu_{h, xxx} + 10u_h^2u_{l, xxx} + 10u_h^2u_{h, xxx}.
\end{equation}
Note that 
\[\begin{aligned}
u_{h,x} =&~{} N^{-\frac{4+\delta}{2}-s}\left(\px\phi_{N}(x) \cos\left(N x -\Phi_N(t) - t\right) +\phi_{N}(x) \px \cos\left(N x -\Phi_N(t) - t\right)\right)\\
=&~{}N^{-\frac{4+\delta}{2}-s}\left(N^{-(4+\delta)}\phi_{N,x}(x)\cos\left(N x -\Phi_N(t) - t\right)- N \phi_{N}(x)\sin\left(N x -\Phi_N(t) - t\right)\right). 
\end{aligned}\]
Thus, one can see that the worst term arises from the case when the derivative acts on $\cos\left(N x -\Phi_N(t) - t\right)$. Using Lemmas \ref{lem:limit} and \ref{lem:low}, one estimates
\[\norm{\eqref{example1}}_{L^2} \lesssim  N^{-3-s} + N^{-\frac{4+\delta}{2}-2s} + N^{-1-\delta -3s}.\]
An analogous argument yield
\[\norm{u_{ap, x}^3 - u_{l, x}^3}_{L^2} \lesssim N^{-5-2(4+\delta)-s} + N^{-1-\frac{3(4+\delta)}{2}-2s} + N^{-1-\delta-3s},\]
\[\norm{u_{ap}u_{ap, x}u_{ap, xx} - u_{l}u_{l, x}u_{l, xx}}_{L^2} \lesssim N^{-8-\delta-s} + N^{-\frac{4+\delta}{2}-2s} + N^{-1-\delta -3s},\]
\[\norm{u_{ap}^4u_{ap,x} - u_{l}^4u_{l,x}}_{L^2} \lesssim N^{-11-s} + N^{-8-\frac{4+\delta}{2}-2s} + N^{-5-(4+\delta)-3s} + N^{-2-\frac{3(4+\delta)}{2}-4s}+ N^{1-2(4+\delta)-5s},\]
\[\norm{u_{ap}u_{ap, x} - u_{l}u_{l, x}}_{L^2} \lesssim N^{-2-s} + N^{1-\frac{4+\delta}{2}-2s}\]
and
\[\norm{u_{ap}^3u_{ap, x} - u_{l}^3u_{l, x}}_{L^2} \lesssim N^{-8-s} + N^{-5-\frac{4+\delta}{2}-2s} +N^{-2-(4+\delta)-3s}+N^{1-\frac{3(4+\delta)}{2}-4s}.\]
Collecting all, we completes the proof of \eqref{eq:F_L2}. Moreover, the fractional Leibniz rule ensure at least $\norm{\mathcal P}_{\dot{H}^{\sigma}} \lesssim_{\sigma} N^{\sigma} \norm{\mathcal P}_{L^2}$, which in addition to \eqref{eq:F_L2} implies \eqref{eq:F_H}, since $u_{l, t}+u_{l, 5x}+10\mu^2\lambda^2u_{l, 3x} + 30\mu^4 \lambda^4 u_{l, x} +\mathcal N_2(u_l) + \mathcal N_3(u_l) + \mathcal{SN}(u_l) = 0$ and the others contains at least one $u_h$. We complete the proof.
\end{proof}

\subsection{Proof of Proposition \ref{prop:ill-posed}}
Let $w^{\pm} := u_N^{\pm} - u_{ap}^{\pm}$. We only show $\norm{w^+}_{H^s} = o(1)$ as $N \to \infty$ and drop the super-index $+$. For $s \ge 2$, the local well-posedness theory is available. A direct calculation gives
\[\Gamma w + \mathcal N_2(u_N) - \mathcal N_2(u_{ap}) + \mathcal N_3(u_N) - \mathcal N_3(u_{ap}) + \mathcal{SN}(u_N) - \mathcal{SN}(u_{ap}) + \mathcal P =0,\]
where $\Gamma := \pt +\px^5 +10\mu^2\lambda^2 \px^3$ and $\mathcal P$ is as in \eqref{F}. For $2 \le \sigma$, the local well-posedness, in particular \eqref{eq:result}, ensures 
\begin{equation}\label{uNHs}
\norm{u_N}_{C_TH^\sigma} + \norm{u_N}_{F^\sigma(T)} \lesssim \norm{u_N(0)}_{H^{\sigma}} \lesssim N^{\sigma -s}.
\end{equation}
Moreover, a direct calculation and the local theory (for $u_l$) gives
\begin{equation}\label{uapHs}
\norm{u_{ap}}_{C_TH^\sigma} + \norm{u_{ap}}_{F^\sigma(T)} \lesssim N^{-\frac{2-\delta}{2}} +  N^{\sigma -s}.
\end{equation}
Using Propositions \ref{prop:linear}, Propositions \ref{prop:bilinear}, \ref{prop:trilinear}, \ref{prop:multilinear} and \ref{prop:energy diff}, and \eqref{eq:F_L2} under \eqref{uNHs} and \eqref{uapHs}, one concludes
\[\norm{w}_{F^0(T)} \lesssim \norm{\mathcal P}_{L_T^1L_x^2} = O(N^{-s -\beta}),\]
for $\beta = \min(\delta, -\frac{2-\delta}{2}+s) > 0$, which, in addition to Proposition \ref{prop:H^s}, implies
\begin{equation}\label{ill_1}
\norm{w}_{L_T^{\infty}L_x^2} = O(N^{-s -\beta}).
\end{equation}
Furthermore, an analogous argument (but using \eqref{eq:F_H} instead of \eqref{eq:F_L2}) in addition to 
\[\norm{u_{ap}}_{F^2s(T)}\norm{w}_{F^0(T)} = O(N^{s}N^{-s-\beta}) = O(N^{-\beta}),\]
ensures $\norm{w}_{F^s(T)}  = O(N^{-\beta})$, which concludes \eqref{eq:approx} as $N \to \infty$ for $s \ge 2$.

\medskip

To fill the regularity range $0 < s < 2$, we use the conservation law and the interpolation theorem. $H^2$ conservation law \eqref{E5-1} and a direct calculation yield
\[\norm{u_N}_{H^2} \lesssim N^{2-s} \quad \mbox{and} \quad \norm{u_{ap}}_{H^2} \lesssim N^{2-s},\]
respectively, which concludes
\begin{equation}\label{ill_2}
\norm{w}_{H^2} \lesssim N^{2-s}.
\end{equation}
The interpolation between \eqref{ill_1} and \eqref{ill_2} ensures
\[\norm{w}_{H^s} \lesssim \norm{w}_{L^2}^{1-\frac{s}{2}}\norm{w}_{H^3}^{\frac{s}{2}} \lesssim N^{-\frac{\beta (2-s)}{2}},\]
which proves \eqref{eq:approx} as $N \to \infty$ for $0 < s < 2$.

\section{Proof of Lemma \ref{Iddificil}.}\label{apIddificil}
We are going to prove the identity \eqref{ide1nvbc}
\[
 \tilde B_{\mu,t} = (\al^2 + \bt^2)^2 B_\mu +   2\big(\al^2 - \bt^2 - 5\mu^2\big)\big(B_{\mu,xx} + 2B_\mu^3 + 6\mu B_\mu^2\big).
\]
Firstly and  for the sake of simplicity, we will use the following notation:
\begin{align*}
& A_1:=  (\al^2+\bt^2)^2 ,\quad A_2:= 2(\al^2 - \bt^2 - 5\mu^2),\nonu\\
&\Delta=\al^2+\bt^2-4\mu^2,\quad e^{z} = \cosh(z) + \sinh(z),\nonu\\
&D:= f^2 + g^2,\quad \text{where}\quad f, g~~\text{and its derivatives are given by:}\nonu\\
\end{align*}
\begin{equation}\label{notacionF}
\begin{aligned}
&  f=\cosh(\bt y_2)-\frac{2\bt\mu}{\al\sqrt{\al^2+\bt^2}\sqrt{\Delta}}(\al\cos(\al y_1) -\bt\sin(\al y_1) ),\\
&  f_1:=f_x=\bt\sinh(\bt y_2)+\frac{2\bt\mu}{\sqrt{\al^2+\bt^2}\sqrt{\Delta}}(\bt\cos(\al y_1) +\al\sin(\al y_1)), \\
&  f_2:=f_t=\bt\ga_5\sinh(\bt y_2)+\frac{2\bt\delta_5\mu}{\sqrt{\al^2+\bt^2}\sqrt{\Delta}}(\bt\cos(\al y_1) +\al\sin(\al y_1) ),\\
&  f_3:=f_{xx}=\bt^2\cosh(\bt y_2)+\frac{2\al\bt\mu}{\sqrt{\al^2+\bt^2}\sqrt{\Delta}}(-\al\cos(\al y_1) +\bt\sin(\al y_1) ),\\
&  f_4:=f_{xxx}=\bt^3\sinh(\bt y_2)-\frac{2\al^2\bt\mu}{\sqrt{\al^2+\bt^2}\sqrt{\Delta}}(\bt\cos(\al y_1) +\al\sin(\al y_1) )
 \end{aligned}
\end{equation}
and
\begin{equation}\label{notacionG}
\begin{aligned}
&  g=\frac{\bt\sqrt{\al^2+\bt^2}}{\al\sqrt{\Delta}}\sin(\al y_1) -  \frac{2\bt\mu e^{\bt y_2}}{\Delta},\\
&  g_1:=g_x=\frac{\bt\sqrt{\al^2+\bt^2}}{\sqrt{\Delta}}\cos(\al y_1) -  \frac{2\bt^2\mu e^{\bt y_2}}{\Delta},\\
&  g_2:=g_t=\frac{\bt\delta_5\sqrt{\al^2+\bt^2}}{\sqrt{\Delta}}\cos(\al y_1) -  \frac{2\bt^2\ga_5\mu e^{\bt y_2}}{\Delta},\\
&  g_3:=g_{xx}=-\frac{\al\bt\sqrt{\al^2+\bt^2}}{\sqrt{\Delta}}\sin(\al y_1) -  \frac{2\bt^3\mu e^{\bt y_2}}{\Delta},\\
&  g_4:=g_{xxx}=-\frac{\al^2\bt\sqrt{\al^2+\bt^2}}{\sqrt{\Delta}}\cos(\al y_1) -  \frac{2\bt^4\mu e^{\bt y_2}}{\Delta},
\end{aligned}
\end{equation}
where velocities $(\ga_5,\delta_5)$ are given in \eqref{speedsnthG}. From the explicit expression of the  breather solution \eqref{5BreG}  but now written in terms of the above derivatives \eqref{notacionF}--\eqref{notacionG}, we obtain that:

\begin{equation}\label{Btexp}
B_\mu   = 2\frac{g_1f-f_1g}{D} 
\qquad\text{and}\qquad \tilde B_{\mu,t}  = 2\frac{g_2f-f_2g}{D}.
\end{equation}
Moreover we get
\begin{equation}\label{B3}
B_\mu^2  = 4\left(\frac{g_1f-f_1g}{D}\right)^2\qquad\text{and}\qquad B_\mu^3  = 8\left(\frac{g_1f-f_1g}{D}\right)^3. 
\end{equation}
%
%
%
%
%
%
Now, we compute $B_{\mu,xx}$. First we get
\[B_{\mu,x} = -\frac{2}{D^2}\left(f^3 g_3-f^2 (2 f_1 g_1+f_3 g)+f g \left(2 f_1^2+g g_3-2 g_1^2\right)+g^2 (2 f_1 g_1-f_3 g)\right),\]
and then
\begin{equation}\label{Bxx}
B_{\mu,xx} = 2\frac{M_1}{D^3},
\end{equation}
where 
\begin{equation}\label{N1}
\begin{aligned}
M_1:=&\Big(f^5 g_4-f^4 (3 f_1 g_3+3 f_3 g_1+f_4 g)+2 f^3 \left(3 f_1^2 g_1+3 f_1 f_3 g+g^2 g_4-3 g g_1 g_3-g_1^3\right)\\
&-2 f^2 g \left(3 f_1^3-9 f_1 g_1^2+f_4 g^2\right)+f g^2 \left(-18 f_1^2 g_1+6 f_1 f_3 g+g^2 g_4-6 g g_1 g_3+6 g_1^3\right)\\
&+g^3 \left(2 f_1^3+f_1 \left(3 g g_3-6 g_1^2\right)+g (3 f_3 g_1-f_4 g)\right)\Big),
\end{aligned}
\end{equation}
and therefore from  \eqref{Btexp}, \eqref{B3}, \eqref{Bxx} and \eqref{N1}, we get
\begin{equation}\label{rhs}
A_1B_\mu + A_2(B_{\mu,xx} + 2B_\mu^3 + 6\mu B_\mu^2)  = \frac{M_2}{D^3},
\end{equation}
where 
\begin{equation}\label{M2}
M_2:=2\Big(A_1D^2(fg_1-f_1g) +  A_2(8(fg_1-f_1g)^3 + 12\mu D(f_1g-fg_1)^2 + M_1)\Big),
\end{equation}
Now, we verify by using the symbolic software \emph{Mathematica} that, after expanding $f's$ and $g's$ terms \eqref{notacionF}--\eqref{notacionG} and lengthy rearrangements, the above term \eqref{M2}  simplifies as follows:
\[M_2 = 2D^2(g_2f-gf_2).\]
Finally, remembering \eqref{rhs}, we have that
\[A_1B_\mu + A_2(B_{\mu,xx} + 2B_\mu^3 + 6\mu B_\mu^2)  = \frac{M_2}{D^3} = \frac{2D^2(g_2f-gf_2)}{D^3} = \tilde{B}_{\mu,t},\]
and we conclude.


\begin{thebibliography}{99}
\small{
%

\bibitem{AC} M. Ablowitz and P. Clarkson, \emph{Solitons, nonlinear evolution equations and inverse scattering}, London Mathematical Society Lecture Note Series, 149. Cambridge University Press, Cambridge, 1991.





\bibitem{Ale1} M.A. Alejo, \emph{On the ill-posedness of the Gardner equation},  J. Math. Anal. Appl.,{\bf396} no. 1,  256-260  (2012).

\bibitem{Ale11} M.A. Alejo, \emph{Well-posedness and Stability results for solitons of the Gardner equation}, NoDEA/Nonlinear Differential Equations and Applications, Volume 19, Number 4 (2012), 503--520.


\bibitem{AC2018} M.A. Alejo and E. Cardoso, \emph{On the ill-posedness of the 5th-order Gardner equation}, preprint, arXiv:1810.10434 [math.AP]

\bibitem{AM} M.A. Alejo  and  C. Mu\~noz, \emph{Nonlinear stability of mKdV breathers}, Comm. Math. Phys., \textbf{37} (2013), 2050--2080.
%

\bibitem{AM1} M.A. Alejo, \emph{Nonlinear stability of Gardner breathers}, Jour. Diff. Equat. \textbf{264}, n.2, 1192-1230 (2018).



%



%
%




\bibitem{Bourgain1993} J. Bourgain, \emph{Fourier transform restriction phenomena for certain lattice subsets and applications to nonlinear evolution equations. Parts I, II}, Geom. Funct. Anal. 3 (1993) 107--156, 209--262.

\bibitem{Bourgain1995} J. Bourgain, \emph{On the Cauchy problem for periodic KdV-type equations}, Proceedings of the Conference in Honor of Jean-Pierre Kahane (Orsay, 1993). J. Fourier Anal. Appl. 1995, Special Issue, 17--86.



\bibitem{BGT2002} N. Burq, P. G\'erard and N. Tzvetkov, \emph{An instability property of the nonlinear Schr\"odinger equation on $S^d$}, Math. Res. Lett. 9 (2002), no. 2-3, 323--335.

\bibitem{BGT2003} N. Burq, P. G\'erard and N. Tzvetkov, \emph{Two singular dynamics of the nonlinear Schr\"odinger equation on a plane domain}, Geom. Funct. Anal. 13 (2003), no. 1, 1--19.

\bibitem{CK2018-1} M. Cavalcante and C. Kwak, \emph{The initial-boundary value problem for the Kawahara equation on the half-line}, preprint, arXiv:1805.05229 [math.AP].

\bibitem{CK2018-2} M. Cavalcante and C. Kwak, {\it Local well-posedness of the fifth-order KdV-type equations on the half-line}, accepted for the publication in CPAA, http://arxiv.org/abs/1808.06494.

\bibitem{CG2011} W. Chen and Z. Guo, \emph{Global well-posedness and I-method for the fifth-order Korteweg-de Vries equation}, J. Anal. Math. 114 (2011) 121--156.


\bibitem{CLMW2009} W. Chen, J. Li, C. Miao and J. Wu, \emph{Low regularity solutions of two fifth-order KdV type equations}, J. Anal. Math. 107 (2009) 221--238.

\bibitem{CCT2003} M. Christ, J. Colliander, T. Tao, \emph{Asymptotics, frequency modulation, and low regularity ill-posedness for canonical defocusing equations}, Amer. J. Math. 125 (2003), no. 6, 1235--1293.

\bibitem{CCT2008} M. Christ, J. Colliander, T. Tao, \emph{A priori bounds and weak solutions for the nonlinear Schr\"odinger equation in Sobolev space of negative order}, Journal of Functional Analysis 254 (2008), 368--395.


\bibitem{CT2005} S. Cui and S. Tao, \emph{Strichartz estimates for dispersive equations and solvability of the Kawahara equation}, J. Math. Anal. Appl. 304 (2005) 683--702.











\bibitem{Ga} C.S. Gardner, M.D. Kruskal and R. Miura, \emph{Korteweg-de Vries equation and generalizations. II. Existence of conservation laws and constants of motion}, J. Math. Phys. \textbf{9}, no. 8, 1204--1209  (1968).


\bibitem{Gomes}  J.F. Gomes, G.S. Fran\c ca and A.H. Zimerman, \emph{Nonvanishing boundary condition for the mKdV
hierarchy and the Gardner equation} 2012 J. Phys. A: Math. Theor. 45 015207.

\bibitem{Guo2011} Z. Guo, \emph{Local well-posedness and a priori bounds for the modified Benjamin-Ono equation}, Advances in Differential Equations, 16/11-12 (2011), 1087--1137.

\bibitem{Guo2012} Z. Guo, \emph{Local well-posedness for dispersion generalized Benjamin-Ono equations in Sobolev spaces}, J. Differential Equations 252 (2012) 2053--2084.

\bibitem{GKK2013} Z.Guo, C. Kwak, S. Kwon, \emph{Rough solutions of the fifth-order KdV equations}, J. Funct. Anal. 265 (2013) 2791--2829.

\bibitem{GO2018} Z. Guo, T. Oh, \emph{Non-existence of solutions for the periodic cubic NLS below $L^2$},  Int. Math. Res. Not. IMRN (2018), no. 6, 1656–1729.




\bibitem{GrSl} R. Grimshaw, A. Slunyaev and E. Pelinovsky,  \textit{Generation of solitons and breathers in the extended Korteweg-de Vries equation with positive cubic nonlinearity.} Chaos 20 (2010), n.1, 01310201--01310210.


\bibitem{Gru} A. Gr\"unrock, \emph{On the hierarchies of higher order mKdV and KdV equations}, Cent. Eur. J. Math. Vol. 8(3), 500-536, (2010).

\bibitem{GPWW2011} Z. Guo, L. Peng, B. Wang and Y. Wang, \emph{Uniform well-posedness and inviscid limit for the Benjamin-Ono-Burgers equation}, Advances in Mathematics 228 (2011) 647--677.






\bibitem{IK2007} A. Ionescu, C. Kenig, \emph{Global well-posedness of the Benjamin-Ono equation in low-regularity spaces}, J. Amer. Math. Soc. 20 (3) (2007) 753--798.

\bibitem{IKT2008} A. Ionescu, C. Kenig, D. Tataru, \emph{Global well-posedness of the KP-I initial-value problem in the energy space}, Invent. Math. 173 (2) (2008) 265--304.



\bibitem{Kato2012} T. Kato, \emph{Well-posedness for the fifth order KdV equation}, Funkcialaj Ekvacioj 55 (1) (2012) 17--53.

\bibitem{KP2015} C. Kenig and D. Pilod, \emph{Well-posedness for the fifth-order KdV equation in the energy space}, Trans. Amer. Math. Soc. 367 (2015) 2551-2612.

\bibitem{KP2016} C. Kenig and D. Pilod, \emph{Local well-posedness for the KdV hierarchy at high regularity}, Adv. Diff. Eq., 21 (2016), 801--836.

\bibitem{KPV1991-0} C. E. Kenig, G. Ponce and L. Vega, \emph{On the hierarchy of the generalized KdV equations}, Singular limits of dispersive waves (Lyon, 1991), NATO Adv. Sci. Inst. Ser. B Phys., vol. 320, Plenum, New York, 1994, pp. 347–356. 

\bibitem{KPV1991.0} C. E. Kenig, G. Ponce and L. Vega, \emph{Well-posedness of the initial value problem for the Korteweg-de Vries equation}, J. Amer. Math. Soc. 4 (1991), no. 2, 323--347.

\bibitem{KPV1991} C. E. Kenig, G. Ponce and L. Vega, \emph{Oscillatory integrals and regularity of dispersive equations}, Indiana U. Math. J 40 (1991) 33--69.


\bibitem{KPV1994} C. E. Kenig, G. Ponce and L. Vega, \emph{Higher-order nonlinear dispersive equations}, Proc. Amer. Math. Soc. 122 (1994), no. 1, 157–166, DOI 10.2307/2160855. 

\bibitem{KPV1996} C. Kenig, G. Ponce, L. Vega, \emph{A bilinear estimate with applications to the KdV equation}, J. Amer. Math. Soc. 9 (1996) 573--603.

\bibitem{KPV2} C.E. Kenig, G. Ponce and L. Vega, \emph{On the ill-posedness of some canonical dispersive equations}, Duke Math. J. \textbf{106}, no. 3, 617--633  (2001).


%




\bibitem{KT2007} H. Koch, D. Tataru, \emph{A priori bounds for the 1D cubic NLS in negative Sobolev spaces}, Int. Math. Res. Not. IMRN 16 (2007), Art. ID rnm053, 36, DOI 10.1093/imrn/rnm053. MR2353092 (2010d:35307)

\bibitem{KT2005} H. Koch, N. Tzvetkov, \emph{Nonlinear wave interactions for the Benjamin-Ono equation}, Int. Math. Res. Not. 30 (2005) 1833--1847.

\bibitem{KT2008} H. Koch, N. Tzvetkov, \emph{On finite energy solutions of the KP-I equation}, Math. Z. 258 (2008), no. 1, 55--68.

\bibitem{Kwak2018} C. Kwak \emph{Low regularity Cauchy problem for the fifth-order modified KdV equations on $\T$}, Journal of Hyperbolic Differential Equations Vol. 15, No. 3 (2018) 463--557. http://dx.doi.org/10.1142/S0219891618500170.

\bibitem{Kwak2016} C. Kwak \emph{Local well-posedness for the fifth-order KdV equations on $\T$}, J. Differential Equations 260 (2016) 7683--7737. http://dx.doi.org/10.1016/j.jde.2016.02.001.

\bibitem{Kwon2008} S. Kwon, \emph{On the fifth order KdV equation: Local well-posedness and lack of uniform continuity of the solution map} J. Differential Equations, 245 (2008) 2627--2659.

\bibitem{Kwon} S. Kwon, \emph{Well posedness and Ill-posedness of the Fifth-order modified KdV equation},  Electr. Journal Diff. Equations. vol. 2008, n.1, 1--15 (2008).

\bibitem{La} G.L. Lamb, \emph{Elements of Soliton Theory}, Pure Appl. Math., Wiley, New York, 1980.


\bibitem{Lin} F. Linares, \emph{A higher order modified Korteweg-de Vries equation}, Comp. Appl. Math. \textbf{14}, n.3,  253-267,   (1995).









\bibitem{Mat1} Y. Matsuno, \emph{Bilinearization of Nonlinear Evolution Equations: Higher Order mKdV}, Jour. Phys. Soc.Japan, {\bf 49}, n.2 (1980).








\bibitem{MPV2018-1} L. Molinet, D. Pilod and S. Vento, \emph{Unconditional uniqueness for the modified Korteweg- de Vries equation on the line}, to appear in Rev. Mat. Iber. (2018).

\bibitem{MPV2018-2} L. Molinet, D. Pilod and S. Vento, \emph{On unconditional well-posedness for the periodic modified Korteweg-De Vries equation}, to appear in J. Math. Soc. Japan (2018)

\bibitem{MST2001} L. Molinet, J.C. Saut and N. Tzvetkov, \emph{Ill-posedness issues for the Benjamin–Ono and related equations}, SIAM J.Math. Anal. 33 (2001) 982--988.

\bibitem{MST2002} L. Molinet, J.C. Saut and N. Tzvetkov, \emph{Well-posed and ill-posedness results for the Kadomtsev–Petviashvili-I equation}, Duke Math. J. 115 (2) (2002) 353--384.

\bibitem{Pilod2008} D. Pilod, \emph{On the Cauchy problem for higher-order nonlinear dispersive equations}, J. Differential Equations 245 (2008), no. 8, 2055–2077, DOI 10.1016/j.jde.2008.07.017.
%




\bibitem{Ponce1993} G. Ponce, \emph{Lax pairs and higher order models for water waves}, J. Differential Equations 102 (2) (1993) 360--381.



\bibitem{Tao2001} T. Tao, \emph{Multilinear weighted convolution of $L^2$ functions and applications to nonlinear dispersive equations}, Amer. J. Math. 123 (5) (2001) 839--908.

\bibitem{Tsugawa2017} K. Tsugawa, \emph{Parabolic smoothing effect and local well-posedness of fifth-order semilinear dispersive equations on torus}, Harmonic analysis and nonlinear partial differential equations, 177–193, RIMS Kôkyûroku Bessatsu, B60, Res. Inst. Math. Sci. (RIMS), Kyoto, 2016, arXiv:1707.09550 [math.AP].






%
}
\end{thebibliography}
\end{document}